\documentclass[10pt, reqno]{amsart}
\usepackage{amsmath, amssymb, amsfonts, amsthm, graphics, mathrsfs}
\usepackage{amscd}
\usepackage{setspace}
\usepackage{mathtools, bbm}
\usepackage{mathdots, color}
\usepackage[all]{xy}
\usepackage[dvipdfmx, colorlinks=true, pagebackref]{hyperref}

\usepackage[lite,abbrev,alphabetic]{amsrefs}
\usepackage[foot]{amsaddr}
\usepackage{lineno}

\newtheorem{thm}{{{Theorem}}}[section]
\newtheorem{prop}[thm]{{Proposition}}
\newtheorem{lem}[thm]{{Lemma}}
\newtheorem{cor}[thm]{{Corollary}}
\newtheorem{rem}[thm]{{Remark}}

\newtheorem{cond}[thm]{{Condition}}

\numberwithin{equation}{section}

\makeindex

\def\M{\mathrm{M}}

\def\bs{\backslash}

\def\al{\alpha}
\def\be{\beta}
\def\ve{\varepsilon}
\def\l{\langle}
\def\r{\rangle}
\def\1{\mathbf{1}}
\def\0{\mathbf{0}}

\def\ol{\overline}

\def\N{\mathbb{N}}
\def\C{\mathbb{C}}
\def\R{\mathbb{R}}
\def\Q{\mathbb{Q}}
\def\Z{\mathbb{Z}}
\def\A{\mathbb{A}}
\def\F{\mathbb{F}}

\def\GL{{\mathop{\mathrm{GL}}}}
\def\PGL{{\mathop{\mathrm{PGL}}}}
\def\SL{{\mathop{\mathrm{SL}}}}
\def\O{{\mathop{\textnormal{O}}}}
\def\SO{{\mathop{\mathrm{SO}}}}
\def\U{{\mathop{\textnormal{U}}}}

\DeclareMathOperator{\Hom}{Hom}
\DeclareMathOperator{\vol}{vol}

\DeclareMathOperator{\Ind}{Ind}
\DeclareMathOperator{\Ad}{Ad}
\DeclareMathOperator{\JL}{JL}

\def\Re{{\mathop{\mathrm{Re}}}}

\def\Tr{{\mathop{\mathrm{tr}}}}

\def\Hom{{\mathop{\mathrm{Hom}}}}
\def\diag{{\mathop{\mathrm{diag}}}}

\def\Ad{{\mathop{\mathrm{Ad}}}} 
\def\vol{{\mathop{\mathrm{vol}}}}

\def\d{{\mathrm{d}}}

\def\bsl{\backslash}
\def\inf{\infty}
\def\fin{\mathrm{fin}}

\def\fo{\mathfrak{o}}

\def\fO{\mathfrak{O}}

\def\cO{\mathcal{O}}
\def\cS{\mathcal S}

\def\trep{{\mathbbm{1}}}

\def\fa{{\mathfrak{a}}}

\def\cP{{\mathcal{P}}}
\def\cE{{\mathcal{E}}}
\def\cR{{\mathcal{R}}}
\def\GJ{{\mathrm{GJ}}}

\def\bK{{\mathbf{K}}}
\def\cK{{\mathcal{K}}}
\def\sK{{\mathscr{K}}}

\def\cF{{\mathcal{F}}}

\def\sF{{\mathscr{F}}}

\def\sG{{\mathscr{G}}}

\def\ff{{\mathfrak{f}}}
\def\tZ{{\tilde{Z}}}

\def\qG{{\mathcal{M}}}
\def\qK{{\mathcal{U}}}

\title[Zeta functions and nonvanishing theorems for toric periods on $\GL_2$]{Zeta functions and nonvanishing theorems for toric periods on $\GL_2$}
\date{\today}

\author[Miyu Suzuki]{Miyu Suzuki} 
\address{
Depertment of Mathematics,  Faculty of Science\\
Kyoto University \\
Kitashirakawa Oiwake-cho, Sakyo-ku,  Kyoto 606-8502,   JAPAN}
\email{suzuki.miyu.4c@kyoto-u.ac.jp}

\author[Satoshi Wakatsuki]{Satoshi Wakatsuki}
\email{wakatsuk@staff.kanazawa-u.ac.jp}
\address{
Faculty of Mathematics and Physics, Institute of Science and Engineering\\
Kanazawa University \\
Kakumamachi, Kanazawa, Ishikawa, 920-1192, JAPAN}

\begin{document}

\begin{abstract}
Let $F$ be a number field and $D$ a quaternion algebra over $F$.
Take a cuspidal automorphic representation $\pi$ of $D_\A^\times$ with trivial central cahracter.
We study the zeta functions with period integrals on $\pi$ for the perhomogeneous vector space $(D^\times\times D^\times\times\GL_2, D\oplus D)$.
We show their meromorphic continuation and functional equation, determine the location and orders of possible poles and compute the residue.
Arguing along the theory of Saito and computing unramified local factors, the explicit formula of the zeta functions is obtained.
Counting the order of possible poles of this explicit formula, we show that if $L(1/2, \pi)\neq0$, there are infinitely many quadratic extension $E$ of $F$ which embeds in $D$, such that $\pi$ has a nonvanishing toric period with respect to $E$.
\end{abstract}

\maketitle

\tableofcontents

\section{Introduction and main results}\label{sec:1}

Let $F$ be a number field, $\A$ the ad\'ele ring of $F$, and $D$ a quaternion algebra over $F$.
In this paper, we study global zeta functions with automorphic forms on $D_\A^\times\coloneqq (D\otimes\A)^\times$ for the perhomogeneous vector space $(D^\times\times D^\times\times\GL_2, D\oplus D)$. 
Our purpose is to prove the analytic properties of the global zeta functions, and present an explicit formula for them, which clarifies the relationship between our global zeta functions and toric periods. 
As one application of them, we derive a nonvanishing theorem for toric periods by using the analytic properties. 
In addition, this study is used to derive meanvalue theorems for toric periods and central values of automorphic $L$-functions in the subsequent paper \cite{SW}.

\subsection{Global zeta functions}

Set $G\coloneqq D^\times\times D^\times\times\GL_2$ and $V\coloneqq D\oplus D$.
The $F$-rational representation $\rho$ of $G$ on $V$ is defined by
    \[
    (x, y)\cdot\rho(g_1, g_2, g_3)\coloneqq (g_1^{-1}xg_2, g_1^{-1}yg_2)g_3, 
    \quad (g_1,g_2,g_3)\in G, \quad (x,y)\in V.
\]
Then, the triple $(G, \rho, V)$ is an $F$-form of a prehomogeneous vector space.
We write the regular locus of $V$ by $V^0$ (see \S\,\ref{sec:space} for details). 
Throughout this paper, we assume that $\pi$ is an irreducible cuapidal automorphic representation of $PD^\times_\A\coloneqq \A^\times\bsl D_\A^\times$ which is not $1$-dimensional. 
For a cusp form $\phi$ in $\pi$ and a Schwartz function $\Phi$ on $V(\A)$,
we define the global zeta function $Z(\Phi, \phi, s)$ with an automorphic form $\phi$ by
    \[
    Z(\Phi, \phi, s)\coloneqq \int_{H(F)\bs H(\A)}|\omega(h)|^s\phi(g_1) \overline{\phi(g_2)}\sum_{x\in V^0(F)}\Phi(x\cdot\rho(h)) \, \d h,
    \]
where $H$ is the quotient of $G$ by the kernel of $\rho\,\colon G\rightarrow\GL(V)$, $\d h$ is the Tamagawa measure on 
$H(\A)$, $\omega$ is the character attached to the fundamental relative invariant of $(G, \rho, V)$, and $(g_1, g_2, g_3)$ is a representative in $G$ of $h\in H$. 
The function $Z(\Phi, \phi, s)$ absolutely converges for large $\Re(s)$ (see Lemma \ref{lem:absconv}), is meromorphically continued to the whole $s$-plane (see Theorem \ref{thm:global}), and satisfies a functional equation (see Corollary \ref{cor:funct}).

\subsection{Toric periods and local zeta functions}\label{sec:1.2}

To explain our explicit formula, let us briefly discuss toric periods and Waldspurger's formula.
Let $\Sigma$ denote the set of places of $F$, $F_v$ the completion of $F$ at $v\, (\in\Sigma)$, and $D_v=D\otimes_F F_v$. 
Write $X(D)$ (resp. $X(D_v)$) for a set of representatives of isomorphism classes of quadratic \'etale $F$-subalgebras of $D$ (resp. $F_v$-subalgebra of $D_v$).
For each $E\in X(D)$ (resp.\,$\cE_v\in X(D_v)$), choose an element $\delta_E\in D$ (resp.\,$\delta_{\cE_v}\in D_v$) so that $E=F+F\delta_E$ and $\Tr(\delta_E)=0$ (resp.\,$\cE_v=F_v+F_v\delta_{\cE_v}$ and $\Tr(\delta_{\cE_v})=0$). 
Set $d_E\coloneqq \delta_E^2$ (resp.\,$d_{\cE_v}\coloneqq \delta_{\cE_v}^2$).

Take an element $E\in X(D)$. 
Unless otherwise mentioned, we assume $\delta_{E_v}=\delta_E$ and $d_{E_v}=d_E$ under the natural embedding $E\hookrightarrow E_v$ and $D\hookrightarrow D_v$.
Let $\eta_E=\otimes_{v\in\Sigma} \eta_{E_v}$ denote the quadratic character of $\A_F^\times/F^\times$ associated with $E$ via the global class field theory if $E$ is a field and the trivial character otherwise.
Set $\A_E\coloneqq\A_F\otimes_FE$.

Let $\pi=\otimes_{v\in\Sigma}\pi_v$ and $\langle\cdot,\cdot\rangle$ denote the Petersson inner product on $PD^\times\bsl PD_\A^\times$ (see \eqref{eq:innprod} for its definition). 
For each place $v\in\Sigma$, take a $D_v^\times$-invariant non-degenerate Hermitian pairing $\langle\cdot,\cdot\rangle_v$ on $\pi_v$ so that we have $\langle\cdot,\cdot\rangle=\prod_{v\in\Sigma}\langle\cdot,\cdot\rangle_v$. 
For $\cE_v\in X(D_v)$,  let $\d h_v=\frac{\d h_{\cE_v}}{\d^\times z_v}$ be the quotient measure on $F_v^\times\bs\cE_v^\times$ (see \eqref{eq:localmeah} for $\d h_{\cE_v}$ and \S\ref{sec:notation} for $\d^\times z_v$). 
The integral
    \[
    \alpha_{\cE_v}(\varphi_v,\varphi'_v)\coloneqq \int_{F_v^\times\bs \cE_v^\times}
    \langle\pi_v(h_v)\varphi_v, \varphi'_v\rangle_v\,\d h_v, 
    \hspace{10pt} \varphi_v, \varphi'_v\in\pi_v
    \]
converges absolutely and defines an element of $\Hom_{\overline{\cE_v^\times}\times \overline{\cE_v^\times}}(\pi_v\boxtimes\bar{\pi}_v, \C)$, 
Here we write $\overline{A}$ for the image of $A(\subset D_v)$ under the projection $D_v^\times\to PD_v^\times$.
Let $E\in X(D)$.
We define the Haar measure $\d h$ on $\A_F^\times E^\times\bs \A_E^\times$ by $\d h= (c_F^{-1} \d^\times z)\bs \d h_E$ (see \eqref{eq:dh_E} for $\d h_E$, \S\ref{sec:measure1} for $\d^\times z$, and $c_F$ is the residue of the finite part of $\zeta_F(s)$ at $s=1$).
A linear form $\cP_E$ on $\pi$ is defined by 
    \[
    \cP_E(\varphi)\coloneqq \int_{\A_F^\times E^\times\bs \A_E^\times}\varphi(h) \, \d h , 
    \hspace{10pt} \varphi \in \pi.
    \]
We say that $\pi$ is $E^\times$-distinguished if $\cP_E$ is not identically zero.
The linear form $\cP_E$ is called a toric period.
Similarly for $\cE_v\in X(D_v)$, we say that $\pi_v$ is $\cE_v^\times$-distinguished if $\Hom_{\overline{\cE_v^\times}}(\pi_v, \C)\neq0$.

Let $\zeta_F(s)$ be the Dedekind zeta function of $F$, $L(s, \eta_E)$ the Hecke $L$-function, $L(s,\pi)$ the principal $L$-function, and $L(s, \pi, \Ad)$ the adjoint $L$-function of $\pi$. 
For any finite set $S$ of places of $F$, we denote the partial Euler product of $\zeta_F(s)$, $L(s, \eta_E)$, $L(s,\pi)$, and $L(s, \pi, \Ad)$ by $\zeta_F^S(s)$, $L^S(s, \eta_E)$, $L^S(s,\pi)$, and $L^S(s, \pi, \Ad)$, respectively. 
We take $\phi_j=\otimes_{v\in\Sigma}\phi_{j,v}\in\pi$ $(j=1$, $2)$. 
In \cite{Wal2}, Waldspurger proved a formula which relates the toric period and the special values of automorphic $L$-functions: 
    \begin{equation}\label{eq:Wald}
    \cP_E(\phi_1)\, \overline{\cP_E(\phi_2)} = \frac{1}{|\Delta_F|} \,
    \frac{\zeta_F(2) \, L(\frac12, \pi) \, L(\frac12, \pi\otimes\eta_E)}
    {L(1, \pi, \Ad) \, L(1, \eta_E)^2} \,
    \prod_v \alpha_{E_v}^\#(\phi_{1,v}, \phi_{2,v}) .
    \end{equation}
Here, $|\Delta_F|$ denotes the absolute discriminant of $ F/\Q$ and
\begin{equation}\label{eq:normalized}
    \alpha_{E_v}^\#(\phi_v, \phi_v)=
    \frac{L(1, \pi_v, \Ad)L(1, \eta_{E_v})}
    {\zeta_{F_v}(2)L(\frac12, \pi_v)L(\frac12, \pi_v\otimes\eta_{E_v})} \,
    \alpha_{E_v}(\phi_v, \phi_v)    
\end{equation}  
is the normalized local period.
In particular, when $\phi_1=\phi_2=\phi=\otimes_{v\in\Sigma}\phi_v\in\pi$, we have
    \begin{equation}\label{eq:Wald3}
    |\cP_E(\phi)|^2=\frac{1}{|\Delta_F|} \,
    \frac{\zeta_F(2)L(\frac12, \pi)L(\frac12, \pi\otimes\eta_E)}
    {L(1, \pi, \Ad)L(1, \eta_E)^2} \,
    \prod_v\alpha_{E_v}^\#(\phi_v, \phi_v).
    \end{equation}
Note that the choice of our Haar measure on $PD_\A^\times$ is $2c_F$ times the Tamagawa measure.

We also need a local zeta function $Z_{\cE_v}(\Phi_v, \phi_v, s)$ with the local integral $\alpha_{\cE_v}$. 
Let $\d x_v$ denote the local Tamagawa measure on $V(F_v)$ (cf. \S\ref{sec:measure1}). 
We take a point $x_{\cE_v}\in V^0(F_v)$ for each $\cE_v\in X(D_v)$ (see Proposition \ref{prop:orbits}) and put $V_{\cE_v}(F_v)\coloneqq x_{\cE_v}\cdot \rho(G(F_v))$. 
Set
    \[
    Z_{\cE_v}(\Phi_v, \phi_v, s)
    \coloneqq\frac{2\,c_v}{L(1, \eta_{\cE_v})^2} \int_{V_{\cE_v}(F_v)}
    \alpha_{\cE_v}(\pi_v(g_1)\phi_v, \pi_v(g_2)\phi_v)\, |P(x_v)|_v^{s-2} \, \Phi_v(x_v) \d x_v,
    \]
where $\cE_v\in X(D_v)$, $x_v=x_{\cE_v}\cdot \rho(g_1,g_2,g_3)$, $c_v\coloneqq\zeta_{F_v}(1)$ if $v<\inf$, and $c_v\coloneqq 1$ if $v\mid \inf$.
It is known that $Z_{\cE_v}(\Phi_v, \phi_v, s)$ is absolutely convergent for large $\Re(s)$ and meromorphically continued to the whole $s$-plane (cf. \cite{Li4}).  
Note that $\pi_v$ is $\cE_v^\times$-distinguished if and only if there exists a test function $\Phi_v$ such that $Z_{\cE_v}(\Phi_v,\phi_v,s)\not\equiv 0$ (see Lemma \ref{lem:loc}).

\subsection{Explicit formula}

We have fixed a cuspidal automorphic representation $\pi=\otimes_{v\in\Sigma} \pi_v$ of $PD_\A^\times$ which is not $1$-dimensional. 
We now fix a decomposable vector $\phi=\otimes_{v\in\Sigma}\phi_v \in\pi$.  
We need the following conditions.
\begin{cond}\label{condition}
Let $S$ be a finite subset of $\Sigma$. 
For every $v\not\in S$, we suppose the following conditions.
\begin{itemize}
\item $v$ is a finite place which is not dyadic.
\item $D_v$ is split, in particular $PD_v^\times\simeq\PGL_2(F_v)$.
\item Let $K_v$ be the maximal compact subgroup corresponding to $\PGL_2(\fo_v)$ under a fixed isomorphism $PD_v^\times\simeq\PGL_2(F_v)$, where $\fo_v=\fo_{F,v}$ is the integer ring of $F_v$. 
\item $\pi_v$ is unramified and $\phi_v$ is the $K_v$-spherical vector, which is normalized so that $\langle \phi_v,\phi_v \rangle_v=1$. 
\item For every $E\in X(D)$, we have $d_{E_v}\in\fo_v\setminus\varpi_v^2\fo_v$, where $\varpi_v$ is a prime element of $F_v$. In addition, the maximal compact subgroup of $\overline{E_v^\times}$ is contained in $K_v$ (that is, $\overline{\fo_{E,  v}^\times}\subset K_v$).
\end{itemize}
\end{cond}
Since we have $PD_\A^\times=PD^\times \prod_{v\in S}PD_v^\times \prod_{v\notin S}K_v$ for sufficiently large $S$, by taking a suitable $PD^\times$-conjugate, we may assume that our fixed embedding $E\hookrightarrow D$ satisfies the last condition of Condition \ref{condition}. 
Under Condition \ref{condition}, we obtain $\alpha_{E_v}^\#(\phi_v,  \phi_v)= 1$ for any $v\notin S$ (cf. \cite{SW}*{Corollaries 29 and 30}). 
Note that this fact follows from a purely local argument. 
Therefore, the formula \eqref{eq:Wald3} is equivalent to
    \begin{equation}\label{eq:Wald2}
    |\cP_E(\phi)|^2=\frac{1}{|\Delta_F|} \,
    \frac{\zeta_F^S(2) \, L^S(\frac12, \pi) \, L^S(\frac12, \pi\otimes\eta_E)}
    {L^S(1, \pi, \Ad)\, L^S(1, \eta_E)\, L(1,\eta_E)} \,
    \alpha_{E,  S}(\phi),
    \end{equation}
where $\alpha_{E, S}(\phi)\coloneqq \prod_{v\in S}\alpha_{E_v}(\phi_v,\phi_v)$. 
Note that if both sides of \eqref{eq:Wald2} is not equal to zero,  then $\alpha_{E,  S}(\phi)^{-1}|\cP_E(\phi)|^2$ depends only on the isomorphism class of $E$ and is independent of its realization as a subalgebra of $D$.

Suppose that $S$ satisfies Condition \ref{condition} and set $X(D_S)=\prod_{v\in S}X(D_v)$.
Take an element $\cE_S=(\cE_v)_{v\in S}\in X(D_S)$.  
Set 
    \[
    X(D, \cE_S)=\{E\in X(D) \mid E_v\cong\cE_v \text{ for each } v\in S\}
    \]
and take $E\in X(D, \cE_S)$. 
In the case $\alpha_{E, S}(\phi)\neq0$, $\alpha_{E, S}(\phi)^{-1}|\cP_E(\phi)|^2$ is explicitly given by \eqref{eq:Wald2}.  
In the case $\alpha_{E, S}(\phi)=0$, by abuse of notation, from the viewpoint of \eqref{eq:Wald2} we formally define 
    \[
     \alpha_{E, S}(\phi)^{-1}|\cP_E(\phi)|^2\coloneqq \frac{1}{|\Delta_F|} \,
    \frac{\zeta_F^S(2) \, L^S(\frac12, \pi) \, L^S(\frac12, \pi\otimes\eta_E)}
    {L^S(1, \pi, \Ad)\, L^S(1, \eta_E)\, L(1,\eta_E)}  
    \] 
    when $\pi$ is $E^\times$-distinguished, and
    \[
     \alpha_{E, S}(\phi)^{-1}|\cP_E(\phi)|^2\coloneqq 0 \qquad \text{when $\pi$ is not $E^\times$-distinguished.}
    \]
Note that $\alpha_{E, S}(\phi)=0$ implies $\cP_E(\phi)=0$ even if $\pi$ is $E^\times$-distinguished.
In addition, we define
    \begin{equation*}
         d(\cE_S/E) \coloneqq  \prod_{v\in S}\left|\frac{d_{\cE_v}}{d_{E_v}}\right|_v^{\frac12}.
    \end{equation*}
Let $\ff_{E_v}$ be the conductor of the quadratic character $\eta_{E_v}$, $N(\ff_{E_v})$ its norm, and $N(\ff_E^S)=\prod_{v\notin S}N(\ff_{E_v})$. 
For $v\not\in S$, set $\lambda_v=q_v^{\frac12}(\alpha_v+\alpha_v^{-1})$ where $\alpha_v\in\C^\times$ is the Satake parameter of $\pi_v$ and $q_v$ is the order of the residue field of $F_v$.
The following is our explicit formula.
\begin{thm}\label{thm:explicit}
Assume that $S$ satisfies Condition \ref{condition}. 
Take a Schwartz function $\Phi=\otimes_{v\in\Sigma} \Phi_v$ on $V(\A)$ such that $\Phi_v$ is the characteristic function of $V(\fo_v)$ for every $v\notin S$. 
For sufficiently large $\Re(s)>0$, we obtain
    \[
    Z(\Phi,\phi,s)= \sum_{\cE_S=(\cE_v)_{v\in S} \in X(D_S)}   \prod_{v\in S} Z_{\cE_v}(\Phi_v,\phi_v,s)  \times  \xi(D, \cE_S, \phi, s),
    \]
    where
\begin{align*}
\xi(D, \cE_S, \phi, s)\coloneqq & \frac{ 1 }{2 \, |\Delta_F|^4  \, c_F }  \frac{\zeta^S_F(2s-1) \, 
    L^S(2s-1,\pi,\mathrm{Ad}) }{\zeta^S_F(2)^3} \\
    &\times \sum_{E\in X(D,\cE_S)} 
    \frac{L(1,\eta)^2 \, \alpha_{E,S}(\phi)^{-1}  |\cP_E(\phi)|^2\, d(\cE_S/E)\, \mathcal{D}_E^S(\pi,s)}
    {N(\mathfrak{f}_E^S)^{s-1}},
\end{align*}
$\mathcal{D}_E^S(\pi,s)\coloneqq \prod_{v\notin S}\mathcal{D}_{E_v}(\pi_v,s)$, and 
\begin{multline*}
    \mathcal{D}_{E_v}(\pi_v,s)\coloneqq \\
        \begin{cases} 
        1+q_v^{-2s+1}+q_v^{-2s}+q_v^{-4s+1} -2\eta_{E_v}(\varpi_v)\, q_v^{-2s}\lambda_v 
        & \text{if $\eta_{E_v}$ is unramified}, \\ 
        1+q_v^{-2s+1}
        & \text{if $\eta_{E_v}$ is ramified}.
        \end{cases}    
\end{multline*}
\end{thm}
\begin{proof}
This is a rewrite of Theorem \ref{thm:20200105t1}.
Note that the definition of the series $\xi(D, \cE_S, \phi, s)$ there is slightly different from the above one.
Their equality follows from the equation
\begin{multline*}
    L(1,\eta)^2 \, \alpha_{E,S}(\phi)^{-1}  |\cP_E(\phi)|^2\, d(\cE_S/E) = \\
    \frac{\zeta_F^S(2) \, L^S(\frac12, \pi) \, L^S(\frac12, \pi\otimes\eta_E)}
    {|\Delta_F| \, L^S(1, \pi, \Ad) }\times \prod_{v\in S} \frac{ L(1,\eta_{\cE_v})\, |d_{\cE_v}|_v^{\frac12}}{|d_{E_v}|_v^{\frac12}}     
\end{multline*}
which is valid whenever $\pi$ is $E^\times$-distinguished.
\end{proof}

In fact, as can be seen from the series $\xi(D, \cE_S, \phi, s)$ in the explicit formula, the global zeta function $Z(\Phi,\phi,s)$ is essentially a Dirichlet series of toric periods.
The series $\xi(D, \cE_S, \phi, s)$ absolutely converges for large $\Re(s)>0$ and is meromorphically continued to the whole $s$-plane. 
These properties follow from the properties of $Z(\Phi,\phi,s)$ and $Z_{\cE_v}(\Phi_v,\phi_v,s)$ together with the explicit formula of Theorem \ref{thm:explicit}.

\subsection{Nonvanishing theorems}

We will now discuss nonvanishing theorems for toric periods. 
First we would like to explain what our zeta functions lead to.
From the proof of the next theorem, we see that the existence of the poles of the global zeta functions implies nonvanishing of periods.
\begin{thm}\label{thm:nonvanishing423}
If $L(1/2,\pi)\neq 0$, then there exists an element $E\in X(D)$ such that $\pi$ is $E^\times$-distinguished. 
\end{thm}
\begin{proof}
Suppose $L(1/2,\pi)\neq 0$. 
By Proposition \ref{prop:202001241}, there exist a vector $\phi=\otimes_{v\in\Sigma}\phi_v\in\pi$ and a Schwartz function $\Phi$ such that $Z(\Phi,\phi,s)$ has a simple pole at $\frac12$, that is, $Z(\Phi,\phi,s)\not\equiv 0$. 
Take $S$ so that it satisfies Condition \ref{condition}.
Then we have $\xi(D, \cE_S, \phi, s)\not\equiv 0$ for some $\cE_S\in X(D_S)$ by Theorem \ref{thm:explicit}. 
Since $\xi(D, \cE_S, \phi, s)\not\equiv 0$ means that $\pi$ is $E^\times$-distinguished for some $E\in X(D,\cE_S)$, this completes the proof.   
\end{proof}

For $\cE_S\in X(D_S)$, set 
    \[
    X(D,\cE_S,\pi)=\{R\in X(D,\cE_S) \mid \text{$\pi$ is $R^\times$-distinguished} \}. 
    \]
\begin{thm}\label{thm:infinite}
Take an element $\cE_S\in X(D_S)$ and suppose $X(D, \cE_S, \pi)\neq\emptyset$
Then $X(D, \cE_S, \pi)$ is an infinite set.
In other words, there are infinitely many $L\in X(D,\cE_S)$ such that $\pi$ is $L^\times$-distinguished.
\end{thm}
\begin{proof}
Take $E\in X(D, \cE_S, \pi)$.
Assume to the contrary that $0<\# X(D,\cE_S,\pi)<\inf$. 
Then we can choose a large finite set $S'\supset S$ and an element $\cR_{S'}=(\cR_v)_{v\in S'}\in X(D_{S'})$ so that $X(D,\cR_{S'},\pi)=\{ E \}$. 
By Theorem \ref{thm:20200105t1} and Lemma \ref{lem:loc} for $\cR_{S'}$, we can choose a test function $\Phi=\otimes_{v\in\Sigma} \Phi_v$ so that $Z_{\cR_v}(\Phi_v,\phi_v,s)$ is entire, $Z_{\cR_v}(\Phi_v,\phi_v,1)\neq 0$ for all $v\in S'$ and
    \[
    Z(\Phi,\phi,s)=c'\left( \prod_{v\in S'} Z_{\cR_v}(\Phi_v,\phi_v,s) \right)  
    \frac{\zeta^{S'}_F(2s-1) \, L^{S'}(2s-1,\pi,\mathrm{Ad}) \, 
    \mathcal{D}_{E}^{S'}(\pi,s)}{N(\mathfrak{f}_E^{S'})^{s-1}}
    \]
for some constant $c'\neq 0$. 
Hence it follows from Theorem \ref{thm:global} that $\mathcal{D}_{E}^{S'}(\pi,s)$ is meromorphically continued to $\C$. 
We may suppose that $S'$ contains any finite places $v$ dividing small primes without loss of generality. 
In addition, we have 
    \[
    \mathcal{D}_{E_v}(\pi_v,s)-1-q_v^{-2s+1}=q_v^{-2s}+q_v^{-4s+1} -2\eta_v(\varpi_v)\, q_v^{-2s}\lambda_v=O(q_v^{-\frac32+\frac{7}{64}}), \quad q_v\to\inf
    \]
uniformly for $s\geq 1$ by a bound of $\lambda_v$ (see \cite{Kim} and \cite{BB}).
Hence $\mathcal{D}_{E}^{S'}(\pi,s)$ is absolutely convergent for $\Re(s)>1$ and $\Re(\mathcal{D}_{E}^{S'}(\pi,s))$ tends to $+\inf$ as $s\to 1+0$. 
This means that $\mathcal{D}_{E}^{S'}(\pi,s)$ has a pole at $s=1$ and hence $Z(\Phi,\phi,s)$ has a pole at $s=1$. 
This contradicts Theorem \ref{thm:global}.
\end{proof}

We will now discuss a comparison of our nonvanishing theorems with prior studies. 
In a series of papers \cites{Wal1, Wal2, Wal3, Wal4}, Waldspurger obtained a detailed description of the automorphic discrete spectrum of the metaplectic group $\mathrm{Mp}_2(\A)$.
As noted in the introduction of \cite{GI}, one of the important steps in the argument of \cites{Wal4} is to show the nonvanishing of the central $L$-value $L(\frac12, \pi\otimes\chi)$ twisted by some quadratic character $\chi$ for any irreducible cuspidal automorphic representation $\pi$ of $\PGL_2(\A)$. 
Waldspurger proved such nonvanishing by using a correspondence of automorphic representations between $\GL_2(\A)$ and its double cover (\cite{Wal4}*{Th\'eor\`em 4}).

Friedberg and Hoffstein \cite{FH} proved that there are infinite number of quadratic characters $\chi$ with $L(\frac12, \pi\otimes\chi)\neq 0$.
We briefly review their result.

Fix a finite set $S$ of places of $F$ and a quadratic character $\chi_0$ of $\A^\times/F^\times$.
Let $\Psi(S; \chi_0)$ denote the set of quadratic characters $\chi$ of $\A^\times/F^\times$ such that $\chi_v=\chi_{0,v}$ for all $v\in S$.
The next theorem is \cite{FH}*{Theorem B}.

\begin{thm}\label{thm:FH}
Let $\pi'$ be an irreducible cuspidal automorphic representation of $\GL_2(\A)$ which is self-dual.
Suppose that we have $\ve(1/2, \pi'\otimes\chi)=1$ for some $\chi\in\Psi(S;\chi_0)$.
Then there are infinitely many $\chi\in\Psi(S;\chi_0)$ such that $L(1/2, \pi'\otimes\chi)\neq0$.
\end{thm}

One can rewrite this theorem on special values of $L$-functions in terms of nonvanishing of periods by using another theorem from \cite{Wal4}.
We see that \ref{thm:FH} has following corollary. 
For each non-trivial $\chi\in\Psi(S;\chi_0)$, let $E_\chi$ be the corresponding quadratic extension of $F$ via the global class field theory.
\begin{cor}\label{cor:FH-period}
Let $\pi'$ be an irreducible cuspidal automorphic representation of $\GL_2(\A)$ with trivial central character.
Suppose that we have $\ve(1/2, \pi'\otimes\chi)=1$ for some $\chi\in\Psi(S;\chi_0)$ and $L(1/2, \pi')\neq0$.
Then there are infinitely many $\chi\in\Psi(S;\chi_0)$ such that there exists a quaternion algebra $D_\chi$ over $F$ such which $E_\chi$ embeds and an irreducible cuspidal automorphic representation $\pi$ of $D_{\chi, \A}^\times$ which is $E_\chi^\times$-distinguished and $\pi'=\JL(\pi)$.
\end{cor}

Our results, Theorems \ref{thm:nonvanishing423} and \ref{thm:infinite}, together lead to the following assertion.
\begin{thm}\label{thm:main}
Let $\pi$ be an irreducible cuspidal automorphic representation of $D_\A^\times$ with trivial central character which is not $1$-dimensional.
Suppose we have $L(1/2, \pi)\neq0$.
Then we can take a quadratic \'etale algebra $\cE_v$ over $F_v$ for each $v\in S$ so that there are infinitely many $E\in X(D)$ such that $\pi$ is $E^\times$-distinguished and $E_v=\cE_v$ for any $v\in S$.
\end{thm}

Compare Theorem \ref{thm:main} with Corollary \ref{cor:FH-period}.
There seems to be no direct implication in neither direction.

\subsection{Functional equations}\label{sec:20200424}

Recall that $D$ is a quaternion algebra over a number field $F$, $\pi$ is an irreducible cuspidal automorphic representation of $PD^\times_\A$ and $\phi=\otimes_v\phi_v\in\pi$ is a cusp form.

The local functional equations of our zeta functions are established by the research of Wen-Wei Li in much more general setting, see \cite{Li4}.
Here, we give a functional equation of the series $\xi(D, \cE_S, \phi, s)$ by combining his local functional equations with the functional equation of the global zeta function.

Define the local zeta function $Z_{\cE_v}^\vee(\phi_v,\Phi_v,s)$ for the contragredient representation $\rho^\vee$ in the same manner as $Z_{\cE_v}(\phi_v,\Phi_v,s)$. 
Li's local functional equation \cite{Li4}*{Theorem 5.6} states that there exist meromorphic functions $\gamma_{\cE_v,\cR_v}(\phi_v, s)$ on $\C$ which satisfy the local functional equation
\begin{equation}\label{eq:glfunct}
Z_{\cE_v}^\vee(\widehat{\Phi_v},\phi_v,2-s)=\sum_{\cR_v\in X(D_v)} \gamma_{\cE_v,\cR_v}(\phi_v, s)\, Z_{\cR_v}(\Phi_v, \phi_v,s)
\end{equation}
for any Schwartz function $\Phi_v$ on $V(F_v)$.
Here $\widehat{\Phi_v}$ denotes the Fourier transform.
Note that in order to derive \eqref{eq:glfunct} from \cite{Li4}*{Theorem 5.6}, we need the multiplicity one of toric periods $\dim \mathrm{Hom}_{\cE_v^\times}(\pi_v,\trep)\leq 1$. 
Using Theorem \ref{thm:explicit}, \eqref{eq:glfunct} and the functional equation $Z(\Phi,\phi,
s)=Z^\vee(\widehat\Phi,\phi,2-s)$ (cf. Corollary \ref{cor:funct}), we obtain
\begin{equation}\label{eq:xifunct}
\xi(D, \cE_S, \phi, s)  
=\sum_{\cR_S\in X(D_S)} \left(\prod_{v\in S}\gamma_{\cR_v,\cE_v}(\phi_v, s) \right)  
\xi(D, \cR_S, \overline{\phi}, 2-s).
\end{equation}
Here, $S$ is a finite set of places satisfying Condition \ref{condition} as in Theorem \ref{thm:explicit}.
The explicit computation of the gamma factors $\gamma_{\cR_v,\cE_v}(\phi_v, s)$ is carried out in \cite{Sato2} and \cite{Sato3} for some representations of real groups.
In the following proposition, we will show a nonvanishing result at the real place by using the functional equation \eqref{eq:xifunct} and an analytic property of $\gamma_{\cR_v,\cE_v}(\phi_v, s)$, which follows from \cite{Sato3}*{Theorem 2}.
\begin{prop}
Suppose $F=\Q$, $D_\inf=\M_2(\R)$, and $\pi_\inf$ is a spherical principal series of $PD_\infty^\times=\PGL_2(\R)$.
Assume $L(1/2, \pi)\neq 0$.
Then, there are infinitely many real (resp. imaginary) quadratic fields $E\in X(D)$ such that $\pi$ is $E^\times$-distinguished.
\end{prop}
\begin{proof}
For $\cE_\inf=\R\oplus \R$ or $\C$, recall that $X(D,\cE_\inf,\pi)$ is the set of $R\in X(D,\cE_\inf)$ such that $\pi$ is $R^\times$-distinguished. 
By Theorem \ref{thm:infinite} it is sufficient to prove that neither $X(D,\R\oplus \R,\pi)$ nor $X(D,\C,\pi)$ is empty. 
From the assumption $L(1/2, \pi)\neq 0$ and Theorem \ref{thm:nonvanishing423}, we see that not both of the $X(D,\R\oplus \R,\pi)$ and $X(D,\C,\pi)$ are empty. 
Hence we need to show the following two cases do not occur:
\begin{itemize}
    \item[(i)] $X(D,\R\oplus \R,\pi)\neq \emptyset$ and $X(D,\C,\pi)=\emptyset$.
    \item[(ii)] $X(D,\R\oplus \R,\pi)= \emptyset$ and $X(D,\C,\pi)\neq \emptyset$.
\end{itemize}
We only consider the first case.
The other one is similar.

Assume we are in the case (i).
Set $S_0\coloneqq S\setminus \{\inf\}$.
Since $L(1/2, \pi)\neq 0$ and Proposition \ref{prop:202001241}, there exists an element $\cR_{S_0}\in X(D_{S_0})$ such that $\xi(D,(\R\oplus\R,\cR_{S_0}),\phi,s)$ is not identically zero.
Besides, $X(D,\C,\pi)=\emptyset$ implies $\xi(D,(\C,\cR_{S_0}),\phi,s)\equiv 0$.
Let $\phi_\inf\in\pi_\inf$ be a spherical vector and put
    \[
    \Gamma(\phi_\inf,s)\coloneqq 
        \begin{pmatrix} 
        \gamma_{\R\oplus\R,\R\oplus\R}(\phi_v,s) & \gamma_{\C,\R\oplus\R}(\phi_v,s)\\ \gamma_{\R\oplus\R,\C}(\phi_v,s) & \gamma_{\C,\C}(\phi_v,s)
        \end{pmatrix}.
    \]
Then the functional equation \eqref{eq:xifunct} becomes
\begin{equation}\label{eq:funct0423}
    \begin{pmatrix} 
    \xi(D,(\R\oplus\R,\cR_{S_0}),\phi,s) \\ 
    \xi(D,(\C,\cR_{S_0}),\phi,s) 
    \end{pmatrix}
= \Gamma(\phi_\inf,s)  
    \begin{pmatrix} 
    \Xi_{\R\oplus \R}(\phi,s) \\ 
    \Xi_{\C}(\phi,s) 
    \end{pmatrix}.
\end{equation}
Here, we set
    \[
    \Xi_{\mathcal{Q}}(\phi,s)\coloneqq 
    \sum_{\substack{\cE_S\in X(D_S) \\\cE_\inf=\mathcal{Q}}} \,
    \left(\prod_{v\in S_0}\gamma_{\cE_v,\cR_v}(\phi_v,s) \right)
    \xi(D,\cE_S,\overline{\phi},2-s)
    \]
for $\mathcal{Q}=\R\oplus \R$ or $\C$.
By \cite{Sato3}*{Theorem 2}, all entries of the matrix $\Gamma(\phi_\inf,s)$ are not identically zero.
From $\xi(D,(\R\oplus\R,\cR_{S_0}),\phi,s)\not\equiv 0$ and \eqref{eq:funct0423}, at least one of $\Xi_{\R\oplus\R}(\phi,s)$ and $\Xi_{\C}(\phi,s)$ is not identically zero.
Since $\xi(D, (\C,\cR_{S_0}), \phi, s)\equiv 0$, \eqref{eq:funct0423} shows that $\Xi_{\R\oplus\R}(\phi,s)$ and $\Xi_{\C}(\phi,s)$ are both not identically zero.
Hence  $X(D,\C,\pi)$ is non-empty that contradicts the assumption (i).
\end{proof}
This observation suggests that the gamma factors know something about nonvanishing of periods at bad places. 
Thus, in order to extend our results to cover Theorem \ref{thm:FH}, we should clarify the analytic properties of the gamma factors $\gamma_{\cE_v,\cR_v}(\phi_v,s)$ in \eqref{eq:glfunct}.
This would be a future work.

Using the result of this paper, we present an explicit mean value formula for toric periods in \cite{SW}.
In order to obtain a remainder term of that mean value formula, the convex bound of $\xi(D,\cE_S,\phi,s)$ is required. 
For the convexity bound of $\xi(D,\cE_S,\phi,s)$, we need further properties of $\gamma_{\cE_v,\cR_v}(\phi_v,s)$, see \cite{SS}*{Theorem 2 (iii)}. 
In \cite{OT}, Oshita and Tsuzuki studied a similar local zeta function for another prehomogeneous vector space.
They obtained the entireness of gamma factors, which is necessary for the convex bound of global zeta functions.

\subsection{Prior studies on zeta functions}

In this subsection, we give an overview of prior research on this topic.
Our global zeta functions were originally studied by F. Sato \cite{Sato3} for spherical Maass cusp forms in the non-ad\'elic setting.

In the paper \cite{Sato}, Sato introduced prehomogeneous zeta functions with automorphic forms. 
After introducing new zeta functions, he focused efforts on proving their global and local functional equations for several specific prehomogeneous vector spaces. 
We will discuss the relationship between his result \cite{Sato3} and ours in detail later. 
The local functional equation of the zeta functions attached to spherical representations are shown in general setting under some conditions by Bopp and Rubenthaler \cite{BR}. 
In that paper, they aimed to generalize the Godement-Jacquet theory.
Recently,  the theory of the local functional equations has greatly developed by the works of Wen-Wei Li \cites{Li1,Li2,Li3,Li4}.
According to his results, the local zeta functions for arbitrary admissible representations satisfy certain local functional equations.
As mentioned above, further study of the gamma factor is desirable for non-vanishing theorems and convex bounds.

As for the global theory, a lot of open questions are raised in \cite{Li1}*{Chapter 8}. 
We mention that this paper answers those questions for a specific prehomogeneous vector space. 
Firstly, he explained in \cite{Li1}*{Corollary 8.3.7} that the nonvanishing of the zeta functions implies the automorphic representation is distinguished.
In this paper, we will show that the zeta function has a pole provided the central value of the $L$-function is nonzero, and obtain a sufficient condition for distinction.
Secondly, the meromorphic continuation and the functional equation of the zeta functions conjectured in \cite{Li1}*{Definition 8.4.1} is obtained in Theorem \ref{thm:global} and Corollary \ref{cor:funct}.
Lastly, we will obtain the essential factorization \eqref{eq:20200105e1} of our zeta functions, which is an expression as a sum of Euler products of local zeta functions proposed in \cite{Li1}*{Definition 8.2.5}.
This expression will be obtained from Saito's formula and Waldspurger's formula and plays an important role in the application to the study of the period integrals.
In addition, we explicitly compute the local zeta functions and deduce from it the explicit formula (Theorem \ref{thm:explicit}) of our global zeta functions.

We will now explain in detail the relationship between the previous studies and our explicit formula.
Let us consider the $L$-series
    \[
    \mathscr{L}^S(\pi,\eta,s)\coloneqq L^2(2s-\frac12, \pi) / L^2(2s,\eta). 
    \]
For $\Re(s)>3/4$, we can rewrite it as 
    \[
    \mathscr{L}^S(\pi,\eta,s)=\sum_\fa \frac{A(\fa)}{N(\fa)^{2s}},
    \]
where $\fa$ runs through integral ideals prime to $S$, and $N(\fa)$ denotes the norm of $\fa$. 
Note that $A(\fa)$ is in $\R$. 
Using this expression, we define its Rankin-Selberg convolution by
    \[
    (\mathscr{L}^S(\pi,\eta)\otimes \mathscr{L}^S(\pi,\eta))(s)\coloneqq \sum_\fa \frac{A(\fa)^2}{N(\fa)^{2s}}.
    \]
Putting $L(2s-\frac12,\pi_v)/L(2s,\eta_v)=\sum_{m=0}^\inf a_v(m)q_v^{-2sm}$, we see
    \[
    (\mathscr{L}^S(\pi,\eta)\otimes \mathscr{L}^S(\pi,\eta))(s)=\prod_{v\notin S} \left( \sum_{m=0}^\inf a_v(m)^2 \, q_v^{-2sm} \right) .
    \]
Therefore, by a direct computation for each $v$, we obtain
\begin{equation}\label{eq:rsformula}
(\mathscr{L}^S(\pi,\eta)\otimes \mathscr{L}^S(\pi,\eta))(s)=L^S(1,\pi,\mathrm{Ad})\times \mathcal{D}_E^S(\pi,s) .    
\end{equation}
We use the equation \eqref{eq:rsformula} to relate our zeta functions to some zeta functions of previous studies.

Let us consider the following conditions:
\begin{itemize}
\item $F=\Q$, $D=M_2$, and $S=\{\inf,2\}$.
\item $\pi=\otimes_v \pi_v$ and $\pi_v$ is spherical for any $v$.
\item $\phi=\otimes_v\phi_v\in \pi$ is an even Hecke-Maass form for $\SL_2(\Z)$ of weight zero. 
\item $T(\phi)$ is a Hecke-Maass form for $\Gamma_0(4)$ of weight $\frac12$, which is the theta lift of $\phi$, cf. \cite{KS}. 
\item $c(n)$ is the $n$-th Fourier coefficient of $T(\phi)$ $(n\in\Z)$. 
\end{itemize}
Write $d_E$ for the discriminant of a quadratic field $E\in X(D)$ over $\Q$. 
By a relation between Hecke eigenvalues and Fourier coefficients, we find
\begin{equation}\label{eq:app1}
\sum_{m\in \N, (m,2)=1} \frac{|c(d_E\, m^2)|^2 }{m^{2s}}=|c(d_E)|^2\, (\mathscr{L}^S(\pi,\eta)\otimes \mathscr{L}^S(\pi,\eta))(s). 
\end{equation}
Applying \eqref{eq:rsformula} and the Kohnen-Zagier formula \cite{BM10}*{Theorem 1.4} to \eqref{eq:app1}, there exists a meromorphic function $G(s)$ independent of $E$ such that 
\begin{equation*}\label{eq:app2}
\xi(D, \cE_S, \phi, s)= G(s)\, \sum_{E\in X(D,\cE_S)} \sum_{m\in \N, (m,2)=1}  \frac{|c(d_E\, m^2)|^2}{|d_Em^2|^s} \, \alpha_E^{\cE_S}(\phi)\prod_{v\in S}\alpha_{E_v}^\#(\phi_v, \phi_v).
\end{equation*}
Hence, this means that $\xi(D, \cE_S, \phi, s)$ is a Dirichlet series of the squares of Fourier coefficients of Maass cusp forms of half-integral weights.  
This formula has become a partial restoration of the formula \cite{Sato3}*{Theorem A}, and our explicit formula can be seen as a sort of its generalization. 
Using the Kohnen-Zagier formula and its generalization \cite{BM07} in a similar argument, our explicit formula can be regarded as a Dirichlet series of squared Fourier coefficients of holomorphic cusp forms of half-integral weights.

It is mentioned in \cite{Sato3} that \cite{Sato3}*{Theorem A} is an analogue of \cite{BS}*{Corollaries 2.2, 2.3} with respect to the Maass cusp forms. 
Let us explain the reason for this. 
Consider the following conditions:
\begin{itemize}
\item $F=\Q$, $D$ is definite, and $\cO$ is a maximal order in $D$. 
\item $\pi=\otimes_v \pi_v$ is infinite dimensional, $\pi_\inf$ is trivial, and $\pi_v$ has a $\cO_v^\times$-fixed vector for each $v<\inf$, where $\cO_v\coloneqq \cO\otimes\Z_v$. 
\item $\phi=\otimes_v\phi_v\in \pi$, and $\phi_v$ is a $\cO_v^\times$-fixed vector for each $v<\inf$. 
\item $\mathcal{W}(\phi)$ is a Waldspurger lift of $\phi$, that is, $\mathcal{W}(\phi)$ is a holomorphic cusp form of weight $\frac32$, cf. \cite{SWY}*{\S3}. 
\item $a(n)$ is the $n$-th Fourier coefficient of $\mathcal{W}(\phi)$ $(n\in\N)$. 
\end{itemize}
It follows from \cite{SWY}*{Theorem 3.5} that
    \[
    |a (|d_E|)|^2= c_\pi(E)^2 \times |d_E|\, L(1,\eta_E)^2 \times |\cP_E(\phi)|^2,
    \]
where $c_\pi(E)$ is a power of $2$ which is bounded and depends only on $E\in X(D)$. 
In addition, according to \cite{SWY}*{Lemma 5.2}, a similar equation as $\eqref{eq:app1}$ with $c(n)$ replaced with $a(n)$ holds. 
Therefore, without using the Kohnen-Zagier formula, we may regard $\xi(D, \cE_S, \phi, s)$ as a Dirichlet series of $|a (|d_E m^2|)|^2$ satisfying \eqref{eq:app2}. 
Under the identification $\mathrm{GO}(4)^+\cong (D^\times \times D^\times) /\Q^\times$ and $V\cong \M(4,2)$, the integration over $D_\A^\times \times D_\A^\times$ in the definition of $Z(\Phi,\phi,s)$ is regarded as the Yoshida lifting $Y(\phi)$ of $\phi\otimes \overline{\phi}$ (a theta correspondence). 
Hence, one can identify $Z(\Phi,\phi,s)$ with the K\"ocher-Maass series of $Y(\phi)$ by choosing a suitable test function $\Phi$. 
This is studied in \cite{BS}.

\subsection{Organization of this paper}

The following is the structure of the paper.
In the next section, we recall the structure theory of the prehomogeneous vector space we consider.
In \S~\ref{sec:globalzeta424}, we recall the global zeta functions with automorphic forms and show their basic properties such as meromorphic continuation, location of poles, and functional equation. 
In \S~\ref{sec:explicit}, we will prove the explicit formula (Theorem \ref{thm:explicit}).
Saito's method \cites{Saito1,Saito2} plays an important roll. 
Finally in \S~\ref{sec:localzeta424}, we compute our local zeta functions at unramified places, which is necessary for Theorem \ref{thm:explicit}.

\vspace{5mm}
\noindent
{\bf Acknowledgments.} 
The authors thank Wen-Wei Li for helpful discussions and sharing his impeccably written preprints with us. 
Needless to say, this paper owes its inspiration to a series of his recent pioneering works.
The authors also would like to thank Kimball Martin for pointing out that the difference between our result and Theorem \ref{thm:FH} is reduced to a local problem.
A further thank you to both Atsushi Ichino and Tamotsu Ikeda for useful comments and kindly answering our questions. 
The first author would like to thank Hang Xue for helpful comments.
The second author thanks Akihiko Yukie for helpful discussions.
M.S. was partially supported by JSPS Research Fellowship for Young Scientists No.20J00434.
S.W. was partially supported by JSPS Grant-in-Aid for Scientific Research (C) No.20K03565, (B) No.21H00972.

\section{Preliminaries}

Throughout this section, we let $F$ be a field of characteristic not two. 

\subsection{Quaternion algebra}\label{sec:quaternoin}

Let $D$ denote a quaternion algebra over $F$.
When $D$ is not division, we may identify $D$ with $\M_2(F)$, and a standard $F$-involution $\iota$ on $D$ can be chosen as $x^\iota\coloneqq \left(\begin{smallmatrix} 0 & 1 \\ -1 & 0 \end{smallmatrix}\right)\,{}^t\!x \left(\begin{smallmatrix} 0 & -1 \\ 1 & 0 \end{smallmatrix}\right)$. 
When $D$ is division, there exist a quadratic extension $E$ of $F$ and an element $b\in F^\times$ such that $D$ is regarded as
    \[
    D=\left\{  \begin{pmatrix} \xi & \eta \\ b\overline{\eta} & \overline{\xi} \end{pmatrix} \in \M_2(E) \mid  \xi,\eta\in E \right\}
    \]
where $E\ni a\mapsto\overline{a}\in E$ is the non-trivial Galois action on $E/F$, and a standard $F$-involution $\iota$ on $D$ can be chosen as $\left(
    \begin{smallmatrix} 
    \xi & \eta \\ 
    b\overline{\eta} & \overline{\xi} 
    \end{smallmatrix}
\right)^\iota=\left(
    \begin{smallmatrix} 
    \overline{\xi} & -\eta \\ 
    -b\overline{\eta} & \xi  
    \end{smallmatrix}
\right)$. 
Set $\det(x)\coloneqq  x\, x^\iota\in F$ and $\Tr(x)\coloneqq x+x^\iota\in F$ $(x\in D)$.

Let $D$ be a quaternion algebra over $F$. 
From the Skolem-Noether theorem, two quadratic \'etale $F$-subalgebras of $D$ are isomorphic if and only if they are conjugate to each other. 
Let $X(D)$ denote a set of representatives of isomorphism classes of quadratic \'etale $F$-algebras of $D$.
For each $E$ in $X(D)$, choose an element $\delta_E\in D$ so that $E=F+F\delta_E$ and $\Tr(\delta_E)=0$. 
Set $d_E\coloneqq \delta_E^2$.
Note that we have $\det(\delta_E)=-d_E$.

\subsection{The space of pairs of quaternion algebras}\label{sec:space}

An algebraic group $G$ over $F$ is defined by 
    \[
    G\coloneqq G_1\times G_2\times G_3, \quad G_1=G_2=\GL_1(D), \quad G_3=\GL_2,
    \]
and an $8$-dimensional vector space $V$ over $F$ is defined by $V(F)=D\oplus D$.
Define a $F$-rational representation $\rho:G\to \GL(V)$ by
    \[
    (x,y) \cdot \rho(g_1,g_2,g_3) \coloneqq  (g_1^{-1}xg_2,g_1^{-1}yg_2)g_3, \quad (g_1,g_2,g_3)\in G, \quad (x,y)\in V.
    \]
Then, the kernel $Z_\rho$ of $\rho$ is
    \[
    Z_\rho\coloneqq \mathrm{Ker}\rho=\{ (a,b,ab^{-1})\in G \mid a,b\in\mathbb{G}_m  \}\cong \mathbb{G}_m\times \mathbb{G}_m.
    \]

Set $H\coloneqq Z_\rho\bsl G$, and we identify $H$ with $\rho(G)$.
Denote by $x\cdot \rho(h)$ or $x\cdot h$ $(x\in V$, $h\in H)$ the faithful $F$-representation of $H$ obtained from the above representation $\rho$ of $G$.
The group of $F$-rational characters of $H$ is generated by $\omega_0$ which is defined by
    \[
    \omega_0(h)\coloneqq \det(g_1)^{-1}\det(g_2)\, \det(g_3) \qquad (h=Z_\rho(g_1,g_2,g_3)\in H).
    \]
Let $\mathcal{W}$ denote the vector space over $F$ defined by $\mathcal{W}(F)=\{ x\in D \mid  \Tr(x)=0 \}$.  
A representation $\rho_0$ of $H$ on $\mathcal{W}$ over $F$ is defined by $x\cdot \rho_0(h)\coloneqq g_1^{-1}x\, g_1$ where $h=Z_\rho(g_1,g_2,g_3)\in H$.
Define a $F$-morphism $P_0$ from $V$ to $\mathcal{W}$ by $P_0(x,y)\coloneqq x\, y^\iota-y \, x^\iota$, $(x,y)\in V$.
Then, $P_0$ satisfies $P_0((x,y)\cdot \rho(h))=\omega_0(h)\, P_0(x,y)\cdot\rho_0(h)$.


For each point $x\in V(F)$, we say that $x$ is regular if $x\cdot \rho(G(\bar{F}))$ is open dense in $V(\bar{F})$, and $x$ is singular otherwise.
Suppose that $x_0$ is a regular element in $V(F)$. 
Let $H_x$ denote the stabilizer of $x$ in $H$, and $H_x^0$ the connected component of $1$ in $H$.
Set $P(x)\coloneqq -\det(P_0(x))$ and $\omega(h)\coloneqq \omega_0(h)^2$.
Then we have $P(x\cdot\rho(h))=\omega(h)P(x)$ for all $h\in H_x(F)$.
The polynomial $P(x)$ on $V$ is called the fundamental relative invariant. 
Set $V^0\coloneqq \{ x\in V \mid P(x)\neq0 \}$, and then $V^0(F)$ means the set of regular elements in $V(F)$.


\subsection{Remarks}
Let us explain a brief history of studies on our prehomogeneous vector space $(G,\rho,V)$. 
Let $\bar{F}$ denote an algebraic closure of $F$.
Set $\underline{G}\coloneqq \GL_2(\bar{F})^3$, and $\underline{V}\coloneqq \bar{F}^2\otimes \bar{F}^2\otimes\bar{F}^2$.
A rational irreducible representation $\underline{\rho}$ is defined by $v\cdot\underline{\rho}(g)=(v_1g_1)\otimes(v_2g_2)\otimes(v_3g_3)$ where $g=(g_1,g_2,g_3)\in \underline{G}$ and $v=(v_1,v_2,v_3)\in \underline{V}$.
Then, the triple $(\underline{G},\underline{\rho},\underline{V})$ is a prehomogeneous vector space, that is, it has an open dense orbit.
This space is classified into a special case $(n,m)=(4,2)$ of the irreducible reduced prehomogeneous vector spaces (20) $(\SO_n\times \GL_m,M_{n,m})$ in \cite{SK}*{Theorem 24 in \S~3}.
Its split $F$-form $F^2\otimes F^2\otimes F^2$ was introduced in \cite{WY} as the $D_4$ case from the viewpoint of field extensions, and its $\Z$-structure $\Z^2\otimes \Z^2\otimes \Z^2$ is well-known as Bhargava's cube \cite{Bhargava}.
Its $F$-forms were classified in [Saito, Theorem 2.11 (3)]. 
The $F$-rational orbits of the non-split $F$-forms $\GL_2(E)\times \GL_2(F)$ and $\GL_2(E')$, where $E$ (resp. $E'$) is a quadratic (resp. cubic) extension of $F$, were studied in \cite{KY} and \cite{GS}.
In addition, Kable and Yukie obtained density theorems in \cite{KY1} and \cite{Yukie} using the prehomogeneous zeta functions.
Besides, Taniguchi \cites{Taniguchi1,Taniguchi2} studied another type non-split $F$-form, which is a pair of division quaternoin algebras. 
The spaces we deal with are the split $F$-form and Taniguchi's non-split $F$-form.

\subsection{Orbit decompositions}\label{sec:orbits}

In what follows, we describe $F$-rational orbits and stabilizers for $(G,\rho,V)$. 
Let $G_x$ denote the stabilizer of $x$ in $G$. 
We write $O_n$ for the zero matrix in $\M_n$, and $E_{ij}$ for the matrix unit of the entry $(i,j)$ in $\M_n$.
The following propositions can be proved by direct calculations.

\begin{prop}\label{prop:sor1} 
Suppose that $D=\M_2(F)$.
The following is a list of representative elements of singular $G(F)$-orbits in $V(F)$.
\begin{itemize}
\item $x_0\coloneqq (O_2,O_2)$, $G_{x_0}=G$.

\item $x_1\coloneqq (O_2,E_{12})$, $G_{x_1}=\{ (  \left(
    \begin{smallmatrix}
    a&*  \\ 
    0 & * 
    \end{smallmatrix}
\right)  ,\left(
    \begin{smallmatrix}
    *&*  \\ 
    0 & b 
    \end{smallmatrix}
\right)  ,\left(
    \begin{smallmatrix}
    *&*  \\ 
    0 & c 
    \end{smallmatrix}
\right)  )\in G \mid a=bc   \}$.

\item $x_2\coloneqq (O_2,I_2)$, $G_{x_2}=\{ ( g,h,\left(\begin{smallmatrix}*&*  \\ 0 & c \end{smallmatrix}\right) )\in G \mid g=c h  \}$.

\item $x_3\coloneqq (E_{11},E_{12})$, $G_{x_3}\cong G_{x_2}$.

\item $x_4\coloneqq (E_{12},E_{22})$, $G_{x_4}\cong G_{x_2}$.

\item $x_5\coloneqq (E_{12},I_2)$, \\
$G_{x_5}=\left\{ \left( \left(
    \begin{smallmatrix}
    a_1&a_2 \\ 
    0&a_3 
    \end{smallmatrix}
\right),\left(
    \begin{smallmatrix}
    a_1&a_2 \\ 
    0&a_3 
    \end{smallmatrix}
\right)\left(
    \begin{smallmatrix}
    1/c_3 & -c_2/c_1 c_3 \\ 
    0 & 1/c_3 
    \end{smallmatrix}
\right),\left(
    \begin{smallmatrix} 
    c_1 & c_2 \\ 
    0 & c_3 
    \end{smallmatrix}
\right) \right)\in G \, \middle| \,  a_1c_3=a_3c_1 \right\}$.
\end{itemize}
\end{prop}

\begin{prop}\label{prop:sor2}
Suppose that $D$ is division. There are only two singular $G(F)$-orbits in $V(F)$.
One is the trivial orbit $x_0\coloneqq (0,0)$, and the other is $x_2\cdot \rho(G(F))$ where $x_2\coloneqq (0,1)$.
\end{prop}


For each $E\in X(D)$, define a similitude orthogonal group $\mathrm{GO}_{2,E}$ over $F$ by
    \[
    \mathrm{GO}_{2,E}\coloneqq \left\{  g\in \GL_2 \, \middle| \, \exists \mu(g)  \in \mathbb{G}_m \;\; \text{s.t.} \;\; {}^t\!g
        \begin{pmatrix} 
        1&0 \\ 
        0&-d_E 
        \end{pmatrix} 
    g=  \mu(g) 
        \begin{pmatrix} 
        1&0 \\ 
        0&-d_E 
        \end{pmatrix}  
    \right\}.
    \]
For each algebraic group $\mathscr{U}$ over $F$, let $\mathscr{U}^0$ denote the connected component of $1$ in $\mathscr{U}$. 
Then, $\mathrm{GO}_{2,E}^0$ is isomorphic to $R_{E/F}(\mathbb{G}_m)$ over $F$.

\begin{prop}\label{prop:orbits}
A set of representative elements of regular $G(F)$-orbits in $V(F)$ is given by $\{ \, x_E \mid E \in X(D) \, \}$, where we set 
    \[
    x_E\coloneqq (1,\delta_E)\in V(F).
    \] 
Notice that $P(x_E)=4d_E\in d_E(F^\times)^2$.
In addition,
    \begin{multline*}
    G_{x_E}^0= \Big\{   (  a_1+ b_1 \, \delta_E , \, a_2+b_2 \, \delta_E, \,  g_3)\in G \, \Big| \, \\
    a_j^2- b_j^2d \neq 0  \; \; (j=1,2), \; \; g_3=\left(
        \begin{smallmatrix} 
        a_1&b_1d \\ 
        b_1&a_1 
        \end{smallmatrix}
    \right)\left(
        \begin{smallmatrix} 
        a_2&b_2d \\ 
        b_2&a_2 
        \end{smallmatrix}
    \right)^{-1}   \Big\},    
    \end{multline*}
    \[
    [G_{x_E}:G_{x_E}^0]=2, \quad G_{x_E}^0\cong \mathrm{GO}_{2,E}^0 \times \mathrm{GO}_{2,E}^0.
    \]
It follows from the Skolem-Noether theorem that there exists an element $\gamma \in D^\times$ such that $(\gamma, \gamma, \diag(1,-1))\in G_{x_E}(F) \setminus G_{x_E}^0(F)$ and $\gamma^{-1} \delta_E \gamma = -\delta_E$.
\end{prop}


\section{Global zeta functions with automorphic forms}\label{sec:globalzeta424}

To simplify the notation, the measure $\d h$ on $H(\A)$ in this section \S\ref{sec:globalzeta424} is normalized differently from \S\ref{sec:1} and \S\ref{sec:explicit}.

\subsection{Number fields}\label{sec:notation}

Let $F$ be an algebraic number field, and let $\Sigma$ denote the set of all the places of $F$.
For any $v\in\Sigma$, we denote by $F_v$ the completion of $F$ at $v$.
If $v<\inf$, we write $\fo_v$ for the ring of integers of $F_v$, and $\varpi_v$ for a prime element of $\fo_v$. We put $q_v=\#(\fo_v/\varpi_v\fo_v)$.
Let $\A$ denote the adele ring of $F$. 
Let $\d x$ denote the Haar measure on $\A$ normalized by $\int_{\A/F} \d x=1$. 
For each $v<\inf$, we fix a Haar measure $\d x_v$ on $F_v$ normalized by $\int_{\fo_v}\d x_v=1$. 
Further, we choose Haar measures $\d x_v$ on $F_v$ for $v\mid\inf$ so that $\d x = |\Delta_F|^{-1/2} \prod_{v\in\Sigma} \d x_v$ holds, where $|\Delta_F|$ denotes the absolute discriminant of $F/\Q$.


We denote by $|\; |_v$ the normalized valuation of $F_v$.
Then, we have $\d (ax_v)=|a|_v \d x_v$ for any $a\in F_v^\times$.
Define the idele norm $|\; |=|\; |_\A$ on $\A^\times$ by $|x|=|x|_\A=\prod_{v\in\Sigma} |x_v|_v$ for all $x=(x_v)\in\A^\times$.
Choose a non-trivial additive character $\psi_\Q$ on $\A_\Q/\Q$, and set $\psi_F=\psi_\Q \circ \Tr_{F/\Q}$. Then, $\d x$ is the self-dual Haar measure with respect to $\psi_F$. For each $v\in \Sigma$, we set $\psi_{F_v}=\psi_{F}|_{F_v}$. 
Set
    \[
    c_v\coloneqq\zeta_{F_v}(1)=(1-q_v^{-1})^{-1}  \;\; \text{if $v<\inf$}, \qquad c_v\coloneqq 1 \;\; \text{if $v\mid\inf$}.   
    \]
For each $v\in\Sigma$, let $\d^\times x_v$ denote a Haar measure on $F_v^\times$ as ${\displaystyle \d^\times x_v= c_v \, \tfrac{\d x_v}{|x|_v} }$. 
The idele norm $|\;|$ induces an isomorphism $\A^\times/\A^1\to\R_{>0}$.
Choose the Haar measure $\d^\times x=|\Delta_F|^{-1/2}\prod_{v\in\Sigma}\d^\times x_v$
on $\A^\times$ and normalize the Haar measure $\d^1 x$ on $\A^1$ in such a way
that the quotient measure on $\R_{>0}$ is $\d t/t$, where $\d t$ is the
Lebesgue measure on~$\R$.
We set
\begin{equation}\label{eq:c_F}
c_F\coloneqq \int_{F^\times\bsl \A^1} \d^1 x.  
\end{equation}

For a finite subset $S$ of $\Sigma$, we set $F_S=\prod_{v\in S}F_v$.
Define the norm $|\; |_S$ on $F_S^\times$ by $|x|_S=\prod_{v\in S}|x_v|_v$ for $x=(x_v)\in F_S^\times$.
Let $\mathcal{V}$ be a finite dimensional vector space over $F$. 
Denote by $\cS(\mathcal{V}(F_S))$ (resp. $\cS(\mathcal{V}(\A))$) the Schwartz space of $\mathcal{V}(F_S)$ (resp. $\mathcal{V}(\A)$). 

\subsection{Global zeta functions and their principal parts}\label{sec:3.2}
Let $D$ be a quaternion algebra over $F$, set $D_\A\coloneqq D\otimes \A$. 
Then we have $V(\A)=D_\A\oplus D_\A$.
Let $Z$ denote the center of $\GL_1(D)$, and set $PD^\times=Z(F)\bsl D^\times$, $PD^\times_\A\coloneqq Z(\A)\bsl D_\A^\times$.
Choose an automorphic form $\phi$ in $L^2(PD^\times\bsl PD_\A^\times)$. 
Here, we refer to \cite{GJ}*{p.145} for the sense of automorphic forms. 
Suppose that $\phi$ is cuspidal.

A bilinear form $\langle \; , \; \rangle$ on $V$ is defined by
    \[
    \langle (x_1,x_2) , (y_1,y_2) \rangle\coloneqq \Tr(x_1 y_1)+\Tr(x_2  y_2),
    \]
and the dual space of $V$ is identified with $V$ by $\langle \; , \; \rangle$. 
Let $\rho^\vee$ denote the contragredient representation of $\rho$ on $V$ with respect to $\langle \;\; , \;\; \rangle$, that is, $\langle x\cdot \rho(g) , y\cdot \rho^\vee(g) \rangle=\langle x , y \rangle$. 
Then, one has $Z_\rho=\mathrm{Ker}\rho^\vee$ and
    \[
    (y_1,y_2)\cdot \rho^\vee(g_1,g_2,g_3) =(g_2^{-1} y_1 g_1,g_2^{-1}y_2g_1)\, ^t{}\!g_3^{-1}.
    \]
For $s\in\C$ and $\Phi\in\cS(V(\A))$, we define the global zeta functions $Z(\Phi,\phi,s)$ and $Z^\vee(\Phi,\phi,s)$ as
    \[
    Z(\Phi,\phi,s)\coloneqq \int_{H(F)\bsl H(\A)} |\omega(h)|^s \phi(g_1)\, \overline{\phi(g_2)}\sum_{x\in V^0(F)} \Phi(x\cdot \rho(h))\, \d h,
    \]
    \[
    Z^\vee(\Phi,\phi,s)\coloneqq \int_{H(F)\bsl H(\A)} |\omega^\vee(h)|^s \phi(g_1)\, \overline{\phi(g_2)}\sum_{x\in V^0(F)} \Phi(x\cdot \rho^\vee(h))\, \d h,
    \]
where $h=Z_\rho(g_1,g_2,g_3)\in H(\A)$, $\omega^\vee\coloneqq \omega^{-1}$ and $\d h$ is a Haar measure on $H(\A)$.

\begin{lem}\label{lem:abs}
There exists a sufficiently large constant $T>0$ such that $Z(\Phi,\phi,s)$ and $Z^\vee(\Phi,\phi,s)$ are absolutely and uniformly convergent on any compact set
 in $\{ s\in \C \mid \Re(s)>T \}$. Hence, they are holomorphic on the domain $\{ s\in \C \mid \Re(s)>T \}$.
\end{lem}
\begin{proof}
The absolute convergence is proved by Lemma \ref{lem:absconv}. 
As for the uniform convergence, one can deduce it from the argument in the proof of Lemma \ref{lem:absconv}. 
\end{proof}

We set
    \[
    Z_+(\Phi,\phi,s)\coloneqq \int_{H(F)\bsl H(\A) , |\omega(h)|\geq 1} |\omega(h)|^s \phi(g_1)\, \overline{\phi(g_2)}\sum_{x\in V^0(F)} \Phi(x\cdot \rho(h))\, \d h ,
    \]
    \[
    Z_+^\vee(\Phi,\phi,s)\coloneqq \int_{H(F)\bsl H(\A),|\omega^\vee(h)|\geq 1} |\omega^\vee(h)|^s \phi(g_1)\, \overline{\phi(g_2)}\sum_{x\in V^0(F)} \Phi(x\cdot \rho^\vee(h))\, \d h .
    \]
The zeta integrals $Z_+(\Phi,\phi,s)$ and $Z_+^\vee(\Phi,\phi,s)$ are entire on $\C$ by Lemma \ref{lem:abs}.


Choose a Haar measure $\d g$ on $D_\A^\times$. 
Fix an infinite place $v_{\infty,1}\in\Sigma$ and consider an embedding $\R_{>0}\ni a \mapsto a^{1/[F_{v_{\infty,1}}:\R]} \in F_{v_{\infty,1}}^\times$. 
We identify $\R_{>0}$ with the image of $\R_{>0}$ in $F_{v_{\infty,1}}^\times$ and then we have $a=|a|$ for arbitrary $a\in\R_{>0}$. 
By this identification, the Haar measure $\d^\times x$ on $\A^\times$ is expressed as
    \[
    \d^\times x=\d^1 y \, \frac{\d t}{t} \quad \text{for}  \;\; x=y\, t \;\; \text{$(y\in\A^1$, $t\in\R_{>0})$}. 
    \]
Define a Haar measure $\d z$ on $Z(\A)$ by 
    \[
    \d z\coloneqq t^{-1}\d t \, \d^1 a \quad \text{for}  \quad z=t^{1/2}a \in Z(\A), \;\; a\in\A^1 \;\; \text{and} \;\; t\in\R_{>0}. 
    \]
Then, we obtain the quotient measure $\d g/\d z$ on $PD_\A^\times$. 
Set
    \[
    f_{\phi}(g')\coloneqq \int_{PD^\times \bsl PD_\A^\times}\phi(g)\, \overline{\phi(gg')} \, \frac{\d g}{\d z} \qquad g'\in PD_\A^\times .
    \]
This function $f_{\phi}$ is called a matrix coefficient.
A Godement-Jacquet zeta integral $Z^\GJ(\Psi,f_{\phi},s)$ is defined by
    \[
    Z^\GJ(\Psi,f_{\phi},s)\coloneqq  \int_{D_\A^\times} \Psi(g) \, f_{\phi}(g) \, |\det(g)|^s \, \d g \qquad (\Psi\in\cS(D_\A))
    \]
which is absolutely convergent for $\Re(s)>2$, \cite{GJ}*{Theorem 13.8 in p.179}.
In addition, it is analytically continued to the whole $s$-plane.
Let $\d x$ denote the Haar measure on $D_\A$ normalized by $\int_{D\bsl D_\A}\d x=1$. 
Set $\widehat\Psi(y)\coloneqq \int_{D_\A} \Psi(x) \psi_F(\Tr(xy)) \, \d x$ and then we obtain the functional equation
    \[
    Z^\GJ(\Psi,f_{\phi},s)=Z^\GJ(\widehat{\Psi},f_{\phi}^\vee,2-s)
    \]
where $f^\vee_{\phi}\coloneqq\overline{f_{\phi}}$.

We choose a self-dual Haar measure $\d x$ on $V(\A)$ with respect to $\psi_F(\langle \; ,  \; \rangle)$, that is, we have $\int_{V(F)\bsl V(\A)} \d x=1$.
Set
    \[
    \widehat\Phi(y)\coloneqq \int_{V(\A)} \Phi(x) \psi_F(\langle x, y \rangle) \, \d x .
    \]
Let $K$ denote the standard maximal compact subgroup in $\GL_2(\A)$, see \S\ref{sec:measure1}.
Choose a Haar measure $\d g_3$ on $G_3(\A)=\GL_2(\A)$ as
    \begin{equation}\label{eq:normalization dg_3}
    \d g_3\coloneqq \frac{\d t_1}{t_1} \frac{\d t_2}{t_2} \, \d^1a \, \d^1c \, \d b \, \d k     
    \end{equation}
for $\diag\left(t_1^{1/2} t_2^{-1/2}, t_1^{1/2}t_2^{1/2}\right) \, \left(\begin{smallmatrix}a&b \\ 0&c\end{smallmatrix}\right) k\in \GL_2(\A)$, $t_1,t_2\in\R_{>0}$, $a,c\in\A^1$, $b\in\A$, $k\in K$.
Here, we give a normalization of $\d k$ by $\int_K \d k=1$. 
For each $j=1$ or $2$, we denote by $\d g_j$ (resp $\d z_j$) the Haar measure on $G_j(\A)=D_\A^\times$ (resp. $Z(\A)\cong \A^\times$) as above. 
Then, a Haar measure $\d h$ on $H(\A)$ is chosen by 
    \[
    \d h\coloneqq \d g_1 \, \d g_2 \, \d g_3/\d z_1\, \d z_2
    \]
for $h=Z_\rho (g_1,g_2,g_3)\in H(\A)$ and $(z_1,z_2,z_1z_2^{-1})\in Z_\rho(\A)$.
Set
    \[
    \Phi_K(x)\coloneqq  \int_K \Phi(x\cdot \rho(1,1,k)) \, \d k.
    \]
The principal part of $Z(\Phi,\phi,s)$ is described as below.
\begin{thm}\label{thm:global}
Assume that $\phi$ is cuspidal, and is orthogonal to the constant functions on $PD^\times(F)\bsl PD^\times_\A$.
For any $\Re(s)>T$ and any $\Phi \in \cS(V(\A))$, we obtain
\begin{align*}
Z(\Phi,\phi,s)=& Z_+(\Phi,\phi,s)+Z_+^\vee(\widehat\Phi,\phi,2-s) \\
&  +c_F \frac{Z^\GJ(\big(\widehat{\Phi}_K\big)_1,f_{\phi}^\vee,1)}{2s-3}-c_F \frac{Z^\GJ(\left(\Phi_K\right)_1,f_{\phi},1)}{2s-1},
\end{align*}
where $\Phi_1(x)\coloneqq \Phi(x,0)$.
Note that we have $\widehat{(\Phi_K)}=\widehat{\Phi}_K$ and $\Phi_K(x_1,x_2)=\Phi_K(x_2,x_1)$. 
Hence, $Z(\Phi,\phi,s)$ is meromorphically continued to the whole $s$-plane, and might have a simple pole at $s=3/2$ or $s=1/2$.
\end{thm}
\begin{proof}
The proof will be given in \S~\ref{sec:principleproof}.
\end{proof}
\begin{cor}\label{cor:funct}
The zeta integrals satisfy the functional equation
    \[
    Z(\Phi,\phi,s)=Z^\vee(\widehat\Phi,\phi,2-s).
    \]
\end{cor}
\begin{proof}
Comparing $\rho^\vee$ with $\rho$, one can show that
    \begin{align*}
    Z^\vee(\Phi,\phi,s)=& \, Z_+^\vee(\Phi,\phi,s)+Z_+(\widehat\Phi,\phi,2-s) \\
    &  +c_F \frac{Z^\GJ(\big(\widehat{\Phi}_K\big)_1,f_{\phi},1)}{2s-3}-c_F\frac{Z^\GJ(\left(\Phi_K\right)_1,f_{\phi}^\vee,1)}{2s-1}.
    \end{align*}
The functional equation follows from this equation and Theorem \ref{thm:global}.
\end{proof}

Set $\sK\coloneqq \{(1,1,k)\in G(\A) \mid k\in K\}$.
We say that $\Phi$ is $\sK$-spherical if $\Phi(x\cdot \rho(k))=\Phi(x)$ holds for any $k\in \sK$.
We may suppose that $\Phi$ is $\sK$-spherical without loss of generality, because $\phi(g_1)\, \overline{\phi(g_2)
}$ and $|\det(g)|^s$ are stable under $\sK$.
Note that $\Phi_K$ is $\sK$-spherical.


\begin{prop}\label{prop:202001241}
Let $\pi=\otimes_v \pi_v$ be an automoprhic representation of $D_\A^\times$ with trivial central character.
When $D$ is division, we suppose that $\pi$ is not one dimensional. When $D=\M_2(F)$, we suppose that $\pi$ is cuspidal.
If $L(1/2,\pi)\neq 0$, then there exist $\phi=\otimes_v\phi_v\in\pi$ and $\sK$-spherical $\Phi=\otimes_v\Phi_v\in \cS(V(\A))$ such that $Z(\Phi,\phi,s)$ has a simple pole at $\frac12$ (this means $Z(\Phi,\phi,s)\not\equiv 0$).
\end{prop}
\begin{proof}
Let $K_v$ denote the standard maximal compact subgroup of $\GL_2(F_v)$, see \S\ref{sec:measure1}.
Then, we have $K=\prod_{v\in\Sigma}K_v$ and $\sK=\prod_{v\in\Sigma} \sK_v$ where $\sK_v\coloneqq \{(1,1,k)\in G(F_v) \mid k\in K_v \}$.
Let $S$ denote a finite set of places which contains all infinite places. 
Take a $\sK$-spherical test function $\Phi=\otimes_{v\in\Sigma}\Phi_v\in \cS(V(\A))$ where $\Phi_v\in\cS(V(F_v))$. 
If $S$ is sufficiently large, then we have
    \[
    Z^\GJ(\left(\Phi\right)_1,f_{\phi},1) = L(1/2,\pi) \times \prod_{v\in S}  \frac{Z^\GJ((\Phi_v)_1,f_v,1)}{L_v(1/2,\pi_v)} 
    \]
where $f_{\phi}=\otimes_v f_v$, $f_v(g_v^{-1})=\langle \pi_v(g_v)\phi_v,\phi_v\rangle_v$, and
    \[
    Z^\GJ(\Phi_v,f_v,s)\coloneqq \int_{D_v^\times} f_v(g_v) \, \Phi_v(g_v) \, |\det(g_v)|_v^s \, \d g_v \quad (\Psi_v\in\cS(D_v)). 
    \]
If this does not vanish, then $Z(\Phi,\phi,s)$ has a simple pole at $s=1/2$ by Theorem \ref{thm:global}, particularly $Z(\Phi,\phi,s)$ is not identically zero.
It is sufficient to prove that $Z^\GJ((\Phi_v)_1,f_v,1)$ does not vanish for each $v\in S$ for some vector $\phi_v$ and a $\sK_v$-spherical test function $\Phi_v$.

Let $v$ be a finite place in $S$, and denote by $\fO_v$ a maximal order of $D_v$.
Let $\Psi_{v,m}$ (resp. $\tilde\Psi_{v,m}$) denote the characteristic function of $\varpi_v^{m}\fO_v$ (resp. $\fo_v^\times+\varpi_v^m\fO_v$) where $m\in\Z_{>0}$.
Set
    \[
    \tilde\Phi_{v}(x_1,x_2)\coloneqq \tilde\Psi_{v,m}(x_1)\, \Psi_{v,m}(x_2), \quad \Phi_v(x)\coloneqq \int_{K_v} \tilde\Phi_{v}(x\cdot\rho(1,1,k))\, \d k.
    \]
For any $\left(
    \begin{smallmatrix}
    a& b \\ c& d
    \end{smallmatrix}
\right)\in K_v$, we have $a\in\fo_v^\times$ or  $b\in\fo_v^\times$, hence
    \[
        \begin{cases}
        ax_1\in \fo_v^\times+\varpi_v^{m}\fO_v \;\;  \\
        \text{and} \;\; bx_1\in \varpi_v^{m}\fO_v 
        \end{cases}
\Leftrightarrow 
        \begin{cases}
        x_1\in \fo_v^\times+\varpi_v^{m}\fO_v \;\; \\
        \text{and} \;\; a\in\fo_v^\times, \;\; b\in \varpi_v^{m}\fo_v.
        \end{cases}
    \]
We have $(\Phi_v)_1(x_1)=\Phi_v(x_1,0)=c\,\tilde\Psi_{v,m}(x_1)$ for some constant $c>0$. 
Therefore, there is a positive integer $m$ and a vector $\phi_v\in\pi_v$ such that $Z^\GJ((\Phi_v)_1,f_v,1)\neq 0$.

In what follows, we suppose that $v$ is an infinite place of $F$. 
Set 
    \[
    \Psi_{v,0}(x)\coloneqq 
        \begin{cases}
        e^{-\pi \Tr(x{}^t\!x)} & \text{if $D_v\cong \M_2(\R)$}, \\
        e^{-2\pi \Tr(x{}^t\!\overline{x})} & \text{if $D_v\cong \M_2(\C)$,} \\
        e^{-\pi \Tr(xx^\iota)} & \text{if $D_v\cong \mathbb{H}$,}
        \end{cases}  
    \]
where $\mathbb{H}$ denotes the Hamiltonian quaternion and $x\mapsto x^\iota$ is the usual involution on $\mathbb{H}$. 
By the classification of the unitary representations, we have only to check the following cases: 
\begin{itemize}
\item[(I)] $D_v\cong \M_2(\R)$ and $\pi_v$ is a discrete series.
\item[(II)] $D_v\cong \mathbb{H}$ and $\pi_v$ is finite dimensional. 
\item[(III)] $D_v\cong \M_2(F_v)$ and $\pi_v$ is a unitary principal series.
\end{itemize}


Case (I): As a matrix coefficient of $\pi_v$, we can take $\phi_v$ so that 
    \[
    f_v(\left(
        \begin{smallmatrix}
        a&b \\ 
        c&d
        \end{smallmatrix}
    \right))=
        \begin{cases} 
        \det\left(
            \begin{smallmatrix}
            a&b\\ 
            c&d
            \end{smallmatrix}
        \right)^l \, (a+d+i(b-c))^{-2l} & \text{if $\det\left(
            \begin{smallmatrix}
            a&b\\ 
            c&d
            \end{smallmatrix}
        \right)>0$}, \\
        0& \text{otherwise}, 
        \end{cases}
    \]
for some $l\in\Z_{\geq 1}$. 
Set 
    \[
    \tilde\Psi_v(x_1)\coloneqq \begin{cases} f_v(x_1)^{-1}\det(x_1)^l\, \Psi_{v,0}(x_1) & \text{if $\det(x_1)> 0$,} \\ 0 & \text{if $\det(x_1)\le 0$}, \end{cases}  
    \]
    \[
    \tilde\Phi_{v}(x_1,x_2)\coloneqq \tilde\Psi_v(x_1)\, \Psi_{v,0}(x_2),
    \]
and $\Phi_v(x)\coloneqq \int_{K_v} \tilde\Phi_{v}(x\cdot\rho(1,1,k))\, \d k$. 
When $\det(x_1)>0$, we deduce
    \begin{multline*}
    (\Phi_v)_1(x_1)=\Phi_v(x_1,0)=\int_{K_v} \tilde\Phi_{v}((x_1,0)\cdot\rho(1,1,k))\, \d k \\
    =\Psi_{v,0}(x_1) \Psi_{v,0}(0)  \int_{K_v} f_v(k_1x_1)^{-1}\det(k_1x_1)^l \d k = c_l\,\tilde\Psi_v(x_1) 
    \end{multline*}
where $k=\left( 
    \begin{smallmatrix}
    k_1&k_2 \\ k_3&k_4
    \end{smallmatrix} 
\right)$ and $c_l\coloneqq \int_{K_v} k_1^{2l} \d k>0$, since the function $\Psi_{v,0}(x_1)\, \Psi_{v,0}(x_2)$ on $V(F_v)$ is $\mathscr{K}_v$-spherical, and $f_v(ax_1)=f_v(x_1)$ holds for any $a\in F_v^\times$.  
Therefore, the integral $Z^\GJ((\Phi_v)_1,f_v,1)$ is convergent and non-zero.

Case (II): Since $\mathbb{H}^\times$ is isomorphic to $\R_{>0}\times \mathrm{SU}_2$, 
there is an element $l\in\Z_{\ge 0}$ such that $f_v(x)\, \det(x)^{l}$ is a homogeneous polynomial of degree $2l$ on $\mathbb{H}$. 
We set
    \[
    \tilde\Psi_v(x_1)\coloneqq \overline{f_v(x_1)}\, \det(x_1)^{l+1}\, \Psi_{v,0}(x_1), \quad \tilde\Phi_{v}(x_1,x_2)\coloneqq \tilde\Psi_v(x_1)\, \Psi_{v,0}(x_2) 
    \]
and $\Phi_v(x)\coloneqq \int_{K_v} \tilde\Phi_{v}(x\cdot\rho(1,1,k))\, \d k$. 
Then, the integral $Z^\GJ((\Phi_v)_1,f_v,1)$ is convergent and non-zero by the same reason as in (I).

Case (III): In this case, $\pi_v$ is an induced representation of $\chi\boxtimes \chi^{-1}$, where $\chi$ is a quasi-character on $F_v^\times$. 
Let $\mu$ denote the minimal $K_v$-type of $\pi_v$. Note that $\mu$ is self-dual. 
When $F_v=\R$, $\mu$ is trivial or $\det$, and we set $l=0$ in this case. 
When $F_v=\C$, for some $l\ge 0$, $\mu$ is the $2l$-th symmetric tensor representation of $K_v$ tensored with $\det^{-l}$. 
Then, $\mu$ extends to a homomorphism $\GL_2(F_v)\to \GL_{2l+1}(F_v)$, $\mu\otimes\det^l$ is a polynomial representation of $\GL_2(F_v)$, and $\mu(K_v)$ is contained in $\U_{2l+1}(\C/\R)$. 

Let us consider its compact picture as the representation space $H_{\pi_v}$ of $\pi_v$. 
Set $T\coloneqq\{\diag(a,b)\in K_v \}$ and we define $\chi(\diag(a,b))\coloneqq\chi(ab^{-1})$ by abuse of notation.
Then, a dense subspace of $H_{\pi_v}$ consists of continuous functions $\varphi$ in $L^2(K)$ such that $\varphi(tk)=\chi(t)\varphi(k)$ holds for any $t\in T$. 
Hence, we can choose a non-zero vector $\phi_v$ in the dense subspace such that $\phi_v$ is a linear combination of matrix coefficients of $\mu$, and then $\phi_v$ is in $\mu\otimes\mu \, (\subset L^2(K_v))$ and $\phi_v\, \det^l$ extends to a polynomial representation of $\GL_2(F_v)$. 
Let $B$ denote the subgroup of upper triangular matrices in $\GL_2(F_v)$. 
Since $\mu$ is the $2l$-th symmetric tensor representation, there exists an element $A\in\GL_{2l+1}(F_v)$ such that $A\mu(b)A^{-1}$ is an upper triangular matrix for any $b\in B$ and we have
    \[
    A\mu(\diag(\alpha,\beta))A^{-1}=\diag(\alpha^{l}\beta^{-l}, \alpha^{l-1}\beta^{-l+1},\dots, \alpha^{-l} \beta^{l}).
    \]
Hence, we can define $\phi_v$ as the $(l+1,l+1)$-component of $A\mu(*)A^{-1}$, and hence it satisfies $\phi_v(\left(\begin{smallmatrix} \alpha& * \\ 0&\beta \end{smallmatrix}\right))=(\alpha \beta^{-1})^l$. 


We set
    \[
    \tilde\Psi_v^\mu(x_1)\coloneqq \overline{\mu(x_1)}\, \det(x_1)^{l+1} \overline{\det(x_1)}^{l+1} \, \Psi_{v,0}(x_1), \quad \tilde\Phi_{v}^\mu(x_1,x_2)\coloneqq \tilde\Psi_v^\mu(x_1)\, \Psi_{v,0}(x_2) 
    \]
and $\Phi_v^\mu(x)\coloneqq \int_{K_v} \tilde\Phi_{v}^\mu(x\cdot\rho(1,1,k))\, \d k$. 
Note that these are matrix valued functions. 
By the same reason as in (I), we obtain $(\Phi_v^\mu)_1(x_1)=\Phi_v^\mu(x_1,0)= c \times \tilde\Psi_v^\mu(x_1)$ for some constant $c>0$. 
By $f_v(g)=\int_{K_v}\phi_v(k g)\overline{\phi_v(k)}\d k$ and the Iwasawa decomposition $g=\left(\begin{smallmatrix} a& c \\ 0&b \end{smallmatrix}\right)k'\in\GL_2(F_v)$, $a$, $b\in\R_{>0}$, $c\in F_v$, we have
    \begin{multline*}
    Z^\GJ ((\Phi_v^\mu)_1,f_v,1)= \\
    c\int_{\GL_2(F_v)}\int_{K_v} (\pi_v(g)\phi_v)(1) \, \overline{\phi_v(k)} \, \overline{\mu_v(k^{-1}g)} \, \Psi_{v,0}(g) \, |\det(g)|^{2l+3} \, \d k \, \d g \\
    =c\int_{K_v}\d k \int_{K_v}  \d k'  \int_0^\inf \d^\times a \int_0^\inf \d^\times b \int_{F_v} \d c \\
    \chi(ab^{-1})\, (ab^{-1})^{\frac12} \, \overline{\phi_v(k)} \, \overline{\mu(k^{-1})} \,\overline{\mu(\left(
        \begin{smallmatrix} 
        a& * \\ 
        0&b 
        \end{smallmatrix}
    \right))}\, \overline{\mu(k')}  \, \phi_v(k') \, \Psi_{v,0}(\left(
        \begin{smallmatrix} 
        a& * \\ 
        0&b 
        \end{smallmatrix}
    \right)) \, a^{2l+2} \, b^{2l+3}.   
    \end{multline*}
Hence, applying the Schur orthogonality relation to the above integration over $K_v$, the trace of the matrix $Z^\GJ ((\Phi_v^\mu)_1,f_v,1)$ is
    \[
    c \int_0^\inf \d^\times a \int_0^\inf \d^\times b \int_{F_v} \d c \,
    \chi(ab^{-1})\, (ab^{-1})^{\frac12}  \, \Psi_{v,0}(\left(
        \begin{smallmatrix} 
        a& * \\ 
        0&b 
        \end{smallmatrix}
    \right)) \, a^{3l+2} \, b^{l+3}. 
    \]
By a direct calculation, we can express this integral as a product of the Gamma functions.  
Since the Gamma function has no zero points, the trace of $Z^\GJ ((\Phi_v^\mu)_1,f_v,1)$ is non-zero. 
\end{proof}


\subsection{Proof of Theorem \ref{thm:global}} \label{sec:principleproof}

For the case that $D$ is division and $\phi_1$, $\phi_2$ belongs to a one dimensional representation of $D_\A^\times$ with trivial central character, Taniguchi determined the principal part in \cite{Taniguchi1} by using the smooth Eisenstein series, but here we do not use it in order to make the argument simpler.
We will prove Theorem \ref{thm:global} by using the Poisson summation formula repeatedly and inserting additional terms.
Our arguments are similar to some previous works, see, e.g., \cite{Shintani}*{Proof of Lemma 4} and \cite{Kogiso}, see also \cite{FHW}*{\S~5.2} for the cancellation, which they did not mention.
To sum up, the technical problem we need to handle is to find cancellations of divergent terms and suitable additional terms.

For a moment, let us suppose $D= \M_2(F)$ and $\phi_1$ and $\phi_2$ are constant functions in the space of the trivial representation.
In this case, the principal part of $Z(\Phi,\phi_1,\phi_2,s)$ is not yet determined.
We briefly explain the difficulty.
There are two points where divergent terms appear in the calculation of the principal part using the Poisson summation formula.
One is the point where the series obtained by changing the order of the integral and the sum over singular elements diverges, and the other one is the point where the fundamental domain of the stabilizer has infinite volume.
In the case of the zeta functions with cusp forms, the second problem does not occur and we can determine the principal part by the method explained above.
On the other hand, when we consider the zeta function for the trivial representation, the second problem occurs for the regular element corresponding to $E=F\oplus F$.
As mentioned above, though this regular element is eliminated from the defining sum of the zeta functions, we should take them into account when using the Poisson summation formula.
Thus we can not apply the method of this paper to the zeta functions of the trivial representation.
The same problem also occurs for the space of binary quadratic forms.
They are handled by using the normalization of the multivariable zeta functions \cite{Shintani}, the Eisenstein series \cite{Yukiebook} and the truncation operators \cite{HW}.
Therefore, also in our spaces, it is expected that we need some additional techniques to deal with this problem and the argument becomes complicated.

\subsubsection{Contributions of singular orbits}

Recall the measures $\d z$ and $\d g_j$, see \S\ref{sec:3.2}. 
Set $Z(\A)^1\coloneqq \{z\in Z(\A)\mid |\det(z)|=1 \}$. 
A Haar measure $\d^1 z$ on $Z(\A)^1$ is defined by $\d^1 z\coloneqq \d^1 a$ for $z=a I_2$, $a\in\A^1$. 
We set
    \[
    \vol(Z_\rho)\coloneqq\left(\int_{Z(F)\bsl Z(\A)^1} \d^1 z \right)^2=c_F^2.
    \]
For each $1\le j\le 3$, we set $G_j(\A)^1\coloneqq \{g\in G_j(\A) \mid |\det(g)|=1 \}$. 
Then, we have an ismorphism
    \[
    G_j(\A)^1\times \R_{>0} \ni (g_j ,t) \mapsto t^{1/2}\, g_j \in G_j(\A). 
    \]
A Haar masure $\d^1 g_j$ on $G_j(\A)^1$ is obtained from the quotient of $\d g_j$ by $t^{-1}\d t$. 
A Haar measure $\d \Tilde{g}$ on $G(\A)^1\coloneqq G_1(\A)^1\times G_2(\A)^1\times G_3(\A)^1$ is defined by 
    \[
    \d \Tilde{g}\coloneqq\d^1 g_1 \, \d^1 g_2 \, \d^1 g_3 \quad \text{for} \quad \Tilde{g}=(g_1,g_2,g_3)\in G(\A)^1. 
    \]

We will use the notation $x_j$ which means the elements given in Propositions \ref{prop:sor1} and \ref{prop:sor2}.
By the Poisson summation formula on $V(F)$, when $\Re(s)$ is sufficiently large, we obtain
    \[
    Z(\Phi,\phi,s)=Z_+(\Phi,\phi,s)+Z_+^\vee(\widehat\Phi,\phi,2-s)+I(\Phi,\phi,s),
    \]
where
    \[
    I(\Phi,\phi,s)\coloneqq  \int_0^1 t^{2s}  I^1(\Phi_t,\phi)  \frac{\d t}{t}, \qquad \Phi_t(x)\coloneqq \Phi(t^{1/2}x),
    \]
\begin{multline*}
I^1(\Phi,\phi)\coloneqq \frac{1}{\vol(Z_\rho)}\int_{G(F)\bsl G(\A)^1}\, \d \tilde{g} \\
 \phi(g_1) \, \overline{\phi(g_2)}  \left(  \sum_{x\in V(F)\setminus V^0(F)} \widehat{\Phi}(x\cdot \rho^\vee(\tilde{g}))-  \sum_{x\in V(F)\setminus V^0(F)} \Phi(x\cdot\rho(\tilde{g}))    \right) ,
\end{multline*}
$(\Tilde{g}=(g_1,g_2,g_3)\in G(\A)^1)$.
For each subset $\mathcal{U}$ in $V(F)\setminus V^0(F)$, the integral
    \[
    \frac{1}{\vol(Z_\rho)}\int_{G(F)\bsl G(\A)^1} \phi(g_1) \, \overline{\phi(g_2)} \left(  \sum_{x\in \mathcal{U}} \widehat{\Phi}(x\cdot \rho^\vee(\tilde{g}))-  \sum_{x\in \mathcal{U}}  \Phi(x\cdot\rho(\tilde{g}))    \right) \, \d \tilde{g}
    \]
is called the contribution of $\mathcal{U}$ to $I^1(\Phi,\phi)$.
Notice that the contribution of $x_0=(0,0)$ to $I^1(\Phi,\phi)$ vanishes under the assumption of $\phi$, because
    \[
    \int_{G(F)\bsl G(\A)^1} \left| \phi(g_1) \, \overline{\phi(g_2)}    \Phi((0,0))    \right| \, \d \tilde{g} ,\;\; \int_{G(F)\bsl G(\A)^1} \left| \phi(g_1) \, \overline{\phi(g_2)}    \widehat{\Phi}((0,0))    \right| \, \d \tilde{g}
    \]
are convergent.
Therefore, to prove Theorem \ref{thm:global}, it is sufficient to determine the contributions of $\sqcup_{j=1}^5 x_j\cdot H(F)$ to $I^1(\Phi,\phi)$.

\begin{rem}
By the proof of Theorem \ref{thm:global}, the integral
    \[
    \int_0^\inf t^{2s} \left|  \int_{G(F)\bsl G(\A)^1} |\omega(\tilde{g})|^{s}\, \phi(g_1) \, \overline{\phi(g_2)}   \sum_{x\in V^0(F)} \Phi(t^{1/2}x\cdot \rho(\tilde{g}))    \, \d \tilde{g}\right| \frac{\d t}{t}
    \]
converges for $\Re(s)>3/2$. 
\end{rem}

\subsubsection{The case that $D$ is division}

Suppose that $D$ is division.
Let $B$ denote the Borel subgroup consisting of upper triangular matrices in $\GL_2$.
Set
    \[
    H_B\left( \begin{pmatrix}1&u \\ 0&1 \end{pmatrix} \begin{pmatrix}a&0 \\ 0&b \end{pmatrix}k\right)\coloneqq \log|a/b|  \qquad  (u\in\A, \;\; a,b\in\A^\times , \;\; k\in K).
    \]
\begin{lem}\label{lem:absdiv}
The contribution of $x_2\cdot H(F)$ to $I^1(\Phi,\phi)$ equals
    \[
    c_F \{ -Z^\GJ((\Phi_K)_1,f_{\phi},1)+Z^\GJ((\widehat{\Phi}_K)_1,f_{\phi}^\vee,1)  \}.
    \]
\end{lem}
\begin{proof}
Let $s\in\C$. 
First, we will prove that the integrals
    \begin{multline}\label{eq:20413e1}
    \int_{G(F)\bsl G(\A)^1} \big| \phi(g_1) \, \overline{\phi(g_2)} \big| \, \sum_{\gamma \in B(F)\bsl \GL_2(F)} \left|e^{-sH_B(\gamma g_3)} \right| \\
    \left|   \sum_{w \in D^\times} \Phi((0,w)\cdot \rho((1,1,\gamma)\tilde{g}))  -  \int_{D_\A} \Phi((0,x)\cdot \rho((1,1,\gamma) \tilde{g}))\, \d x    \right| \, \d \tilde{g},
    \end{multline}
    \begin{multline}\label{eq:20413e2}
    \int_{G(F)\bsl G(\A)^1}\big| \phi(g_1) \, \overline{\phi(g_2)} \big| \sum_{\gamma \in B(F)\bsl \GL_2(F)}   \left| e^{-sH_B(\gamma g_3)}  \right| \\
    \left| \sum_{w \in D^\times} \widehat{\Phi}((w,0)\rho^\vee((1,1,\gamma)\tilde{g}))   -   \int_{D_\A} \widehat{\Phi}((x,0)\rho^\vee((1,1,\gamma)\tilde{g}))\, \d x    \right| \, \d \tilde{g}
    \end{multline}
are convergent for $-1<\Re(s)<1$.
Set $B(\A)^1\coloneqq \{ \left(\begin{smallmatrix}a&b \\ 0&c\end{smallmatrix}\right) \in B(\A) \mid |a|=|c|=1 \}$. 
By using \eqref{eq:normalization dg_3}, we obtain a Haar measure $\d q \coloneqq \d^1 a \, \d^1 c \, \d b$ on $B(\A)^1$ for $q=\left(\begin{smallmatrix}a&b \\ 0&c\end{smallmatrix}\right)$. 
Since $\d^1 g_3=\d q \d t_2 \d k $ holds for $g_3=q \, \diag(t_2^{-1/2},t_2^{1/2}) \, k$, the integral \eqref{eq:20413e1} is transformed into 
    \begin{multline*}
    \int_{G_1(F)\bsl G_1(\A)^1}\d^1 g_1 \, \int_{G_2(F)\bsl G_2(\A)^1}\d^1 g_2 \\
    \int_{B(F)\bsl B(\A)^1 }\d q \, \int_0^\inf  \frac{\d t_2}{t_2} \, \int_K \d k  \,  \big| \phi(g_1) \, \overline{\phi(g_2)} \big| \, t_2^{s+1} \\ 
    \left|\sum_{w \in D^\times} \Phi((0, t_2^{1/2} c g_1^{-1} w g_2)\cdot \rho((1,1,k)))  -  \int_{D_\A} \Phi((0,t_2^{1/2}c g_1^{-1}xg_2)\cdot \rho((1,1,k)))\, \d x    \right|.
    \end{multline*}
We set
    \[
    |\Phi|_B(x)\coloneqq \int_K \int_{B(F)\bsl B(\A)^1} |\Phi(  (0,x) \, \rho(1,1,qk)) | \,  \d q \, \d k , 
    \]
    \[
    |\Phi|_B^{(2)}(x)\coloneqq  \int_K \int_{B(F)\bsl B(\A)^1} \left| \int_{D_\A} \Phi((0,z)\, \rho(1,1,qk))\, \psi_F(\Tr(xz)) \, \d z \right|  \, \d q \, \d k .
    \]
Since $B(F)\bsl B(\A)^1K$ is compact, $|\Phi|_B$ and $|\Phi|_B^{(2)}(x)$ are bounded by certain Schwartz functions on $D_\A$.
Dividing $\int_0^\inf \frac{\d t_2}{t_2}$ into $\int_1^\inf \frac{\d t_2}{t_2}+\int_0^1 \frac{\d t_2}{t_2}$ and applying the Poisson summation formula to $\sum_{w \in D^\times}$ inside $\int_0^1 \frac{\d t_2}{t_2}$, we deduce
    \begin{multline*}
    \eqref{eq:20413e1}\ll \int_{G_2(\A)^1}\d^1 g \, \int_1^\inf\frac{\d t}{t} \, f_{|\phi|}(g)  \, t^{s+1} \,   |\Phi|_B(t^{1/2}g)  \\
    + \int_{G_2(F)\bsl G_2(\A)^1}\d^1 g \, \int_1^\inf\frac{\d t}{t} \,   f_{|\phi|}(g) \, t^{s-1} \, |\Phi|_B^{(2)}(0) \\
    +\int_{G_2(\A)^1}\d^1 g \, \int_0^1 \frac{\d t}{t} \,  f_{|\phi|}(g)  \, t^{s-1}  |\Phi|_B^{(2)}(t^{-1/2} \; {}^t\!g^{-1} ) \\
    + \int_{G_2(F)\bsl G_2(\A)^1}\d^1 g  \, \int_0^1 \frac{\d t}{t} \,   f_{|\phi|}(g) \, t^{s+1} \, |\Phi|_B(0).
    \end{multline*}
Hence, \eqref{eq:20413e1} converges for the range $-1<\Re(s)<1$.
The convergence of \eqref{eq:20413e2} can be proved similarly.
By using the convergence of \eqref{eq:20413e1} and \eqref{eq:20413e2}, the integral 
    \begin{multline*}
    \mathscr{Y}_1(s)\coloneqq \int_{G(F)\bsl G(\A)^1}  \phi(g_1) \, \overline{\phi(g_2)}  \, \sum_{\gamma \in B(F)\bsl \GL_2(F)} e^{-sH_B(\gamma g_3)} \\
    \left\{  -\sum_{w \in D^\times} \Phi((0,w)\cdot \rho((1,1,\gamma)h))  + \sum_{w \in D^\times} \widehat{\Phi}((w,0)\rho^\vee((1,1,\gamma)h))   \right\} \, \d \tilde{g} 
    \end{multline*}
is absolutely convergent in the range $-1<\Re(s)<1$.
Note that we have used the cancellation
    \[
    \int_{D_\A} \Phi((0,x)\cdot \rho(\tilde{g}))\, \d x- \int_{D_\A} \widehat{\Phi}((x,0)\cdot \rho^\vee(\tilde{g}))\, \d x =0    \qquad (\tilde{g}\in G(\A)^1).
    \]
Furthermore, for $-1<\Re(s)<1$ we get $\mathscr{Y}_1(s)=\mathscr{Y}_2(s)$ where
    \begin{multline*}
    \mathscr{Y}_2(s)\coloneqq c_F^2\int_0^\inf t^{s+1} \,  \int_{ G_2(F)\bsl G_2(\A)^1} \int_{ G_1(F)\bsl G_1(\A)^1} \phi(g_1) \, \overline{\phi(g_2)} \\
    \left\{  - \sum_{w\in D^\times} \Phi_K((0,t^{1/2}g_1^{-1}w g_2)) +   \sum_{w\in D^\times} \widehat\Phi_K((t^{1/2}g_2^{-1}w g_1,0)) \right\} \, \d g_1 \, \d g_2 \, \frac{\d t}{t}.
    \end{multline*}
Note that we can change the order of $\int_{G(F)\bsl G(\A)^1}$ and $\sum_{\gamma \in B(F)\bsl \GL_2(F)}$ for $-1<\Re(s)<1$ since \eqref{eq:20413e1} and \eqref{eq:20413e2} converge. 
By the Poisson summation formula and the same argument as above, we can prove that $\mathscr{Y}_2(s)$ is analytically continued to the whole $s$-plane.
In addition, we obtain
    \[
    \mathscr{Y}_2(s)=c_F^3\{-Z^\GJ((\Phi_K)_1,f_{\phi},s+1)+Z^\GJ((\widehat{\Phi}_K)_1,f_{\phi}^\vee,s+1)\}
    \]
in the range $\Re(s)>1$. 
Since the Godement-Jacquet integral is analytically continued to the whole $s$-plane, the contribution of $x_2\cdot H(F)$ to $I^1(\Phi,\phi)$ equals 
    \[
    c_F^{-2}\mathscr{Y}_1(0)=c_F^{-2}\mathscr{Y}_2(0)=c_F \{ -Z^\GJ((\Phi_K)_1,f_{\phi},1)+Z^\GJ((\widehat{\Phi}_K)_1,f_{\phi}^\vee,1)\}. 
    \]
Thus, this completes the proof.
\end{proof}
For the case that $D$ is division, we obtain Theorem \ref{thm:global} by Proposition \ref{prop:sor2} and Lemma \ref{lem:absdiv}.

\subsubsection{The case $D=\M_2(F)$}

Assume that $D=\M_2(F)$.
\begin{lem}\label{lem:conv0816}
The integral
    \[
    \int_{B(F)\bsl \GL_2(\A)^1} |\phi(g)| \, e^{(1+\epsilon)\, H_B(g)} \, \d^1 g
    \]
is convergent for any $\epsilon>0$, where $\GL_2(\A)^1\coloneqq\{g\in\GL_2(\A)\mid |\det(g)|=1\}$.
\end{lem}
\begin{proof}

Recall the measure $\d q$ on $B(\A)^1$, see the proof of Lemma \ref{lem:absdiv}, and we have $\d^1 g=\d q t^{-2}\d t \d k $ for $g=q \, \diag(t^{1/2},t^{-1/2}) \, k$, 
Hence, the integral equals
    \[
    \int_{B(F)\bsl B(\A)^1} \int_0^\infty \int_K |\phi(q \, \diag(t^{1/2},t^{-1/2}) \, k)| \, t^{-1+\epsilon} \, \d k \, \d t \, \d q .
    \]
Since $\phi$ is cuspidal, it is known that $\phi$ is rapidly decreasing on a Siegel set 
    \[
    \{ pak   \mid p\in \omega, \;\; a=\diag(e^t,e^{-t}), \;\; k\in K, \;\; t>T_0   \},
    \]
for any $T_0\in\R$ and any compact set $\omega$ in $B(\A)$. 
Hence, we obtain the convergence of the above integral, because $B(F)\bsl B(\A)^1$ is compact. 
\end{proof}
\begin{lem}\label{lem:contx2}
The contribution of $x_1\cdot H(F)\sqcup x_2\cdot H(F)$ to $I^1(\Phi,\phi)$ is
    \[
    c_F \{ -Z^\GJ((\Phi_K)_1,f_{\phi},1)+Z^\GJ((\widehat{\Phi}_K)_1,f_{\phi}^\vee,1) \}.
    \]
\end{lem}
\begin{proof}
One can apply the same argument as in the proof of Lemma \ref{lem:absdiv}.
Hence, 
    \begin{multline*}
    \int_{G(F)\bsl G(\A)^1} \big| \phi(g_1) \, \overline{\phi(g_2)} \big| \, \sum_{\gamma \in B(F)\bsl \GL_2(F)}  \left| e^{-sH_B(\gamma g_3)} \right| \\
    \left|   \sum_{ \substack{ w \in \M_2(F)\\ \mathrm{rank}(w)>0} } \Phi((0,w)\cdot \rho((1,1,\gamma)\tilde{g}))  -  \int_{\M_2(\A)} \Phi((0,x)\cdot \rho((1,1,\gamma) \tilde{g}))\, \d x    \right| \, \d \tilde{g} ,
    \end{multline*}
    \begin{multline*}
    \int_{G(F)\bsl G(\A)^1}\big| \phi(g_1) \, \overline{\phi(g_2)}  \big| \sum_{\gamma \in B(F)\bsl \GL_2(F)}   \left| e^{-sH_B(\gamma g_3)}  \right| \\
    \left| \sum_{ \substack{w \in \M_2(F)\\ \mathrm{rank}(w)>0}} \widehat{\Phi}((w,0)\cdot\rho^\vee((1,1,\gamma)\tilde{g}))   -   \int_{\M_2(\A)} \widehat{\Phi}((x,0)\cdot \rho^\vee((1,1,\gamma)\tilde{g}))\, \d x    \right| \, \d \tilde{g}
    \end{multline*}
are convergent for $-1<\Re(s)<1$. 
In the proof, we have to divide the integration domain into $H_B(\gamma g_3)>0$ and $H_B(\gamma g_3)<0$ in order to use the Poisson summation formula, and take care of the convergence of the sum over
    \[
    \mathscr{A}_1\coloneqq  \{w\in \M_2(F) \mid \mathrm{rank}(w)=1\}.
    \]
We prove, as one example, the convergence of the integral for $H_B(\gamma g_3)<0$ of the sum over $\mathscr{A}_1$, which is
    \begin{multline}\label{eq:intA1}
    \int_{\substack{G(F)\bsl G(\A)^1 \\ H_B(\gamma g_3)<0} } \big| \phi(g_1) \, \overline{\phi(g_2)} \big| \, \sum_{\gamma \in B(F)\bsl \GL_2(F)}  \left| e^{-sH_B(\gamma g_3)} \right| \\
    \left|   \sum_{ w \in \mathscr{A}_1 } \Phi((0,w)\cdot \rho((1,1,\gamma)\tilde{g}))     \right| \, \d \tilde{g} .
    \end{multline}
In this case, it is unnecessary to apply the Poisson summation formula. 
Since
    \[
    \mathscr{A}_1=\bigsqcup_{ \gamma_1\in B(F)\bsl \GL_2(F)} \bigsqcup_{\gamma_2 \in B(F)\bsl\GL_2(F) } \gamma_1^{-1} \left\{ 
        \begin{pmatrix}
        0&u \\ 
        0&0
        \end{pmatrix} 
    \; \middle| u\in F^\times\right\}\gamma_2,
    \]
the integral \eqref{eq:intA1} is bounded by the integral
    \begin{multline*}
    \mathscr{A}_1(s)\coloneqq\int_{F^\times\bsl\A^1} \d^1 \alpha \, \int_1^\inf  \frac{\d t}{t} \,  \int_{B(F)\bsl \GL_2(\A)^1} \d g_1 \,  \int_{B(F)\bsl \GL_2(\A)^1} \d g_2 \\ 
    \left| \phi(g_1) \, \overline{\phi(g_2)} \right| \, t^{1+\Re(s)} e^{-H_B(g_1)}e^{-H_B(g_2)}  \\
    \sum_{u\in F^\times}\left| \Phi((0,
        \begin{pmatrix} 
        0& \alpha t^{1/2} e^{-H_B(g_1)/2}e^{-H_B(g_2)/2} u \\ 
        0&0 
        \end{pmatrix})
    \cdot \rho(k_1,k_2,k_3)) \right| 
    \end{multline*}
up to constant, where $g_j=b_j k_j\in G(\A)^1$, $b_j\in B(F)\bsl B(\A)$, $k_j\in K$. 
By the condition $t>1$, we have $\mathscr{A}_1(s)<\mathscr{A}_1(s+2)$.
Then, we see that $\mathscr{A}_1(s+2)$ is convergent by Lemma \ref{lem:conv0816} if we change the variable $t$ by $te^{H_B(g_1)}e^{H_B(g_2)}$.  
Thus, the integral \eqref{eq:intA1} also converges. 
We can prove the convergence of other terms by the same argument.

By the convergence of the above integrals, the contribution of $x_1\cdot H(F)\sqcup x_2\cdot H(F)$ to $I^1(\Phi,\phi)$ is equal to $c_F^{-2}\mathscr{Y}_3(0)$, where we set
    \begin{multline*}
    \mathscr{Y}_3(s)\coloneqq c_F\int_0^\inf  \frac{\d t}{t}  \int_{G_1(F)\bsl G_1(\A)^1}\d^1 g_1 \int_{G_2(F)\bsl G_2(\A)^1} \d^1 g_2 \, t^{s+1} \, \phi(g_1) \, \overline{\phi(g_2)} \\
    \left\{  - \sum_{\substack{w \in \M_2(F)\\ \mathrm{rank}(w)>0}} \Phi_K((0,t^{1/2}g_1^{-1}w g_2))  +  \sum_{ \substack{w \in \M_2(F)\\ \mathrm{rank}(w)>0}} \widehat\Phi_K((t^{1/2}g_2^{-1}w g_1,0)) \right\} ,
    \end{multline*}
and $\mathscr{Y}_3(s)$ is convergent in the range $-1<\Re(s)<1$.
Furthermore, we can prove that $\mathscr{Y}_3(s)$ is analytically continued to the whole $s$-plane by using the Poisson summation formula.
Therefore, by the same argument as in the proof of Lemma \ref{lem:absdiv} we find that the contribution of $x_2\cdot H(F)$ equals $c_F \{ -Z^\GJ((\Phi_K)_1,f_{\phi},1)+Z^\GJ((\widehat{\Phi}_K)_1,f_{\phi}^\vee,1) \}$.
Consider the contribution of $x_1\cdot H(F)$. 
The partial sum of $\mathrm{rank}(w)=1$ in the first term inside of the integral $\mathscr{Y}_3(s)$ can be transformed into
    \begin{multline*}
    - 2c_F \int_{B(F)\bsl \GL_2(\A)^1} \phi (g_1)e^{sH_B(g_1)} \d g_1 \,\int_{B(F)\bsl \GL_2(\A)^1} \overline{\phi (g_2)} e^{sH_B(g_2)} \d g_2 \\ 
    \int_{F^\times\bsl \A^\times} |a|^{2+2s} \sum_{x\in F^\times}\Phi_K(0,
        \begin{pmatrix} 
        0& ax \\ 
        0&0 
        \end{pmatrix}) 
    \, \d^\times a      
    \end{multline*}
for $\Re(s)>1$ by Lemma \ref{lem:conv0816}, and hence it vanishes by the cuspidality of $\phi$. 
Since the second one vanishes similarly, the contribution of $x_1\cdot H(F)$ is zero. 
This completes the proof.
\end{proof}

\begin{lem}\label{lem:contx34}
The contribution of $x_3\cdot H(F)\sqcup x_4\cdot H(F) $ to $I^1(\Phi,\phi)$ is zero.
\end{lem}
\begin{proof}
For $x=\left( \left(
    \begin{smallmatrix}
    x_{11}&x_{12} \\ 
    x_{21}&x_{22} 
    \end{smallmatrix}
\right),\left(
    \begin{smallmatrix}
    y_{11}&y_{12} \\ 
    y_{21}&y_{22} 
    \end{smallmatrix}
\right) \right)\in V(F)$, we set
    \[
    \sF_3(x)\coloneqq \left(
        \begin{pmatrix} 
        x_{11}&x_{12} \\ 
        y_{11}&y_{12} 
        \end{pmatrix},
        \begin{pmatrix}
        x_{21}&x_{22} \\ 
        y_{21}&y_{22} 
        \end{pmatrix}
    \right) ,\quad \sF_4(x)\coloneqq \left(
        \begin{pmatrix} 
        x_{11}&y_{11} \\ 
        x_{21}&y_{21} 
        \end{pmatrix},
        \begin{pmatrix}
        x_{12}&y_{12} \\ 
        x_{22}&y_{22} 
        \end{pmatrix} 
    \right) .
    \]
Then, we have $\sF_j(x_2\cdot H(F))=x_j\cdot H(F)$ $(j=3,4)$. 
Suppose that $\mathscr{B}(w,\gamma,\tilde{g})$ is one of the following functions:
    \[
    \Phi(\sF_3(w,0)\cdot \rho((\gamma,1,1) \tilde{g}))   , \quad   \Phi(\sF_4(0,w)\cdot \rho((1,\gamma,1)\tilde{g})) ,
    \]
    \[
    \widehat\Phi( \sF_3(w,0)\cdot \rho^\vee((1,\gamma,1)\tilde{g})) , \quad \widehat\Phi( \sF_4(0,w)\cdot \rho^\vee((\gamma,1,1)\tilde{g}))  .
    \]
Then, one can prove that
    \[
    \int_{G(F)\bsl G(\A)^1} \left| \phi(g_1) \, \overline{\phi(g_2)} \right| \, \sum_{\gamma \in B(F)\bsl \GL_2(F)} \left|  \sum_{w \in D^\times} \mathscr{B}(w,\gamma,\tilde{g})  \right| \, \d \tilde{g}
    \]
is convergent by dividing the integration domain into $H_B(\gamma g_j)>0$ and $H_B(\gamma g_j)<0$, and using the fact that $\phi_j(g_j) \, e^{r H_B(g_j)}$ is bounded on $\{ g_j\in B(F)\bsl \GL_2(\A)^1\mid H_B(g_j)>0 \}$ for any $r>0$.
Hence, the contribution of $x_3\cdot H(F)\sqcup x_4\cdot H(F) $ vanishes by the cuspidality of $\phi$.
\end{proof}

\begin{lem}\label{lem:contx5}
The contribution of $x_5\cdot H(F)$ to $I^1(\Phi,\phi)$ is zero.
\end{lem}
\begin{proof}
Set 
    \[
    \mathscr{A}_2\coloneqq \{  (u_1E_{12},u_2E_{11}+u_3E_{22}+cE_{12}) \mid u_1,u_2,u_3\in F^\times, \;\; c\in F \},
    \]
    \[
    \mathscr{A}_3\coloneqq \{  (u_2E_{11}+u_3E_{22}+cE_{12},u_1E_{12}) \mid u_1,u_2,u_3\in F^\times, \;\; c\in F \},
    \]
    \begin{multline*}
    \mathscr{Y}_4(s)\coloneqq -\int_{G(F)\bsl G(\A)^1}  \phi(g_1) \, \overline{\phi(g_2)} \\
    \sum_{\gamma\in (B(F)\bsl \GL_2(F))^3} e^{-s\sum_{j=1}^3H_B(\gamma_j g_j)} \sum_{w\in\mathscr{A}_2} \Phi(w\cdot \rho(\gamma \tilde{g}))  \, \d \tilde{g} ,   
    \end{multline*}
    \begin{multline*}
    \mathscr{Y}_5(s)\coloneqq \int_{G(F)\bsl G(\A)^1}  \phi(g_1) \, \overline{\phi(g_2)} \\
    \sum_{\gamma\in (B(F)\bsl \GL_2(F))^3} e^{-s\sum_{j=1}^3H_B(\gamma_j g_j)} \sum_{w\in\mathscr{A}_3} \widehat{\Phi}(w\cdot\rho^\vee(\gamma \tilde{g}))   \, \d \tilde{g} , 
    \end{multline*}
where $\gamma=(\gamma_1,\gamma_2,\gamma_3)$.
The contribution of $x_5\cdot H(F)$ equals $c_F^{-2}(\mathscr{Y}_4(0)+\mathscr{Y}_5(0))$ if they are convergent. 
The function $\mathscr{Y}_4(s)$ is transformed into
    \begin{multline*}
    \int_{B(F)\bsl B(\A)^1 }\d b_1 \int_{B(F)\bsl B(\A)^1 }\d b_2 \, \int_\bK \d k  \, \int_0^\inf \frac{\d t_1}{t_1} \, \int_0^\inf \frac{\d t_2}{t_2} \, \int_0^\inf \frac{\d t_3}{t_3} \\
    \int_\A \d c    \, \int_{F^\times\bsl \A^1} \d^1 a_1  \, \int_{F^\times\bsl \A^1} \d^1 a_2    \,  (t_1t_2t_3)^{-1-2s} \, \phi(b_1\mathfrak{t}_1k_1) \, \overline{\phi(b_2\mathfrak{t}_2k_2)}\, \\
    \sum_{u_1,u_2,u_3\in F^\times} \Phi(( t_1^{-1} t_2^{-1}t_3 u_1 E_{12},  t_1^{-1}t_2t_3^{-1}u_2E_{11} +  t_1t_2^{-1}t_3^{-1}u_3E_{22}+cE_{12}) \\
    \cdot \rho((\mathfrak{t}_1^{-1} b_1\mathfrak{t}_1,\mathfrak{t}_2^{-1} b_2\mathfrak{t}_2,\mathfrak{a}_3)k) 
    \end{multline*}
up to constant, where 
    \begin{gather*}
    k=(k_1,k_2,k_3)\in \bK=K\times K\times K,  \quad \mathfrak{t}_1=\diag(t_1^{1/2},t_1^{-1/2}), \\
    \mathfrak{t}_2=\diag(t_2^{1/2},t_2^{-1/2}), \quad \mathfrak{a}_3=\diag(a_1,a_2). 
    \end{gather*}
Hence, one can prove that $\mathscr{Y}_4(s)$ is absolutely convergent for any $s\in\C$ and entire on $\C$ by using Lemma \ref{lem:conv0816} and dividing the integration domain into $t_1t_2t_3<1$ and $t_1t_2t_3>1$.
Since $\mathscr{Y}_5(s)$ also satisfies a similar transformation, it follows that $\mathscr{Y}_5(s)$ is absolutely convergent for any $s\in\C$ and entire on $\C$.
Furthermore, for $\Re(s)>0$, $ \mathscr{Y}_4(s)$ is equal to
    \begin{multline*}
    \int_{ \A} \d c  \,  \int_{\A^\times}\d^\times a_1  \,  \int_{\A^\times}\d^\times a_2  \,  \int_{\A^\times}\d^\times a_3 \, \int_{N(F)\bsl N(\A)}\d n_1 \, \int_{N(F)\bsl N(\A)}\d n_2 \, \int_\bK \d k \\ 
    |a_1a_2a_3|^{2s+1} \phi(n_1 \diag(a_1^{-1},1)k_1) \, \overline{\phi(n_2\diag(a_1^{-1}a_3^{-1}a_2,1)k_2)} \\ 
    \Phi((a_1E_{12} ,a_2E_{11}+a_3E_{22}+cE_{12} ) \cdot \rho(k))  
    \end{multline*}
up to constant, where $N$ denotes the unipotent radical of $B$, $k=(k_1,k_2,k_3)\in\bK$. 
In addition, $\mathscr{Y}_5(s)$ also satisfies a similar equality.
Hence, $\mathscr{Y}_4(s)$ and $\mathscr{Y}_5(s)$ are identically zero by the cuspidality of $\phi$.
\end{proof}

By Proposition \ref{prop:sor1} and Lemmas \ref{lem:contx2}, \ref{lem:contx34}, and \ref{lem:contx5}, we obtain Theorem \ref{thm:global} for the case $D=\M_2(F)$.

\section{Explicit formula for global zeta functions}\label{sec:explicit}

Keep the notation of \S\ref{sec:notation}.
Recall that $F$ is a number field.  
In this section, Haar measures are the Tamagawa measures. 
Note that the normalization of measures is the same as in \S\ref{sec:1}, different from that in \S\ref{sec:globalzeta424}.

\subsection{On Saito's explicit formula}
In \cites{Saito1,Saito2}, Saito obtained an explicit formula for general prehomogeneous zeta functions after the joint work \cite{IS} together with Ibukiyama.
His method is based on the idea of endoscopy theory of automorphic representations.
In this section, we derive an explicit formula for our zeta integral $Z(\Phi,\phi,s)$ by applying Saito's method.

\subsubsection{Motivation for an explicit formula}
In \cite{SW}, we use the explicit formula for our zeta function to obtain a mean value theorem for the central value of the $L$-function. 
Especially, in the course of the study in \cite{SW}, we need the non-negativity of the coefficients of our zeta function, which is proved only after we have obtained an explicit formula. 
See \cite{SW}*{Lemma 17}.
Note that for the ordinary perhomogeneous zeta functions, the non-negativity of the coefficients is rather obvious from its construction, but this is not the case for our zeta function.

\subsubsection{Reason for using Saito's formula}
Before \cites{Saito1,Saito2}, several explicit formulas for prehomogeneous zeta functions were studied in various ways.
Saito obtained a general and uniform explicit formula.
In addition to that, Saito's formula has the following advantages (see also \cite{HW}*{p.7}):
\begin{itemize}
\item The relation between adelic orbits and rational orbits becomes clear.
\item The relation between the global measure on the adelic group and the local measures on the vector space becomes clear.
\item Absolute convergence of the zeta integral follows in general from Saito's formula \cite{Saito2} (see also Lemma \ref{lem:absconv}). 
\end{itemize}

\subsection{Orbit decomposition}

Let $D$ be a quaternoin algebra over the number field $F$. 
Recall the notations $X(D)$, $\delta_E$, and so on in \S\ref{sec:quaternoin}. 
Suppose $E\in X(D)$.
For a place $v\in\Sigma$, $D_v=D\otimes_FF_v$ is a quaternion algebra over $F_v$ and $E_v=E\otimes_FF_v$ is an $F_v$-subalgebra of $D_v$.
Unless otherwise mentioned, we assume $\delta_{E_v}=\delta_E$ and $d_{E_v}=d_E$ under the natural embeddings $ E\hookrightarrow E_v$ and $D\hookrightarrow D_v$.


For each $E\in X(D)$, we set
    \[
    \Pi_{x_E}\coloneqq H_{x_E}^0\bs H_{x_E} ,
    \]
and by Proposition \ref{prop:orbits} we obtain
\begin{equation}\label{eq:s2}
 \# \Pi_{x_E}(F)=2, \quad  H_{x_E}^0\cong (R_{E/F}(\mathbb{G}_m)/\mathbb{G}_m)\times (R_{E/F}(\mathbb{G}_m)/\mathbb{G}_m),
\end{equation}
and
\begin{equation}\label{eq:s1}
V^0(F)/H(F)=\bigsqcup_{E\in X(D)} x_E\cdot H(F).
\end{equation}
We note on $R_{E/F}(\mathbb{G}_m)/\mathbb{G}_m\cong R_{E/F}^{(1)}(\mathbb{G}_m)$. 

In what follows, the data necessary for the Saito's explicit formula \cite{Saito1} are explained in the same way as for \cite{HW}*{Section 4.2}. 
Fix an element $E_0\in X(D)$. 
For a regular element $x_0\coloneqq x_{E_0}\in V^0(F)$, the map $\varphi_{x_0}$ given in \cite{Saito1}*{p.590} is the composition of the natural mappings as:
    \[
    V^0(F)\longrightarrow V^0(F)/H(F) \longrightarrow H^1(F,H_{x_0})\longrightarrow H^1(F,\Pi_{x_0}).
    \]
Here, the second map is induced from the exact sequence $1\to H_{x_0}\to H \to V^0 \to 1$ and the third one is induced from the exact sequence $1\to H_{x_0}^0\to H_{x_0} \to \Pi_{x_0} \to 1$. 
For each $E\in X(D)$ we set $F'\coloneqq F(\sqrt{d_E/d_{E_0}})$, and have
    \[
    F'(\delta_{E_0})=F'(\delta_E) \quad \text{in $D\otimes F'$},
    \]
Hence, by the Skolem-Noether theorem, for every $E\in X(D)$, there exists an element $g_{E}\in G_1(F(\sqrt{d_E/d_{E_0}}))$ $(G_1=\PGL_1(D))$ such that
    \[
    \sqrt{d_E/d_{E_0}}\, g_{E}^{-1}\delta_{E_0} g_{E}= \delta_E ,\quad \det(g_E)=\sqrt{d_E/d_{E_0}}.
    \]
For each $E\in X(D)$, a $1$-cocycle $z_E$ is defined by 
    \[
    z_E(\sigma)\coloneq h_E^{-1}\, (h_E)^\sigma \qquad (\sigma\in \mathrm{Gal}(\overline{F}/F)),
    \]
where $h_E\coloneqq (g_E,g_E,\diag(1,\sqrt{d_E/d_{E_0}}))\in H$.

\begin{lem}\label{lem:2}
The map $X(D)\ni E\mapsto [z_E] \in H^1(F,\Pi_{x_0})$ is injective, where $[z_E]$ denotes the class of $z_E$ in $H^1(F,\Pi_{x_0})$.  
\end{lem}
\begin{proof}
If $[z_E]$ is trivial, then we get $E=E_0$ by checking the third factor of $h_E$. 
\end{proof}

For $E\in X(D)$, we write the $H(F)$-orbit of $x_E$ as
    \[
    V_E(F)\coloneqq x_E\cdot H(F)= \{ x\in V(F) \mid  P(x)\in d_E(F^\times)^2 \}. 
    \]
\begin{lem}
For any $x\in V^0(F)$, we have $x\in V_E(F)$ if and only if $\varphi_{x_0}(x)=\varphi_{x_0}(x_E)$. 
\end{lem}
\begin{proof}
This assertion follows from Lemma \ref{lem:2}.     
\end{proof}
By this lemma, $V_E(F)$ agrees with the subset $X^\mathrm{ss}(F,\tilde{y})$ $(y=x_E)$ given in the paper \cite{Saito1}. 
In particular, the key point in using Saito's formula \cite{Saito1} is that the invariants $\tau(H_x^0)$, $A(H_x^0)$, $\mathrm{ker}^1(H_x^0)$ of $H_x^0$, which we explain later, depends only on the $H(F)$-orbit of $x$.

For a connected algebraic group $M$ over $F$, denote by $\tau(M)$ the Tamagawa number of $M$. 
Write $A(M)$ for the torsion group of $M$ (the notation of Boronoi and Kottowitz. See also \cite{Saito1}*{p.594}). 
Write $\mathrm{ker}^1(M)$ for the kernel of the Hasse map $H^1(F,M)\to \prod_v H^1(F_v,M)$.  
When $\mathrm{ker}^1(M)=1$, we say that $M$ satisfies the Hasse principle.
In that case, the Hasse map becomes injective. 
\begin{lem}\label{lem:3}
Let $E\in X(D)$. 
\begin{itemize}
\item Suppose $E\not\cong F\times F$. In this case, we have $\tau(R_{E/F}^{(1)}\mathbb{G}_m)=2$, hence $\tau(H^0_{x_E})=4$. 
By a direct calculation, we have $A(H_{x_E}^0)\cong \Z/2\Z\oplus \Z/2\Z$. 
By the relation $\tau(H_{x_E}^0)\, \mathrm{ker}^1(H_{x_E}^0)=A(H_{x_E}^0)$, we obtain $\mathrm{ker}^1(H_{x_E}^0)=1$. 
\item Suppose $E\cong F\times F$. Since $H_{x_E}^0\cong \mathbb{G}_m\times \mathbb{G}_m$, we see $\tau(H_{x_E}^0)=1$, $A(H_{x_E}^0)=1$, and $\mathrm{ker}^1(H_{x_E}^0)=1$.
\end{itemize}
\end{lem}
\begin{proof}
See \cite{Saito1}*{p.594} and \cite{HW}*{Section 4.2} for the details of computations.
\end{proof}

The following is a well-known fact.
For the sake of the readers, we will write the proof based on the book \cite{PR}.
\begin{lem}\label{lem:4}
The Hasse principle for $H$ holds.  
\end{lem}
\begin{proof} 
It is known that $H^1(F,G)=1$, see \cite{PR}*{Lemma 2.8 in p.80}. 
From the exact sequence $1\to Z_\rho\to G\to H\to 1$ we deduce
    \[
    H^1(F,G)=1 \longrightarrow H^1(F,H) \longrightarrow H^2(F,Z_\rho).
    \]
Hence we obtain an injection $f\colon H^1(F,H) \to H^2(F,Z_\rho)$. 
Furthermore, the commutativity of the diagram
    \[
    \begin{CD}
     1 @>>>  H^1(F,H)  @>f>>  H^2(F,Z_\rho)  \\
    @.     @VV{\tilde{f}}V  @VV{g}V   \\
     1 @>>>  \prod_v H^1(F_v,H) @>>>  \prod_v H^2(F_v,Z_\rho) 
     \end{CD}
     \]
shows that $\tilde{f}$ is also injective, since $g$ is injective, see \cite{PR}*{Theorem 1.12 in p.38}. 
\end{proof}

\begin{lem}\label{lem:5}
Let $E\in X(D)$ and set $Y_{x_E}\coloneqq H_{x_E}^0\bsl H$. 
Consider the natural projection
    \[
    \eta_{x_E} \colon Y_{x_E}(F)/ H(F) \ni H_{x_E}^0 h H(F) \mapsto H_{x_E}^0 h H(\A) \in Y_{x_E}(F)H(\A)/ H(\A).
    \]
Then, we have $|\eta_{x_E}^{-1}(\eta_{x_E}(H_{x_E}^0 h H(F)))|=\mathrm{ker}^1(H_{x_E}^0)=1$ for arbitrary $h\in Y_{x_E}(F)$. 
\end{lem}
\begin{proof}
This follows from Lemmas \ref{lem:3} and \ref{lem:4} and \cite{Saito1}*{Proposition 1.4}.
\end{proof}

\begin{lem}\label{lem:6}
For any $E\in X(D)$, we have $Y_{x_E}(F)H(\A)=Y_{x_E}(\A)$. 
\end{lem}
\begin{proof}
By a direct calculation, one can prove that the natural map $\iota_{x_E,A}^0$ from $A(H_{x_E}^0)$ to $A(H)$ is injective, see \cite{HW}*{Section 4.2}. 
Hence, we obatin the assertion from \cite{Saito1}*{Lemma 1.5}. 
\end{proof}

\subsection{Measures on algebraic groups}\label{sec:measure1}

For an algebraic group $G$ over $F$, we write $G(F_\infty)$, $G(\A)$ and $G(\A_\fin)$ by $G_\infty$, $G_\A$ and $G_{\A_\fin}$, respectively.
Let $K=\prod_v K_v$ be a maximal compact subgroup  of $\GL_2(\A)$,  where
    \[
    K_v=
        \begin{cases}
        \GL_2(\fo_v) & \text{ if $v$ is finite,} \\
        \O(2) & \text{ if $v$ is real,} \\
        \U(2) & \text{ if $v$ is complex}.
        \end{cases}
    \]
Let $\d k_v$ be the Haar measure on $K_v$ normalized so that $\vol(K_v)=1$ and define the Haar measure $\d k$ on $K$ by $\d k=\prod_v\d k_v$.

Let $Z$ denote the center of $D^\times$ and $\d z$ the Haar measure on $Z(\A)=\A^\times$ given by $\d z=\frac{\d t}{t} \, \d^1 a$, where $z=t^{\frac12}a I_2$ with $t\in\R_{>0}$ and $a\in\A^1$.
Note that $\d z=2\d^\times x$, where $\d^\times x$ is the measure on $\A^\times$ chosen in \S\ref{sec:notation}.
Let $\d g$ be the Tamagawa measure on $D_\A^\times$. 
The algebraic group $D^\times/ F^\times$ over $F$ will be denoted by $PD^\times$.
Then we have $\vol(PD^\times\bs PD^\times_\A)=c_F^{-1}$ with respect to the quotient measure $\frac{\d g}{\d z}$.
Let $\langle \cdot,\cdot\rangle$ be the inner product on $L^2(PD^\times\bs PD^\times_\A)$ given by
    \begin{equation}\label{eq:innprod}
    \langle \varphi,\varphi'\rangle=\int_{PD^\times\bs PD^\times_\A} 
    \varphi(g)\, \overline{\varphi'(g)} \frac{\d g}{\d z}, 
    \hspace{25pt} \varphi, \varphi'\in L^2(PD^\times\bs PD^\times_\A).
    \end{equation}


For a quadratic \'etale algebra $\cE_v$ over $F_v$,  we write an element $h_{\cE_v}\in\cE_v$ in the form $h_{\cE_v}=a_v+\delta_{\cE_v}b_v$ with $a_v,  b_v\in F_v$.
Define a Haar measure $\d h_{\cE_v}$ on $\cE_v^\times$ by
    \begin{equation}\label{eq:localmeah}
    \d h_{\cE_v}= c_v\, L(1,\eta_{\cE_v}) \, 
    \frac{\d a_v \, \d b_v}{|a_v^2-d_{\cE_v} b_v^2|_v},        
    \end{equation}
where $\eta_{\cE_v}$ is the quadratic character on $F^\times_v$ corresponding to $\cE_v$, $L(s, \eta_{\cE_v})$ is the local Hecke $L$-factor and $\d a_v$,  $\d b_v$ are the Haar measures on $F_v$ chosen in \S\ref{sec:notation}.
For a quadratic \'etale algebra $E$ over $F$, the Tamagawa measure $\d h_E$ on $(\A_F\otimes_FE)^\times$ is given by
    \begin{equation}\label{eq:dh_E}
    \d h_E=\frac{1}{c_F\, L(1,\eta_E)\, |\Delta_F|} \prod_v \d h_{E_v}.    
    \end{equation}
Here, $\eta_E=\otimes_{v\in\Sigma} \eta_{E_v}$ denotes the quadratic character on $\A_F^\times$ corresponding to $E$ and $L(s,\eta)$ denotes the Hecke $L$-function.

We take an $F$-rational gauge form $\omega$ on $H$. 
Then the Tamagawa measure $\d h$ on $H(\A)$ is written as
    \[
    \d h= c_F^{-1}\, |\Delta_F|^{-5}\prod_{v\in\Sigma} c_v \, \omega_v
    \]
where $\omega_v$ is the measure on $H(F_v)$ obtained from $\omega$.

We choose an $F$-basis $x_1,\dots,x_8$ on $V$ and define a gauge form
    \[
    \d X\coloneqq \d x_1\, \d x_2 \, \cdots \d x_8.
    \]
A local Tamagawa measure $\d x_v$ on $V(F_v)$ is obtained from $\d X$ and the Haar measure on $F_v$.
Then, we have the Tamagawa measure $\d x$ on $V(\A)$ as
    \[
    \d x=|\Delta_F|^{-4} \prod_v \d x_v.
    \]

Take an element $E\in X(D)$. 
Let us define a measure concerned with the algebraic variety $Y_{x_E}\coloneqq H_{x_E}^0 \bsl H$. 
A map $\mu_{x_E}$ is defined by
    \[
    \mu_{x_E}:  \, Y_{x_E}\ni H_{x_E}^0 g \mapsto H_{x_E} g   \in H_{x_E} \bsl H \cong V^0.
    \]
Set $\d Y=\mu_{x_E}^* \d X$, and we have
    \[
    \omega_E=(\mu_{x_E}^*P)(Y)^{-2} \d Y
    \]
is an $H$-invariant gauge  on $Y_{x_E}$.
Let $\xi=\omega/\lambda_{x_E}^*\omega_E$ denote the gauge form on $H_{x_E}^0$ determined by $\omega$ and $\lambda_{x_E}^*\omega_\delta$, where $\lambda_{x_E} :H\to Y_{x_E}$, $\lambda_{x_E}(g)=H_{x_E}^0g$.
Recall that $\eta_E=\otimes_v \eta_{E_v}$ is the Hecke character on $F^\times\R_{>0}\bsl \A^\times$ corresponding to $E$, $L(s,\eta_E)$ denote the Hecke $L$-function, and we set $\mathfrak{d}_v:=  (1-\eta_{E_v}(\varpi_v) q_v^{-1})^{-2}$ if $v<\inf$ and $\eta_{E_v}$ is unramified, and $\mathfrak{d}_v:= 1$ otherwise.
Then, we obtain the Tamagawa measure
    \[
    L(1,\eta_E)^{-2}|\Delta_F|^{-1}\prod_{v\in\Sigma} \mathfrak{d}_v \, \xi_v
    \]
on $H_{x_E}^0(\A)$.
Note that $L(s, \eta)$ has a simple pole at $s=1$ when $E=F\oplus F$.
By abuse of notation, we let
    \[
    L(1,\eta_E):=c_F \, (=\mathrm{Res}_{s=1}L(s,\eta_E))
    \]
in this case.
We take the measure $\d y$ on $Y_{x_E}(\A)$ as
    \begin{equation}\label{eq:measuredy}
    \d y\coloneqq c_F^{-1}\, L(1,\eta_E)^2\, |\Delta_F|^{-4}\prod_{v\in\Sigma}  c_v \mathfrak{d}_v^{-1} \omega_{E,v}.    
    \end{equation}

For $\cE_v\in X(D_v)$, we set
    \[
    V_{\cE_v}(F_v)\coloneqq x_{\cE_v}\cdot H(F_v)= \{ x_v\in V(F_v) \mid  P(x_v)\in d_{\cE_v}(F_v^\times)^2 \}.
    \]
It follows from \cite{Saito1}*{Lemma 1.1} that the map $\lambda_{x_E}\colon Y_{x_{E}}(F_v)\to V_E(F_v)$ is surjective, hence the cardinality of each fiber in $V_E(F_v)$ equals $|\Pi_{x_E}(F_v)|=2$.  
Set $\d y_v\coloneqq \omega_{E,v}$. 
For $\Phi\in\cS(V(F_v))$, we have 
    \begin{equation}\label{eq:mu}
    \int_{Y_{x_{E}}(F_v)} \Phi(\mu_{x_E}(y_v)) \, \d y_v = 2\int_{V_{E}(F_v)} \Phi(x_v) \, \d x_v. 
    \end{equation}

\subsection{Local zeta functions}

For $\cE_v\in X(D_v)$, $\Phi_v\in\cS(V(F_v))$ and $\phi_v\in\pi_v$, we introduce the local zeta function
    \[
    Z_{\cE_v}(\Phi_v, \phi_v, s)
    \coloneqq\frac{2\,c_v}{L(1, \eta_{\cE_v})^2} \int_{V_{\cE_v}(F_v)}
    \alpha_{\cE_v}(\pi_v(g_1)\phi_v, \pi_v(g_2)\phi_v)\, |P(x)|_v^{s-2} \, \Phi_v(x) \d x,
    \]
where $x=x_{\cE_v}\cdot \rho(g_1,g_2,g_3)$.

\begin{rem}
For the element $(\gamma,\gamma,\diag(1,-1))\in G_{x_E(F)}\setminus G_{x_E}^0(F)$ given in Proposition \ref{prop:orbits}, we have
    \[
    \alpha_{E_v}(\pi_v(\gamma g_{1})\phi_v,\pi_v(\gamma g_{2})\phi_v)= \alpha_{E_v}(\pi_v(g_{1})\phi_v,\pi_v(g_{2})\phi_v).
    \]
It follows from this that the factor $\alpha_{E_v}(\pi_v(g_{1})\phi_v,\pi_v(g_{2})\phi_v)$ depends only on the point $x_E\cdot \rho(g_{1},g_{2},g_{3})$ in $V_E(F_v)$.
\end{rem}

\begin{lem}\label{lem:loc}
Let the notation be as above. 
\begin{itemize}
\item[(1)] The local zeta function $Z_{\cE_v}(\Phi_v, \phi_v, s)$ is absolutely convergent when $\Re(s)$ is sufficiently large and meromorphically continued to the whole $s$-plane.

\item[(2)] 
When $\phi_v\neq 0$, $\pi_v$ is $\cE_v$-distinguished if and only if there exists a test function $\Phi_v\in \cS(V(F_v))$ such that $Z_{\cE_v}(\Phi_v, \phi_v, s)$ is non-trivial. 

\item[(3)] Assume that $\phi_v\neq 0$ and $\pi_v$ is $\cE_v^\times$-distinguished. 
For $s_0\in\C$ there exists a test function $\Phi_v\in \cS(V(F_v))$ supported on $V_{\cE_v}(F_v)$ such that $Z_{\cE_v}(\Phi_v, \phi_v, s)$ is entire and non-zero around $s=s_0$.

\end{itemize}
\end{lem}
\begin{proof}
(1) is proved in \cite{Li4}.

If $\pi_v$ is not $\cE_v^\times$-distinguished, then $\alpha_{\cE_v}(\pi_v(g_1)\phi_v,\pi_v(g_2)\phi_v)=0$ for all $g_1$, $g_2\in D_v^\times$, hence we have $Z_{\cE_v}(\Phi_v, \phi_v, s)\equiv 0$ for any $\Phi_v\in \cS(V(F_v))$. 

Suppose $\pi_v$ is $\cE_v^\times$-distinguished. 
Since $\pi_v$ is irreducible, there exists an element $g\in D_v^\times$ so that $\alpha_{\cE_v}(\pi_v(g)\phi_v,\pi_v(g)\phi_v)\neq 0$. 
Set $x_0\coloneqq x_{\cE_v}\cdot \rho(g,g,1)$. 
Take a function $\Phi$ whose support is contained in a neighborhood $\Omega$ of $x_0$. 
If $\Omega_0$ is sufficiently small for $x_0$ and $s_0$, then $Z_{\cE_v}(\Phi_v, \phi_v, s)$ satisfies the conditions of (3). 
This completes the proof of (2) and (3). 
\end{proof}



Let $v\in S$.
Since $\cE_v\simeq E_v$,  there exist $t\in \GL_2(F_v)$ and $a\in F_v^\times$ such that $at^{-1}\delta_{\cE_v}t=\delta_{E_v}$.
This implies $x_{E_v}=x_{\cE_v}\cdot \rho(t,t,\diag(1,a))$ and $a^2d_{\cE_v}=d_{E_v}$.
It follows that for $\phi_v,  \phi'_v\in\pi_v$ 
    \begin{align*}
    \alpha_{E_v}(\phi_v,  \phi'_v) 
    &= |a|_v^{-1}\int_{F_v^\times\bs t^{-1}\cE_v^\times t}\langle \pi_v(h_v)\phi_v,  \phi'_v \rangle_v \d h_v  \\
    & =\left|\frac{d_{\cE_v}}{d_{E_v}}\right|_v^\frac12\alpha_{\cE_v}(\pi(t)\phi_v,  \pi(t)\phi'_v).
    \end{align*}
Hence for $v\in S$, we get
    \begin{equation}\label{eq:20240311}
    Z_{E_v}(\Phi_v, \phi_v, s)=\left|\frac{d_{\cE_v}}{d_{E_v}}\right|_v^\frac12
    Z_{\cE_v}(\Phi_v, \phi_v, s).  
    \end{equation}

\subsection{Explicit form of global zeta integral}

Choose a cuspidal automorphic representation $\pi=\otimes_v \pi_v$ of $D_\A^\times$ and take an automorphic form $\phi\in \pi$. 
For $s\in\C$ and $\Phi\in\cS(V(\A))$, we defined the global zeta function $Z(\Phi,\phi,s)$ in \S\,\ref{sec:3.2} as
    \[
    Z(\Phi,\phi,s) \coloneqq \int_{H(F)\bsl H(\A)} |\omega(h)|^s \phi(g_1)\, \overline{\phi(g_2)}\sum_{x\in V^0(F)} \Phi(x\cdot \rho(h))\, \d h,
    \]
where $h=Z_\rho\, (g_1,g_2,g_3)\in H(\A)$.
Following the argument of \cite{Saito1}*{Proof of Theorem 2.1}, we express this zeta integral $Z(\Phi,\phi,s)$ as an infinite sum of Euler products.

For each $E\in X(D)$, we set
    \[
    Z_E(\Phi,\phi,s)\coloneqq\int_{H(F)\bsl H(\A)} |\omega(h)|^s \phi(g_1)\, \overline{\phi(g_2)}\sum_{x\in V_E(F)} \Phi(x\cdot \rho(h))\, \d h .
    \]
Obviously we have the decomposition into a sum of the contribution of $H(\A)$-orbits:
    \begin{equation}\label{eq:20200105e1}
    Z(\Phi,\phi,s)=\sum_{E \in X(D)}Z_E(\Phi,\phi,s).
    \end{equation}
\begin{lem}\label{lem:Saito}
    \[
    Z_E(\Phi,\phi,s)=\frac{1}{2} \int_{Y_{x_E}(\A)} \cP_E(\pi(g_1)\phi)\, \overline{\cP_E(\pi(g_2)\phi)}   \,  |P(\mu_{x_E}(y))|^s \Phi(\mu_{x_E}(y)) \, \d y ,
    \]
where $y=H_{x_E}^0 h=H_{x_E}^0(g_1,g_2,g_3)\in Y_{x_E}(\A)$.
\end{lem}
\begin{proof}
For the point $y_0=H_{x_E}^0$ in $Y_{x_E}(F)$, we have $Y_{x_E}(F)=y_0 H(F)$.  
Using $|\Pi_{x_E}(F)|=2$ and
    \[
    V_E(F)=x_E\cdot H(F)=\mu_{x_E}(y_0\cdot H(F)), 
    \]
and putting $h=Z_\rho(g_1,g_2,g_3)\in H(\A)$, we obtain 
    \begin{align*}
    Z_E(\Phi,\phi,s)\, 
    &=\frac{1}{2} \int_{H(F)\bsl H(\A)} \phi(g_1)\, \overline{\phi(g_2)} \, |\omega(h)|^s \sum_{\gamma \in H_{x_E}^0(F)\bsl H(F)} \Phi( \mu_{x_E}(y_0  \gamma h)) \, \d h \\
    &=\frac{1}{2} \int_{H_{x_E}^0(F)\bsl H(\A)} \phi(g_1)\, \overline{\phi(g_2)} \, |\omega(h)|^s  \Phi( \mu_{x_E}(y_0 h)) \, \d h \\
    &=\frac{1}{2} \int_{H_{x_E}^0(\A)\bsl H(\A)} \cP_E(\pi(g_1)\phi)\, \overline{\cP_E(\pi(g_2)\phi)} \, |\omega(\tilde{h})|^s  \Phi( \mu_{x_E}( y_0\tilde{h})) \, \d \tilde{h} \\
    &=\frac{1}{2} \int_{y_0 H(\A)} \cP_E(\pi(g_1)\phi)\, \overline{\cP_E(\pi(g_2)\phi)} \, |\omega(h)|^s  \Phi( \mu_{x_E}( y)) \, \d y.
    \end{align*}
Here, $\d \tilde{h}$ is a quotient measure on $H_{x_E}^0(\A)\bsl H(\A)$.
Note that the final equality follows from the uniqueness of the invariant measure. 
By Lemma \ref{lem:5} we obtain
    \[
    y_0H(\A)=Y_{x_E}(F)H(\A)=Y_{x_E}(\A).
    \]
This completes the proof.
\end{proof}

For $\cE_S\in X(D_S)$, we set 
    \[
    X(D,\cE_S,\pi)=\{E\in X(D,\cE_S) \mid \text{ $\pi$ is $E^\times$-distinguished.}  \}.
    \]
We define the Dirichlet series $\xi(D, \cE_S, \phi, s)$ by
    \begin{multline*}
    \xi(D, \cE_S, \phi, s) \coloneqq \frac{ 1 }{2 \, |\Delta_F|^4  \, c_F }  \frac{\zeta^S_F(2s-1) \, 
    L^S(2s-1,\pi,\mathrm{Ad}) }{\zeta^S_F(2)^3 } \\
    \times \sum_{E\in X(D,\cE_S,\pi)} \frac{\zeta^S_F(2), L^S(\frac12, \pi) \, L^S(\frac12, \pi\otimes\eta_E)}{|\Delta_F| \, L^S(1, \pi, \Ad) } \frac{\mathcal{D}_E^S(\pi,s)}
    {N(\mathfrak{f}_E^S)^{s-1} }  \prod_{v\in S} \frac{L(1,\eta_{\cE_v}) \, |d_{\cE_v}|_v^{\frac12}}{|d_{E_v}|_v^{\frac12}} .
    \end{multline*}
Here,  $N(\mathfrak{f}_E^S)=\prod_{v\notin S}N(\mathfrak{f}_{\cE_v})$ and $\mathcal{D}_E^S(\pi,s)=\prod_{v\notin S}\mathcal{D}_{E_v}(\pi_v,s)$ is the product of
    \begin{equation*}
    \mathcal{D}_{E_v}(\pi_v,s)= 
        \begin{cases} 
        1+q_v^{-2s+1}+q_v^{-2s}+q_v^{-4s+1} -2\eta_v(\varpi_v)\, q_v^{-2s}\lambda_v 
        & \text{if $\eta_v$ is unramified}, \\ 
        1+q_v^{-2s+1}
        & \text{if $\eta_v$ is ramified}.
        \end{cases}
 \end{equation*}
Then, we obtain an explicit formula of $Z(\Phi,\phi,s)$ as follows.
\begin{thm}\label{thm:20200105t1}
Suppose that $\pi$ is cuspidal if $D=\M_2(F)$, and $\pi$ is not one dimensional if $D$ is division.
Let $\phi\in\pi$ be a non-zero vector.
For sufficiently large $\Re(s)>0$, the Dirichlet series $\xi(D, \cE_S, \phi, s)$ converges absolutely and we have
    \[
    Z(\Phi,\phi,s)= \sum_{\cE_S=(\cE_v)_{v\in S} \in X(D_S)}   \prod_{v\in S} Z_{\cE_v}(\Phi_v,\phi_v,s)  \times  \xi(D, \cE_S, \phi, s).
    \]
\end{thm}
\begin{rem}
By the proof of this theorem, we see that $\pi$ is $E^\times$-distinguished if and only if there exists a test function $\Phi\in\cS(V(\A))$ such that $Z_E(\Phi,\phi,s)$ is not identically zero. 
\end{rem}
\begin{proof} 
If $\pi$ is not $E^\times$-distinguished, then $\cP_E(\pi(g)\phi)=0$ for any $g\in D_\A^\times$, hence we have $Z_E(\Phi,\phi,s)\equiv 0$ for any $\Phi\in\cS(V(\A))$ by Lemma \ref{lem:Saito}. 
The assertion is obvious in this case.

Hereafter, we assume that $\pi$ is $E^\times$-distinguished. 
From \eqref{eq:Wald}, \eqref{eq:measuredy}, \eqref{eq:mu} and Lemma \ref{lem:Saito} we obtain
    \[
    Z_E(\Phi,\phi,s)= \frac{\zeta_F(2) \, L(\frac12, \pi) \, L(\frac12, \pi\otimes\eta_E)}
    {2\, c_F\, |\Delta_F|^5\, L(1, \pi, \Ad)}  \prod_{v\in\Sigma} Z_{E_v}^\#(\Phi_v,\phi_v,s).
    \]    
Here,  $Z_{\cE_v}^\#(\Phi_v, \phi_v, s)$ is the normalized local zeta function which is obtained by replacing $\alpha_{\cE_v}(\bullet)$ in the definition of $Z_{\cE_v}(\Phi_v, \phi_v, s)$ with  $\alpha_{\cE_v}^\#(\bullet)$. 

Take a finite set $S$ of places of $F$ satisfying Condition \ref{condition} so that $\Phi_v$ is the characteristic function of $V(\fo_v)$ for any $v\not\in S$. 
Since we have $\alpha^\#_{E_v}(\phi_v,\phi_v)=1$ for $v\notin S$, we see
    \[
    Z_{E_v}^\#(\Phi_v,\phi_v,s)= \frac{2\,c_v}{L(1, \eta_{\cE_v})^2} \,  \int_{V_{E_v}(\fo_v)}
    \frac{\alpha_{\cE_v}(\pi_v(g_1)\phi_v, \pi_v(g_2)\phi_v)}{\alpha_{\cE_v}(\phi_v, \phi_v)}\, |P(x)|_v^{s-2} \,  \d x.
    \]
We explicitly compute the above integral in Theorem \ref{thm:localzeta} and obtain for $v\not\in S$
    \[
    Z_{E_v}^\#(\Phi_v,\phi_v,s)= \frac{\zeta_{F_v}(2s-1)\, L(2s-1,\pi_v,\mathrm{Ad})\, \mathcal{D}_{E_v}(\pi_v,s)}{\zeta_{F_v}(2)^3 \, N(\ff_{E_v})^{s-1}}.
    \]
The above Euler product becomes
    \begin{multline*}
    Z_E(\Phi,\phi,s)= \frac{\zeta_F(2) \, L(\frac12, \pi) \, L(\frac12, \pi\otimes\eta_E)}
    {2\, c_F\, |\Delta_F|^5\, L(1, \pi, \Ad)}  \frac{\zeta_F^S(2s-1)\, L^S(2s-1,\pi,\mathrm{Ad})\, \mathcal{D}_E^S(\pi,s)}{\zeta_F^S(2)^3 \, N(\ff_E^S)^{s-1}}  \\
    \times  \prod_{v\in S} Z_{E_v}^\#(\Phi_v,\phi_v,s).     
    \end{multline*}
From \eqref{eq:normalized} we obtain
    \begin{multline*}
    Z_E(\Phi,\phi,s)= \frac{L^S(\frac12, \pi) \, L^S(\frac12, \pi\otimes\eta_E)}
    {2\, c_F\, |\Delta_F|^5\, L^S(1, \pi, \Ad)}  \frac{\zeta_F^S(2s-1)\, L^S(2s-1,\pi,\mathrm{Ad})\, \mathcal{D}_E^S(\pi,s)}{\zeta_F^S(2)^2 \, N(\ff_E^S)^{s-1}}  \\
    \times \prod_{v\in S} L(1,\eta_{E_v}) \times  \prod_{v\in S} Z_{E_v}(\Phi_v,\phi_v,s).     
    \end{multline*}
Suppose $\cE_v\cong E_v$ for any $v\in S$.
It follows from \eqref{eq:20240311} that
    \begin{multline}\label{eq:202403111}
    Z_E(\Phi,\phi,s)= \frac{L^S(\frac12, \pi) \, L^S(\frac12, \pi\otimes\eta_E)}
    {2\, c_F\, |\Delta_F|^5\, L^S(1, \pi, \Ad)} \times \frac{\zeta_F^S(2s-1)\, L^S(2s-1,\pi,\mathrm{Ad})\, \mathcal{D}_E^S(\pi,s)}{\zeta_F^S(2)^2 \, N(\ff_E^S)^{s-1}\, \prod_{v\in S}|d_{E_v}|_v^{\frac12}}  \\
    \times \prod_{v\in S} L(1,\eta_{\cE_v}) \, |d_{\cE_v}|_v^{\frac12}  \times  \prod_{v\in S} Z_{\cE_v}(\Phi_v,\phi_v,s).     
    \end{multline}
The desired equality follows from \eqref{eq:20200105e1} and \eqref{eq:202403111}.
Note that the absolute convergence is ensured by Lemma \ref{lem:absconv}.
\end{proof}

\begin{prop}\label{prop:merozeta}
Take an element $(\cE_v)_{v\in S}\in X(D_S)$. 
Assume that $\pi_v$ is $\cE_v^\times$-distinguished for each $v\in S$. 
Then, $\xi(D, \cE_S, \phi, s)$ is meromorphically continued to the whole $s$-plane. 
Note that $\xi(D, \cE_S, \phi, s)$ does not depend on the choice of $\otimes_{v\in S}\phi_v$. 
\end{prop}
\begin{proof}
The assertion follows from Theorems \ref{thm:global} and \ref{thm:20200105t1} and Lemma \ref{lem:loc}.
\end{proof}





\subsection{Absolute convergence of global zeta integral}

\begin{lem}\label{lem:absconv}
The zeta integral $Z(\Phi,\phi,s)$ converges absolutely for sufficiently large $\Re(s)$.  
\end{lem}
\begin{proof}
Since $\phi$ is bounded on $PD^\times\bsl D_\A^\times$, we obtain
    \[
    \int_{H(F)\bsl H(\A)}\Big| |\omega(h)|^s \phi(g_1)\, \overline{\phi(g_2)}\sum_{x\in V^0(F)} \Phi(x\cdot \rho(h))\, \Big| \,  \d h \ll \eqref{eq:cov1} + \eqref{eq:cov2},
    \]
    \begin{equation}\label{eq:cov1}
    \sum_{E\in X(D), \, E\neq F\times F}  \int_{H(F)\bsl H(\A)} |\omega(h)|^{\Re(s)} \sum_{x\in V_E(F)} \Big| \, \Phi(x\cdot \rho(h))\, \Big| \,  \d h ,
    \end{equation}
    \begin{equation}\label{eq:cov2}
    \int_{H(F)\bsl H(\A)} |\omega(h)|^{\Re(s)} |\phi(g_1)|\, |\phi(g_2)| \sum_{x\in V_{F\times F}(F)} \Big| \, \Phi(x\cdot \rho(h))\, \Big| \,  \d h.
    \end{equation}
When $E\neq F\times F$, even if $\pi$ is the trivial representation, Lemma \ref{lem:Saito} is available and $\cP_E(\pi(g_1)\phi)\, \cP_E(\pi(g_2)\phi)$ equals $\tau(H_{x_E}^0)$. 
Hence, using the explicit calculation of local zeta functions in \cite{Saito0}*{Theorem 3.4 (2) $n=4$, $r=2$}, we have
    \[
    \eqref{eq:cov1} \ll \sum_{E\in X(D), \, E\neq F\times F} \frac{L(1,\eta_E)^2}{N(\ff_E)^{\Re(s)-1}}  
    \]
for $\Re(s)\ge 2$. 
Since we have $L(1,\eta_E)\ll_\varepsilon N(\ff_E)^\varepsilon$ for any $\varepsilon>0$ by \cite{LiX}, \eqref{eq:cov1} converges for $\Re(s)>2$. 
Therefore, it is sufficient to porve the convergence of \eqref{eq:cov2}.

The term of $E=F\times F$ appears only in the case $D=\M_2$. 
We may replace $x_E$ with $x_E=(I_2, \diag(1,-1))$. 
Let $T$ denote the diagonal torus in $\GL_2$.
Then we have $H_{x_E}^0\cong T\times T$. 
Furthermore, \eqref{eq:cov2} is bounded by
    \begin{multline}\label{eq:splitb}
    \int_{T(\A)\bsl \PGL_2(\A)} \d g_1\, \int_{T(\A)\bsl \PGL_2(\A)} \d g_2 \, \int_{\GL_2(\A)}\d g_3 \\
    \cP_E(|\pi(g_1)\phi|) \, \cP_E(|\pi(g_2)\phi|)\, |\Phi(x_E\cdot (g_1,g_2,g_3))| \, |\det(g_3)|^{2\Re(s)} .
    \end{multline}
Let $N$ be a positive integer, $a\in\A^\times$, $u\in\A$, and $k\in K$. 
Since $\phi$ is a cusp form, it is rapidly decreasing on $\PGL_2(F)\bsl \PGL_2(\A)$. 
Hence, we have
    \[
    \phi( 
        \begin{pmatrix}
        a&0 \\  
        0& 1  
        \end{pmatrix}
        \begin{pmatrix}
        1&u \\  
        0& 1  
        \end{pmatrix} 
    k) \ll |a |^{-N} \qquad \text{when $|a|\ge 1$}.
    \] 
Set $w\coloneqq\begin{pmatrix}0&1 \\ 1&0\end{pmatrix}\in \PGL_2(F)\cap K$
Taking the Iwasawa decomposition of $\begin{pmatrix}1&0 \\  u& 1  \end{pmatrix}$, we obtain
    \[
    \phi( 
        \begin{pmatrix}
        a&0 \\  
        0& 1  
        \end{pmatrix}
        \begin{pmatrix}
        1&u \\  
        0& 1  
        \end{pmatrix} 
    k)=\phi( 
        \begin{pmatrix}
        a^{-1}&0 \\  
        0& 1  
        \end{pmatrix}
        \begin{pmatrix}
        1&0 \\  
        u& 1  
        \end{pmatrix} 
    wk) \ll |a|^N \,  \| (1,u) \|^{2N}
\] 
when $|a| \, \|(1,u)\|^2 \le 1$. 
Here, we write $\| (1,u_v)\|_v$ for the standard norm on $F_v\times F_v$ related to the valuation, and we set $\| (1,u) \|\coloneqq \prod_v \| (1,u_v)\|_v$ for $u=(u_v)\in\A$. 
By $|\phi(g)|\ll 1$, for any $g=tnk\in \PGL_2(\A)$, $t\in T_\A$, $n\in N_\A$, $k\in K$, 
    \begin{multline*}
    \cP_E(|\pi(g)\phi|)=  (\int_{|a|\geq 1} + \int_{\|(1,u)\|^{-2}\le |a|\le 1} +\int_{|a| \, \|(1,u)\|^2 \le 1})\,  |\phi(\diag(a,1) nk)| \, \d^\times a  \\
    \ll \|(1,u)\|^2+ \int_{\|(1,u)\|^{-2}\le |a|\le 1} \d^\times a  \ll \|(1,u)\|^2+ \|(1,u)\|^2\int_{\|(1,u)\|^{-2}\le |a|\le 1} |a|\d^\times a \\
    \ll \|(1,u)\|^2  \ll \| n \|^2. 
    \end{multline*}
Here, we used $1\le |a|\, \|(1,u)\|^2$ in the bound of $\int_{\|(1,u)\|^{-2}\le |a|\le 1}\d^\times a$. 
Set $g_j=t_jn_jk_j$ for $j=1$, $2$, $3$, where $n_1$ and $n_3$ are upper triangular and $n_2$ is lower triangular.
Then \eqref{eq:splitb} is bounded by the integral
    \begin{multline}\label{eq:splitb2}
    \int_{T(\A)\bsl \PGL_2(\A)} \d g_1\, \int_{T(\A)\bsl \PGL_2(\A)} \d g_2 \, \int_{\GL_2(\A)}\d g_3 \\
    \|n_1\|^2\|n_2\|^2\, |\Phi( ( *,*   ) \cdot (k_1,k_2,k_3)| \, |\det(t_3)|^{2\Re(s)}.    
    \end{multline}
In addition, if we put $n_1=
    \begin{pmatrix} 
    1 & u_1 \\ 
    0 & 1
    \end{pmatrix}$, 
$n_2=
    \begin{pmatrix} 
    1 & 0 \\ 
    u_2 & 1
    \end{pmatrix}$, 
$t_3=\diag(a,b)$, $n_3=
    \begin{pmatrix} 
    1 & u_3 \\ 
    0 & 1
    \end{pmatrix}$, 
then the part $(*,*)$ of $\Phi$ is
    \[
    (*,*)=(
        \begin{pmatrix} 
        a-au_1u_2 & -au_1 \\ 
        au_2 & a 
        \end{pmatrix}, 
    u_3    
        \begin{pmatrix} 
        a-au_1u_2 & -au_1 \\  
        au_2 & a 
        \end{pmatrix}
    +
        \begin{pmatrix} 
        b+bu_1u_2 & bu_1 \\  
        -bu_2 & -b 
        \end{pmatrix}
    ).
    \]
By the change of variables $u_1\mapsto a^{-1}u_1$, $u_2\mapsto a^{-1}u_2$, $u_3\mapsto a^{-1}(u_3+b)$, we obtain
    \[
    (*,*)\mapsto (
        \begin{pmatrix} 
        a-a^{-1}u_1u_2 & -u_1 \\ 
        u_2 & a 
        \end{pmatrix}, 
        \begin{pmatrix} 
        2b+u_3-a^{-1}u_1u_2u_3 & -a^{-1}u_1u_3 \\ 
        a^{-1}u_2u_3 & u_3 
        \end{pmatrix}
    ).
    \]
This means that $\Phi$ is rapidly decreasing for the five variables $u_1$, $u_2$, $u_3$, $a$, $b$. 
Therefore, we find that \eqref{eq:splitb2} is convergent for sufficiently large $\Re(s)$. 
This completes the proof. 
\end{proof}

\section{Explicit calculations of local zeta functions over \texorpdfstring{$p$-adic fields}{TEXT}}\label{sec:localzeta424}

\subsection{Notation}\label{sec:notation2}

recall that $F$ is a number field.
Fix a finite place $v$ which is not dyadic.
We omit the subscript $v$, hence in particular $F=F_v$ is a finite extension of $\Q_p$.
Let $\fo$ be the integer ring of $F$, $\varpi$ a prime element. 
Let $q$ denote the order of the residue field of $F$.
Throughout this section, we suppose that the quaternion algebra $D$ is split over $F$, \emph{i.e.} $D=\M_2(F)$.
We use the notation and normalization in \S~\ref{sec:space} and \S~\ref{sec:notation}.
Note that we have $|\varpi|=q$ and $\int_\fo \d x=1$.

Take an \'etale quadratic algebra $E$ over $F$ and fix an embedding $E\hookrightarrow\M_2(F)$.
Recall that we choose an element $\delta=\delta_E\in \M_2(F)$ so that $E=F+F\delta$ and $\Tr(\delta)=0$
in \S~\ref{sec:orbits}.
Put $d=d_E = \delta^2$.
We may suppose that we have 
    \[
    d_E\in\fo^\times\sqcup\varpi\fo^\times
    \]
without loss of generality. 
Let $\eta_E$ be the quadratic character on $F^\times$ corresponding to $E$ and $\ff_E$ be its conductor.
Put $N(\ff_E)\coloneqq \#(\fo/\ff_E)$.
Then we have 
    \[
    L(s,\eta_E) = 
        \begin{cases}
        (1-\eta_E(\varpi)q^{-s})^{-1} & \text{if $\eta_E$ is unramified,} \\
        1 & \text{otherwise}. 
        \end{cases}
    \]
We write $\trep$ for the trivial representation of $F^\times$.

Let $W= \{x\in \M_2 \mid x={}^t\!x \}$ be the space of symmetric matrices.
To simplify the notation, we write an element $\left(
    \begin{smallmatrix} 
    x_1& x_{12}/2 \\ 
    x_{12}/2 & x_2 
    \end{smallmatrix}
\right)\in W(F)$ as $(x_1,x_{12},x_2)$.
For $j=1,2$, we define maps $\cF_j:V(F)\to W(F)$ by 
    \[
    \cF_1(x)\coloneqq \left(\det
        \begin{pmatrix}
        x_{11} & x_{12} \\ 
        y_{11} & y_{12} 
        \end{pmatrix},  
    \det
        \begin{pmatrix}
        x_{11} & x_{12} \\ 
        y_{21} & y_{22} 
        \end{pmatrix}
    +\det
        \begin{pmatrix} 
        x_{21} & x_{22} \\ 
        y_{11} & y_{12} 
        \end{pmatrix}    , 
    \det
        \begin{pmatrix}
        x_{21} & x_{22} \\ 
        y_{21} & y_{22} 
        \end{pmatrix} 
    \right),
    \]
    \[
    \cF_2(x)\coloneqq \left(\det
        \begin{pmatrix} 
        x_{11} & y_{11} \\ 
        x_{21} & y_{21} 
        \end{pmatrix},  
    \det
        \begin{pmatrix}
        x_{11} & y_{12} \\ 
        x_{21} & y_{22} 
        \end{pmatrix}
    +\det
        \begin{pmatrix}
        x_{12} & y_{11} \\ 
        x_{22} & y_{12} 
        \end{pmatrix}    , 
    \det
        \begin{pmatrix}
        x_{12} & y_{12} \\ 
        x_{22} & y_{22} 
        \end{pmatrix}
    \right).
    \]
for $x=\left( \left( 
    \begin{smallmatrix}
    x_{11}&x_{12} \\ 
    x_{21}&x_{22} 
    \end{smallmatrix}
\right),\left(
    \begin{smallmatrix}
    y_{11}&y_{12} \\ 
    y_{21}&y_{22} 
    \end{smallmatrix}
\right) \right)\in V(F)$.
Then one can check 
    \begin{align*}
    \cF_1(x\, \rho(g))&=\det(g_2) \, \det(g_3) \, g_1^{-1}\cF_1(x)\, {}^t\!g_1^{-1},\\ 
    \cF_2(x\, \rho(g))&=\det(g_1)^{-1} \, \det(g_3) \, {}^t\! g_2\cF_2(x)\, g_2
    \end{align*}
for $g=(g_1,g_2,g_3)\in G$ and $x\in V$.
Note that we have $P(x)=-4\det(\cF_j(x))$ for each $j$.

Let $B(F)$ be the upper triangular Borel subgroup of $\GL_2(F)$, $T(F)$ the diagonal torus and $N(F)$ the upper triangular unipotent subgroup.
For smooth characters $\chi_1, \chi_2$ of $F^\times$, we obtain a smooth character $\chi_1\otimes\chi_2$ of $T(F)$.
We regard it as a character of $B(F)$ by extending trivially on $N(F)$.
Let $\chi_1\times\chi_2$ denote the normalized parabolically induced representation $\Ind_{B(F)}^{\GL_2(F)}(\chi_1\otimes\chi_2)$ of $\GL_2(F)$.

Fix $\pi$ an irreducible unitary unramified representation of $\GL_2(F)$ with trivial central character.
This can be written uniquely as a quotient of an induced representation $\chi\times\chi^{-1}$ with an unitary unramified character $\chi$ of $F^\times$.
Let $\alpha=\chi(\varpi)\in\C^\times$ be the Satake parameter.
We use following local $L$-factors of $\pi$:
\setlength{\leftmargini}{20pt}
\begin{itemize}\setlength{\itemsep}{5pt}    
\item the standard $L$-factor $L(s, \pi)=(1-\alpha q^{-s})^{-1}(1-\alpha^{-1}q^{-s})^{-1}$,
\item the adjoint $L$-factor $L(s, \pi, \Ad)=(1-q^{-s})^{-1}(1-\alpha^2q^{-s})^{-1}(1-\alpha^{-2}q^{-s})^{-1}$. 
\end{itemize}
\setlength{\leftmargini}{30pt}

\subsection{Waldspurger's model}

We fix $\GL_2(F)$-invariant Hermitian pairing $\l\,,\r$ on $\pi$ and define the $E^\times\times E^\times$-invriant bilinear form $\al_E$ on $\pi\boxtimes\bar{\pi}$ by
    \[
    \al_E(\phi_1, \phi_2)=\int_{F^\times\bs E^\times}\l\pi(t)\phi_1, \phi_2\r\,d^\times t.
    \]
Though this integral is over a non-compact region when $E^\times$ is a split torus, it converges absolutely.
Let $\phi\in\pi$ be the $K$-spherical vector of norm $1$ and assume $\al_E(\phi, \phi)\neq0$.
For $l_1, l_2\in\Z$, we set
    \[
    \be_E(l_1) = \frac{\al_E(\pi((
        \begin{smallmatrix}
        \varpi^{-l_1} & \\
        & 1
        \end{smallmatrix}
    ))\phi,\phi)}{\al_E(\phi, \phi)}, \qquad
    \be_E(l_1, l_2) = \frac{\al_E(\pi((
        \begin{smallmatrix}
        \varpi^{-l_1} & \\
        & 1
        \end{smallmatrix}
    ))\phi, \pi((
        \begin{smallmatrix}
        1 & \\
        & \varpi^{l_2}
    \end{smallmatrix}
    ))\phi)}{\al_E(\phi, \phi)}.
    \]
Note that
    \[
    \beta_E(0,0)=1, \quad  \beta_E(l,0)=\beta_E(l), \quad \beta_E(l_1,l_2)=\overline{\beta_E(l_2,l_1)}.
    \]

Recall that $\pi$ is unitary.
From \cite{Tadic}, we have $q^{-\frac12}\leq |\alpha|\leq q^\frac12$.
Set $\lambda = q^\frac12(\alpha+\alpha^{-1})\in\R$.
We can rewrite the above $L$-factors as
    \begin{align*}
    L(s,\pi) & =(1-q^{-\frac12}\lambda q^{-s}+q^{-2s})^{-1}, \\ 
    L(s,\pi,\mathrm{Ad}) & =\{1-(q^{-1}\lambda^2-1) q^{-s}+(q^{-1}\lambda^2-1)q^{-2s}-q^{-3s} \}^{-1}. \end{align*}
It follows from \cite{KP}*{Lemmas 3.2, 3.3} that the sequence $\{ \beta_E(l, m)\}_l$ satisfies the recursion formula
    \begin{equation}\label{eq:re1}
    q\beta_E(l+2,m) - \lambda\beta_E(l+1,m) + \beta_E(l,m) = 0 
    \end{equation}
for $l,m\in\Z_{\geq0}$.
In addition
    \begin{equation}\label{eq:re3}
    \beta_E(1)=
        \begin{cases}
        \frac{\lambda-2}{q-1} & \text{if $d_E\in(\fo^\times)^2$}, \\[2pt] 
        \frac{\lambda}{q+1} & \text{if $d_E\in\fo^\times\setminus(\fo^\times)^2$}, \\[2pt] 
        \frac{\lambda-1}{q} & \text{if $d_E\in\varpi\fo^\times$.}
        \end{cases}
    \end{equation}


\begin{lem}\label{lem:20240314}
For $l_1,  l_2\in\Z_{\geq0}$, we have
    \begin{equation}\label{eq:re5}
    \be_E(l_1, l_2)=\be_E(l_1) \, \be_E(l_2).
    \end{equation}
\end{lem}
\begin{proof}
For $x\in\GL_2(F)$, two linear forms 
    \[
    \varphi\mapsto\al_E(\varphi, \phi), \hspace{10pt} \varphi\mapsto\al_E(\varphi, \pi(x)\phi)
    \]
on $\pi$ are both $E^\times$-invariant.
Since we know from \cite{G97} that $\dim\Hom_{E^\times}(\pi, \C)\leq 1$, there exists a constant $c_\phi(x)$ such that  
    \begin{equation}\label{eq0}
    \al_E(\varphi, \pi(x)\phi)=c_\phi(x)\al_E(\varphi, \phi),  \qquad \varphi\in\pi.
    \end{equation}

We compute $c_\phi(\diag(1, \varpi^{l_2}))$. 
From (\ref{eq0}), we obtain
    \[
    c_\phi(1, \diag(\varpi^{l_2}))=\ol{\beta_E(l_2)}.
    \]
Here, we used the fact that $\al_E(\phi, \phi)\in\R_{>0}$ and  the obvious equation
    \[
    \ol{\al_E(\varphi, \pi(x)\phi)}=\al_E(\pi(x)\phi, \varphi).
    \]

From \eqref{eq:re1} and \eqref{eq:re3}, we see that $\be_E(l_2)\in\R$.
Thus $c_\phi(\diag(1, \varpi^{l_2}))=\be_E(l_2)$.
Substituting $\varphi=\pi(\diag(\varpi^{-l_1}, 1))\phi$ and $x=\diag(1, \varpi^{l_2})$ in (\ref{eq0}), we obtain
    \[
    \be_E(l_1, l_2)=\be_E(l_2)\al_E(\pi(\diag(\varpi^{-l_1}, 1))\phi, \phi)\, \,  \al_E(\phi, \phi)^{-1}=\be_E(l_1) \, \be_E(l_2).  \qedhere 
    \]
\end{proof}

\subsection{Some formulas for \texorpdfstring{$\beta_E(l_1,l_2)$}{TEXT}}
We give some formulas for $\beta_E(l_1,l_2)$, which is needed in the proof of the explicit formula for the local zeta function Theorem \ref{thm:localzeta}.
We use the techniques of \cite{BFF} and \cite{KP}. 
Throughout this subsection, $x$, $y$, $z$ are variables, and suppose $j\in \Z_{\geq 0}$. 
Set
    \[
    A_j(x) = \sum_{k=0}^\inf \beta_E(j+k) \, x^k.
    \]
From \eqref{eq:re1} we have the recursion formula
    \[
    x^2\{ q A_{j+2}(x) - \lambda A_{j+1}(x) + A_j(x) \} = 0.
    \]
Substituting $x A_{j+1}(x) = A_j(x) - \beta_E(j)$ into the above recursion formula, we obtain 
    \[
    q\left( A_j(x)-\beta_E(j)-\beta_E(j+1)x \right)- x\lambda \left( A_j(x)-\beta_E(j)\right)+ x^2 A_j(x) = 0.
    \]
Hence we have
    \begin{equation}\label{eq:re6}
    A_j(x)=\frac{ \beta_E(j)+\beta_E(j+1)x-q^{-1}\lambda\cdot \beta_E(j)x  }{1-\lambda q^{-1}x+q^{-1}x^2}=\frac{ \beta_E(j)-\beta_E(j-1)\, q^{-1}x  }{1-\lambda q^{-1}x+q^{-1}x^2},
    \end{equation}
cf. \cite{KP}*{Proposition 3.4}.
Set
    \[
    B_j(x) =  \sum_{l=0}^\inf \beta_E(l+j)^2 \, x^l , \qquad 
    C_j(x) =  \sum_{l=0}^\inf \beta_E(l+j+1) \, \beta_E(l+j) \, x^l .
    \]
\begin{lem}\label{lem:relation}
    \[
    C_j(x)=\frac{q^{-1}\lambda}{1+q^{-1}x}B_j(x) + \frac{ \beta_E(j) \, \beta_E(j+1)-q^{-1}\lambda \cdot \beta_E(j)^2 }{1+q^{-1}x}.
    \]
\end{lem}
\begin{proof}
From \eqref{eq:re1} one obtains
    \[
    x\{ q C_{j+1}(x) - \lambda B_{j+1}(x) + C_j(x) \} = 0.
    \]
By $xB_{j+1}(x)=B_j(x)-\beta_E(j)^2$ and $xC_{j+1}(x)=C_j(x)-\beta_E(j+1)\, \beta_E(j)$, we see that
    \[
    (C_j(x)-\beta_E(j+1)\, \beta_E(j))-q^{-1}\lambda(B_j(x)-\beta_E(j)^2)+q^{-1}xC_j(x)=0.
    \]
Hence we obtain the assertion.
\end{proof}

\begin{prop}\label{prop:410x}
We have
    \begin{multline*}
    B_j(x)=D(x) 
    \Big[\beta_E(j)^2 + x\{ \beta_E(j+1)^2-q^{-1}(q^{-1}\lambda^2-1)\, \beta_E(j)^2\} \\
    + q^{-1}x^2\{ \beta_E(j+1)-q^{-1}\lambda\cdot \beta_E(j)\}^2\Big],
    \end{multline*}
where
    \[
    D(x)\coloneqq \{1-(q^{-1}\lambda^2-1)\, q^{-1}x +(q^{-1}\lambda^2-1)\, q^{-2}x^2 -q^{-3}x^3 \}^{-1} .
    \]
\end{prop}
\begin{proof}
From \eqref{eq:re1} one obtains
    \[
    q^2B_{j+2}(x)-\lambda^2B_{j+1}(x)+2\lambda C_j(x)-B_j(x)=0.
    \]
By Lemma \ref{lem:relation} and $xB_{j+1}(x)=B_j(x)-\beta_E(j)^2$, we have
    \begin{multline*}
    (B_j(x)-\beta_E(j)^2-x\beta_E(j+1)^2)-q^{-2}\lambda^2x(B_j(x)-\beta_E(j)^2)-q^{-2}x^2B_j(x) \\
    +2q^{-2}x^2\lambda \left\{ \frac{q^{-1}\lambda}{1+q^{-1}x}B_j(x) + \frac{ \beta_E(j) \, \beta_E(j+1)-q^{-1}\lambda \cdot \beta_E(j)^2 }{1+q^{-1}x}  \right\}=0.
    \end{multline*}
Hence we have
    \begin{align*}
    & (1-(q^{-1}\lambda^2-1)\, q^{-1}x +(q^{-1}\lambda^2-1)\, q^{-2}x^2 -q^{-3}x^3)B_j(x)  \\
    & \qquad = (\beta_E(j)^2+x\beta_E(j+1)^2   -q^{-2}\lambda^2x\beta_E(j)^2)(1+q^{-1}x) \\
    & \hspace{100pt} - 2q^{-2}x^2\lambda(\beta_E(j) \, \beta_E(j+1)-q^{-1}\lambda \cdot \beta_E(j)^2). 
    \end{align*}
This proves the proposition.
\end{proof}
Set
    \[
    U_j(x,y) = \sum_{l_1=0}^\inf \sum_{l_2=j}^\inf \beta_E(l_1+l_2+1) \, \beta_E(l_2) \, x^{l_1} y^{l_2} .
    \]
\begin{prop}\label{prop:412x}
We have
    \[
    U_j(x,y) = y^j \cdot \frac{ (q^{-1}\lambda-q^{-1}x- q^{-2}xy)B_j(y) + \beta_E(j)\beta_E(j+1)-q^{-1}\lambda \cdot \beta_E(j)^2}{(1-\lambda q^{-1}x+q^{-1}x^2)(1+q^{-1}y)}.
    \]
\end{prop}
\begin{proof}
From \eqref{eq:re6} we see that
    \[
    U_j(x,y)  =y^j \sum_{l=0}^\inf A_{l+j+1}(x)\, \beta_E(l+j)\, y^l 
     =\frac{y^j\left( C_j(y)-q^{-1}x \, B_j(y) \right) }{1-\lambda q^{-1}x+q^{-1}x^2}.
    \]
Hence, the assertion follows from Lemma \ref{lem:relation}.
\end{proof}

\begin{lem}\label{lem:20191230}
We have
    \[
    \sum_{k=0}^\inf \sum_{u=0}^\inf  x^k y^u \beta_E(k+u+1)^2 =D(y)  \{ B_0(x) \, f_1(x,y)+f_2(x,y) \},
    \]
where 
    \begin{align*}
    f_1(x,y)& \coloneq x^{-2}y+x^{-1}+q^{-1}x^{-1}y-q^{-2}\lambda^2 x^{-1}y+q^{-3}y^2  ,\\
    f_2(x,y)& \coloneq -x^{-2}y-x^{-1}y\, \beta_E(1)^2 -x^{-1}-q^{-1}x^{-1}y+q^{-2}\lambda^2 x^{-1}y.
    \end{align*}
\end{lem}
\begin{proof}
By the definition of the series $B_j(y)$ and Proposition \ref{prop:410x}, we obtain
    \begin{align*}
    & \sum_{k=0}^\inf \sum_{u=0}^\inf  x^k y^u \beta_E(k+u+1)^2 
    = \sum_{k=0}^\inf x^k B_{k+1}(y) \\
    &=D(y)\sum_{k=0}^\inf x^k \{ \beta_E(k+1)^2 + y(\beta_E(k+2)^2 \\
    & \qquad -(q^{-2}\lambda^2 -q^{-1})\beta_E(k+1)^2)+q^{-3}y^2\beta_E(k)^2 \}\\
    &= D(y)\, \{  yB_2(x)+ (1+q^{-1}y-q^{-2}\lambda^2)B_1(x)+q^{-3}y^2B_0(x)  \}.
    \end{align*} 
Substituting $xB_1(x)=B_0(x)-1$ and $x^2B_2(x)=B_0(x)-1-x\beta_E(1)^2$ into the last form, we obtain the assertion.
\end{proof}

\begin{lem}\label{lem:0409l}
We have
    \begin{multline*}
    (1+q^{-1}z)\, (1-\lambda q^{-1}y+q^{-1}y^2) \sum_{k=1}^\inf \sum_{l=0}^\inf \sum_{u=0}^\inf x^k y^l z^u \beta_E(k+u,k+l+u+1) \\
    =(q^{-1}\lambda-q^{-1}y-q^{-2}yz) x \sum_{k=0}^\inf x^k B_{k+1}(z) - q^{-1}x\, C_0(x).
    \end{multline*}
\end{lem}
\begin{proof}
This equality follows from \eqref{eq:re6} and Lemma \ref{lem:relation}.
\end{proof}

\subsection{Local explicit formula}

We compute the local zeta function $Z_{\cE_v}(\Phi_v,\phi_v,s)$ in the unramified situation.
Let $\cK = \GL_2(\fo)\times \GL_2(\fo)\times \GL_2(\fo)$ be the maximal open compact subgroup of $G(F)$, $V_E(\fo) = \{x\in V(\fo) \mid P(x)\in d_E(F^\times)^2\}$ be a compact subset of $V_E(F)$.
For $l_1,  l_2, m, M\in\Z$ and $c\in F$, we set
    \[
    g(l_1,l_2,m_1,m_2,c) \coloneqq \left( 
        \begin{pmatrix} 
        \varpi^{-l_1} & 0 \\ 
        0 & 1
        \end{pmatrix}, 
        \begin{pmatrix} 
        1 & 0 \\ 
        0 & \varpi^{l_2}
        \end{pmatrix}, 
        \begin{pmatrix} 
        \varpi^{m_1} & c \\ 
        0 & \varpi^{m_2}
        \end{pmatrix} 
    \right) \in  G(F),
    \]
    \[
    S_E(l_1,l_2,m_1,m_2,c) \coloneqq x(\delta_E)\cdot \rho\left( g(l_1,l_2,m_1,m_2,c)\cK\right).
    \]
The goal of this section is an explicit formula for the local zeta function
    \[
    Z_E(\pi,s)\coloneqq \int_{V_E(\fo)} \beta_E(l_1,l_2) \, |P(x)|^{s-2}, \, \d x \qquad s\in\C.
    \]
Here, $l_1$ and $l_2$ are determined so that $x \in S(E,l_1,l_2,m_1,m_2,c)$ for some $m_1$, $m_2\in\Z$ and some $c\in F$.
See \eqref{eq:decoV}.

\begin{thm}\label{thm:localzeta}
Let $\pi$ be an irreducible unitary unramified representation of $\GL_2(F)$ with trivial central character.
We have
    \[
    Z_E(\pi,s)=\frac{L(1,\eta_E)^2 \,L(2s-1,\trep)\,L(2s-1,\pi,\mathrm{Ad}) }
    {2L(1,\trep) \, L(2,\trep)^3 \, N(\ff_E)^{s-1}} \cdot \mathcal{D}_{E}(\pi,s) ,
    \]
where 
    \[
    \mathcal{D}_E(\pi,s) \coloneq
        \begin{cases}
        1+q^{-2s+1}+q^{-2s}+q^{-4s+1} -2\eta(\varpi)\, q^{-2s}\lambda 
        & \text{if $\eta_E$ is unramified}, \\
        1+q^{-2s+1} & \text{if $\eta_E$ is ramified.}
        \end{cases}
    \]
\end{thm}

\subsection{Decompositions}\label{sec:binaryquad}

Recall that $X(D)$ is the set of \'etale quadratic algebras over $F$ embedded in $\M_2(F)$.
Set $Q(x)\coloneqq -4\det(x)$ $(x\in W(F))$ and $W(\fo)\coloneqq \left\{(x_1, x_{12}x_2) \mid x_1,x_{12},x_2\in\fo  \right\}$.
For each $E\in X(D)$, we set
    \[
    W_E(F)\coloneqq \{   x\in W(F) \mid Q(x)\in d_E(F^\times)^2   \} , \quad W_E(\fo)\coloneqq W(\fo)\cap W_E(F).
    \]
Let $\qG = \GL_1(F)\times \GL_2(F)$.
Consider the rational representation $\rho_W$ of $\qG$ on $W$ defined by
    \[
    x \, \rho_W(a,h)=\frac{a}{\det(h)} \, {}^t\! h x h ,\qquad (a,h)\in \qG, \quad x\in W(F).
    \]
Obviously $\mathrm{Ker}\rho_W=\{ (1,a I_2)\in \qG \mid a\in F^\times \}\cong F^\times$.
We have an orbit decomposition of $W^0(F)\coloneqq \{x\in W(F) \mid Q(x)\neq 0\}$ such that
    \[
    W^0(F)= \bigsqcup_{E\in X(D)}W_E(F),
    \]
where $W_E(F)$ denotes the orbit of $\diag( 1,-d_E)$.
The stabilizer $\qG_E$ of $\diag( 1,-d_E)$ and its connected component $\qG_E^0$ at $1$ are given by
    \[
    \qG_E^0=\{(1,g)\in \qG \mid g\in \mathrm{GO}_{2,E}^0(F)\}, \quad 
    \qG_E=\qG_E^0 \sqcup \diag( 1,-1) \qG_E^0 .
    \]
Let $\qK=\GL_1(\fo)\times\GL_2(\fo)$ be the maximal open compact subgroup of $\qG$ and $S_E(l_1,l_2)$ the $\qK$-orbit of $\varpi^{l_1} \, \diag( 1 ,-d_E \varpi^{2l_2})$.
Then we have the decomposition 
    \[
    W_E(F)= \bigsqcup_{l_1\in\Z} \bigsqcup_{l_2\in\Z_{\geq 0}} S_E(l_1,l_2), \quad W_E(\fo)= \bigsqcup_{l_1\in\Z_{\geq 0}} \bigsqcup_{l_2\in\Z_{\geq 0}} S_E(l_1,l_2).
    \]
From this we obtain
    \begin{equation}\label{eq:906e1}
    \qG_E^0\bsl \qG/\qK=\qG_E\bsl \qG/\qK=\bigsqcup_{m_1\in\Z} \bigsqcup_{m_2\in\Z_{\geq 0}} \qG_E^0 (\varpi^{m_1}, \diag(1,\varpi^{m_2})) \qK.
    \end{equation}
In addition, we have 
    \begin{equation}\label{eq:sym2}
    \vol( S_E(l_1,l_2)) = \frac{1}{2} \, \frac{1}{L(1,\trep)\, L(2,\trep)} \begin{cases} q^{-3l_1-2l_2} & \text{if $d_E\in \fo^\times$ and $l_2>0$}, \\ q^{-3l_1-2l_2-1} & \text{if $d_E\in \varpi \fo^\times$},  \\ q^{-3l_1}L(1,\eta) & \text{if $d_E\in \fo^\times$ and $l_2=0$}, \end{cases}
    \end{equation}
under the normalization $\vol(W(\fo))=1$.
See \cite{IS}*{Lemma 3.2} or \cite{Saito0}*{Proposition 2.8} for the proof.

By \eqref{eq:906e1} we have
    \[
    G_{x(\delta)}^0(F)\bsl G(F)/\cK=\bigsqcup_{l_1,l_2\in \Z_{\geq 0}}\bigsqcup_{m_1,m_2\in \Z , \; c\in F/\varpi^{m_1}\fo}   G_{x(\delta)}^0(F)\, g(l_1,l_2,m_1,m_2,c) \cK  .
    \]
Set $l_3\coloneqq \min(l_1,l_2)$.
Then we have
    \begin{equation}\label{eq:decoV}
    V_E(\fo)=\bigsqcup_{l_1,l_2\in\Z_{\geq 0}} \bigsqcup_{m_1\geq -l_3}
    \bigsqcup_{m_2\in\Z_{\geq 0} } \bigsqcup_{c\in \varpi^{-l_3}\fo/\varpi^{m_1}\fo}  S_E(l_1,l_2,m_1,m_2,c).
    \end{equation}

Set 
    \[
    \max \cF_j(x) = \max\{ |y_1|, \ |y_{12}|, \ |y_2| \},
    \]
where we write $\cF_j(x)=(y_1,y_{12},y_2)$.
See \S~\ref{sec:notation2} for notation.
\begin{lem}\label{lem:param}
For each $x\in V_E(\fo)$ and the parameter $(l_1,l_2)$, we have the relation
    \[
    |d_E| \, q^{-2l_j} \left(\max \cF_j(x)\right)^2  =|P(x)|   \qquad (j=1, \, 2).
    \]
\end{lem}
\begin{proof}
The assertion follows from
    \[
    \cF_1(x(\delta)\cdot\rho(g(l_1,l_2,m_1,m_2,c)))=(\delta\varpi^{2l_1},0,-1) \times \varpi^{l_2+m_1+m_2},
    \]
    \[
    \cF_2(x(\delta)\cdot\rho(g(l_1,l_2,m_1,m_2,c)))=(1,0,-\delta\varpi^{2l_2}) \times \varpi^{l_1+m_1+m_2}. \qedhere
    \]
\end{proof}

\subsection{Proof of Theorem \ref{thm:localzeta}}

Set $\sG = G(\F_q)$.
By Propositions \ref{prop:sor1} and \ref{prop:orbits}, we have a list of $\mathscr{G}$-orbits in $V(\F_q)$.
We use the same representative elements $x_j$ $(0\leq j\leq 5)$ given in Proposition \ref{prop:sor1}. 
By a direct calculation, one obtains the following lemma.
\begin{lem}\label{lem:card}
The order of each orbit is as follows: 
    \[
    \# \sG = \sG_{x_0}=q^{12}(1-q^{-1})^3(1-q^{-2})^3, \quad \#(x_0 \cdot \rho(\sG))=1,
    \]
    \[
    \# \mathscr{G}_{x_1}=q^8 (1-q^{-1})^5, \quad \#(x_1 \cdot \rho(\mathscr{G}))= q^4(1+q^{-1})^2 (1-q^{-2}),
    \]
    \[
    \# \mathscr{G}_{x_j}=q^7 (1-q^{-1})^3(1-q^{-2}), \quad \#(x_j \cdot \rho(\mathscr{G}))= q^5 (1-q^{-2})^2 \quad(j=2,3,4),
    \]
    \[
    \# \mathscr{G}_{x_5}=q^5 (1-q^{-1})^3, \quad \#(x_5 \cdot \rho(\mathscr{G}))= q^7(1-q^{-2})^3,
    \]
    \[
    \# \mathscr{G}_{x_6}=2q^4 (1-q^{-1})^4, \quad \#(x_6 \cdot \rho(\mathscr{G}))=\frac{1}{2} q^8(1+q^{-1})(1-q^{-2})^2,
    \]
    \[
    \# \mathscr{G}_{x_7}=2q^4 (1-q^{-2})^2, \quad \#(x_7 \cdot \rho(\mathscr{G}))=\frac{1}{2} q^8(1-q^{-1})^3(1-q^{-2}).
    \]
Here, we have choose regular elements $x_6, x_7\in V(\F_q)$ so that $P(x_6)\in (\F_q^\times)^2$ and $P(x_7)\in \F_q^\times\setminus (\F_q^\times)^2$.
\end{lem}
For each $x_j$ $(0\leq j\leq 7)$, set
    \[
    Z_{E,j}(\pi,s) = \int_{(x_j\rho(G(\fo/\varpi\fo)) + \varpi V(\fo))\cap V_E(\fo)} \beta_E(l_1,l_2) \, |P(x)|^{s-2} \, \d x.
    \]
For any $x\in V_E(\fo)$ and any $m\in\Z$,  $\varpi^m x$ has the same parameter 
$(l_1,l_2)$ as that of $x$ by Lemma \ref{lem:param}.
Hence, we obtain 
    \begin{equation}\label{eq:1} 
    (1-q^{-4s})Z_E(\pi,s)=\int_{V_E(\fo)\setminus \varpi V_E(\fo)}\beta_E(l_1,l_2) \, |P(x)|^{s-2} \, \d x =  \sum_{j=1}^7 Z_{E,j}(\pi,s).
    \end{equation}
Therefore, the explicit calculation for $Z_E(\pi,s)$ is reduced to those of $Z_{E,j}(\pi,s)$.
This is a standard method by Igusa for computations of local zeta functions, see \cite{Igusa}.

\subsubsection{Contribution of $x_2\cdot \rho(\sG)$}
Since any element in $x_2\cdot \rho(\sG)+\varpi V(\fo)$ can be reduced to the form
    \[
    X_2\coloneqq  \left( 
        \begin{pmatrix} 
        1&0 \\ 
        0&1 
        \end{pmatrix}, 
        \begin{pmatrix} 
        x&z \\ 
        y&-x 
        \end{pmatrix}
    \right) \quad (x,y,z\in\varpi\fo)
    \]
by the action of $\sG$, the integral $Z_{E,2}(\pi,s)$ becomes
    \[
    \#(x_2 \cdot \rho(\sG)) \, q^{-5} \, \int_{(\varpi\fo)^{\oplus 3},\, P(X_2)\in d_E(F^\times)^2}  \beta_E(l_1,l_2) \, |P(X_2)|^{s-2} \d x \, \d y \, \d z .
    \]
In addition, we have $\cF_1(X_2)= (z,-2x,-y)$ and $\cF_2(X_2)=  (y,-2x,-z)$. 
Hence, by \eqref{eq:sym2}, Lemma \ref{lem:card} and the fact that $\rho(g,g,1)$ acts on $Y=\left(
    \begin{smallmatrix} 
    -y&x \\ 
    x&z 
    \end{smallmatrix}
\right)$ as $Y\cdot\rho(g,g,1)=\det(g)^{-1}{}^t\!gYg$, one finds that $Z_{E,2}(\pi,s)$ equals
    \[
    \frac{1}{2} \frac{L(2s-1,\trep)}{L(1,\trep) \, L(2,\trep)^3} \, q^{-2s+1} \, \left( L(1,\eta)+ \sum_{l=1}^\inf \beta_E(l,l) \, q^{(-2s+2)l}  \right) \qquad \text{if $d_E\in\fo^\times$},
    \]
    \[
    \frac{1}{2} \frac{L(2s-1,\trep)}{L(1,\trep) \, L(2,\trep)^3}\, q^{-3s+2}  \sum_{l=0}^\inf \beta_E(l,l) \, q^{(-2s+2)l} \qquad \text{if $d_E\in\varpi\fo^\times$}.
    \]

\subsubsection{Contribution of $x_3\cdot\rho(\sG)$}
Since any element in $x_3\cdot \rho(\sG)+\varpi V(\fo)$ is reduced to the form
    \[
    X_3\coloneqq \left( 
        \begin{pmatrix} 
        1&0 \\ 
        x&z 
        \end{pmatrix}, 
        \begin{pmatrix} 
        0&1 \\ 
        y&-x 
        \end{pmatrix}
    \right) \quad (x,y,z\in\varpi\fo)
    \]
by the action of $\sG$, the integral $Z_{E,3}(\pi,s)$ becomes
    \[
    \#(x_3 \cdot \rho(\sG)) \, q^{-5} \, \int_{(\varpi\fo)^{\oplus 3},\, P(X_3)\in d_E(F^\times)^2}  \beta_E(l_1,l_2) \, |P(X_3)|^{s-2} \d x \, \d y \, \d z .
    \]
In addition, we have $\cF_1(X_3)= (1,0,- x^2- y z)$ and $\cF_2(X_3)=  (y,-2x,-z)$. 
By \eqref{eq:sym2} and the fact that $\rho(1,g,{}^t\!g^{-1})$ acts on $Y=\left(
    \begin{smallmatrix} 
    -y&x \\ 
    x&z 
    \end{smallmatrix}\right)$ as $Y\cdot \rho(1,g,{}^t\!g^{-1})= \det(g)^{-1}\, {}^t\!gYg$, one finds that $Z_{E,3}(\pi,s)$ equals
    \begin{multline*}
    \frac{1}{2} \frac{1}{L(1,\trep) \, L(2,\trep)^3} \, q^{-2s+1}  \Big( L(1,\eta)\sum_{l_1=0}^\inf \beta_E(l_1+1,0) q^{(-2s+1)l_1} \\
    + \sum_{l_1=0}^\inf \sum_{l_2=1}^\inf \beta_E(l_1+l_2+1,l_2) \, q^{(-2s+1)l_1} q^{(-2s+2)l_2}  \Big) \qquad \text{if $d_E\in\fo^\times$},
    \end{multline*}
    \[
    \frac{1}{2} \frac{1}{L(1,\trep) \, L(2,\trep)^3}\, q^{-3s+2} \sum_{l_1=0}^\inf \sum_{l_2=0}^\inf \beta_E(l_1+l_2+1,l_2) \, q^{(-2s+1)l_1} q^{(-2s+2)l_2}   \;\; \text{if $d_E\in\varpi\fo^\times$}.
    \]

\subsubsection{Contribution of $x_4\cdot \rho(\sG)$}

Since any element in $x_4\cdot \rho(\sG)+\varpi V(\fo)$ is reduced to the form
    \[
    X_4\coloneqq \left( 
        \begin{pmatrix} 
        1&-x \\ 
        0&y 
        \end{pmatrix}, 
        \begin{pmatrix} 
        0&z \\ 
        1&x 
        \end{pmatrix} 
    \right) \quad (x,y,z\in\varpi\fo)
    \]
by the action of $\sG$, the integral $Z_{E,4}(\pi,s)$ becomes
    \[
    \#(x_4 \cdot \rho(\sG)) \times q^{-5} \times \int_{(\varpi\fo)^{\oplus 3},\, P(X_4)\in d_E(F^\times)^2}  \beta_E(l_1,l_2) \, |P(X_4)|^{s-2} \d x \, \d y \, \d z .
    \]
In addition, we have $\cF_1(X_4)= (z,-2x,-y)$ and $\cF_2(X_4)= (1,0,-x^2- y z)$. 
By \eqref{eq:sym2} and the fact that $\rho(g,1,g)$ acts on $Y=\left( 
    \begin{smallmatrix} 
    -y&x \\ 
    x&z 
    \end{smallmatrix} 
\right)$ as $Y\cdot \rho(g,1,g) = \det(g)^{-1} \, {}^t\!gYg$, one finds that $Z_{E,4}(\pi,s)$ equals
    \begin{multline*}
    \frac{1}{2} \frac{1}{L(1,\trep) \, L(2,\trep)^3} \, q^{-2s+1}  \Big( L(1,\eta)\sum_{l_1=0}^\inf \beta_E(0,l_1+1) q^{(-2s+1)l_1} \\
    + \sum_{l_1=0}^\inf \sum_{l_2=1}^\inf \beta_E(l_2,l_1+l_2+1) \, q^{(-2s+1)l_1} q^{(-2s+2)l_2}  \Big) \qquad \text{if $d_E\in\fo^\times$},
    \end{multline*}
    \[
    \frac{1}{2} \frac{1}{L(1,\trep) \, L(2,\trep)^3}\, q^{-3s+2} \sum_{l_1=0}^\inf \sum_{l_2=0}^\inf \beta_E(l_2,l_1+l_2+1) \, q^{(-2s+1)l_1} q^{(-2s+2)l_2}   \;\; \text{if $d_E\in\varpi\fo^\times$}.
    \]

\subsubsection{Contribution of $x_5\cdot \rho(\sG)$}

Since any element in $x_5\cdot \rho(\sG)+\varpi V(\fo)$ is reduced to
    \[
    X_5\coloneqq \left( 
        \begin{pmatrix} 
        1&0 \\ 
        0&1 
        \end{pmatrix},
        \begin{pmatrix} 
        0&1 \\ 
        y&0 
        \end{pmatrix} 
    \right) \quad (y\in\varpi\fo)
    \]
by the action of $\sG$, the integral $Z_{E,5}(\pi,s)$ becomes
    \[
    \#(x_5 \, \rho(\sG)) \times q^{-7} \times \int_{\varpi\fo,\, P(X_5)\in d_E(F^\times)^2}  \beta_{\pi,\delta}(l_1,l_2) \, |P(X_5)|^{s-2} \d y .
    \]
In addition, we have $\cF_1(X_5)=  (1,0,- y)$ and $\cF_2(X_5)= ( y,0,-1)$. 
Therefore, $Z_{E,5}(\pi,s)$ equals
    \[
    \frac{1}{2} \frac{1}{L(1,\trep) \, L(2,\trep)^3}\sum_{l=1}^\inf \beta_E(l,l) \, q^{(-2s+2)l} \qquad \text{if $d_E\in\fo^\times$},
    \]
    \[
    \frac{1}{2} \frac{1}{L(1,\trep) \, L(2,\trep)^3}\, q^{-s+1}  \sum_{l=0}^\inf \beta_E(l,l) \, q^{(-2s+2)l} \qquad \text{if $d_E\in\varpi\fo^\times$}.
    \]

\subsubsection{Contribution $x_6\cdot \rho(\sG)$}

By Lemma \ref{lem:card}, one finds
    \[
    Z_{E,6}(\pi,s) = \frac{1}{2}\frac{L(1,\eta)^2}{L(1,\trep) \, L(2,\trep)^3} \quad \text{if $d_E\in (\fo^\times)^2$,} \quad \text{and $=0$ otherwise}.
    \]

\subsubsection{Contribution $x_7\cdot \rho(\sG)$}

By Lemma \ref{lem:card}, one gets
    \[
    Z_{E,7}(\pi,s)= \frac{1}{2}\frac{L(1,\eta)^2}{L(1,\trep) \, L(2,\trep)^3} \quad \text{if $d_E\in \fo^\times \setminus (\fo^\times)^2$,} \quad \text{and $=0$ otherwise}.
    \]

\subsubsection{Summand except $Z_{E,1}(\pi,s)$}

Combining Propositions \ref{prop:410x} and \ref{prop:412x} and the explicit forms of $Z_{E,j}(\pi,s,\delta)$ $(2\leq j\leq 7)$ described above, we obtain
    \begin{equation}
    \sum_{j=2}^7 Z_{E,j}(\pi,s) = \frac{1}{2} \frac{L(1,\eta)^2 \, L(2s-1,\trep)\,L(2s-1,\pi,\mathrm{Ad}) }{L(1,\trep) \, L(2,\trep)^3 \, N(\ff_\delta)^{s-1}} \times \mathscr{T}_{E,0}(\pi,s),
    \end{equation}
where
    \begin{align*}
    \mathscr{T}_{E,0}(\pi,s) \coloneqq
    & 1 + q^{-2s+1} - q^{-2s} + 3 q^{-4s+1} - q^{-4s} + 3 q^{-6s+2} \\ 
    &+ q^{-6s+1} - q^{-8s+3}  + 2 q^{-8s+2} - 2 q^{-2s} \lambda - 2 q^{-4s+1} \lambda  \\
    &+  2 q^{-4s} \lambda - 6 q^{-6s+1} \lambda + q^{-4s} \lambda^2 + q^{-6s+1} \lambda^2  \qquad \text{if $d_E\in(\fo^\times)^2$}, 
    \end{align*}
    \begin{align*}
    \mathscr{T}_{E,0}(\pi,s) \coloneqq 
    & 1+q^{-2s+1}+q^{-2s}+q^{-4s+1}-q^{-4s} +q^{-6s+2} \\
    & +q^{-6s+1} + q^{-8s+3} +2q^{-8s+2} + 2q^{-2s}\lambda  - 2q^{-4s+1}\lambda \\ 
    & - 2q^{-4s}\lambda + 2q^{-6s+1}\lambda  - q^{-4s}\lambda^2 - q^{-6s+1}\lambda^2  \qquad \text{if $d_E\in\fo^\times\setminus(\fo^\times)^2$},    
    \end{align*}
    \[
    \mathscr{T}_{E,0}(\pi,s) \coloneqq  1+q^{-2s+1}-q^{-2s}+q^{-4s+1}-2q^{-4s+1}\lambda+2q^{-6s+2}  \qquad \text{if $d_E \in \varpi\fo^\times$}.
    \]


\subsection{Calculations for the contribution of \texorpdfstring{$x_1\cdot \rho(\sG)$}{TEXT}}\label{sec:Zx1}

In this section, we prove the following explicit formula for $Z_{E,1}(\pi,s)$:
    \begin{equation}\label{eq:Zx1}
    Z_{E,1}(\pi,s)=\frac{L(1,\eta)^2 \, L(2s-1,\trep)\,L(2s-1,\pi,\mathrm{Ad}) }
    { 2\, L(1,\trep) \, L(2,\trep)^3\, N(\ff_\delta)^{s-1}} \, \mathscr{T}_{E,1}(\pi,s).
    \end{equation}
Here, $\mathscr{T}_{E,1}(\pi,s)$ is defined by
    \begin{align*}
    \mathscr{T}_{E,1}(\pi,s)
    & = 2q^{-2s} - 2 q^{-4s+1} - 3 q^{-6s+2} - 2 q^{-6s+1} - q^{-6s} + q^{-8s+3} - 2 q^{-8s+2} \\
    & \quad - q^{-8s+1} + 2 q^{-4s+1} \lambda - 2 q^{-4s} \lambda + 6 q^{-6s+1} \lambda \\
    & \quad + 2 q^{-6s} \lambda - q^{-4s} \lambda^2 - q^{-6s+1} \lambda^2 
    \end{align*}
if $d_E \in  (\fo^\times)^2$,
    \begin{align*}
    \mathscr{T}_{E,1}(\pi,s)
    & = -q^{-6s+2} - 2 q^{-6s+1} - q^{-6s} - q^{-8s+3} - 2 q^{-8s+2} - q^{-8s+1} \\
    & \quad + 2 q^{-4s+1}\lambda +  2 q^{-4s} \lambda - 2 q^{-6s+1} \lambda - 2 q^{-6s} \lambda + q^{-4s} \lambda^2 +  q^{-6s+1} \lambda^2 
    \end{align*}
if $d_E \in\fo^\times\setminus (\fo^\times)^2$, and
\[
    \mathscr{T}_{E,1}(\pi,s) = q^{-2s}-q^{-4s+1}-q^{-4s}+2q^{-4s+1}\lambda-2q^{-6s+2}-q^{-6s+1} 
    \]
if $d_E \in\varpi\fo^\times$.

To prove \eqref{eq:Zx1}, we begin with reduction of elements as we did for other terms.
Since any element in $x_1\cdot\rho(\sG)+\varpi V(\fo)$ can be reduced to the form
    \[
    X_1\coloneqq  \left( 
        \begin{pmatrix} 
        1&0 \\ 
        0&x 
        \end{pmatrix}, 
        \begin{pmatrix} 
        0&y \\ 
        z&w 
        \end{pmatrix} 
    \right)  \quad (x,y,z,w\in\varpi\fo)
    \]
by the action of $\sG$, the integral $Z_{E,1}(\pi,s)$ equals
    \[
    \#(x_1 \cdot \rho(\sG)) \, q^{-4} \, \int_{(\varpi\fo)^{\oplus 4}, \, P(X_1)\in d_E(F^\times)^2}  \beta_E(l_1,l_2) \, |P(X_1)|^{s-2} \d x \, \d y \, \d z \, \d w.
    \]
From the above equality, we deduce
    \begin{equation}\label{eq:04091}
    Z_{E,1}(\pi,s)=(1+q^{-1})(1-q^{-2})^2\, q^{-2s} \sum_{k=0}^\inf q^{-k} \tZ_E(\pi,s,k).
    \end{equation}
Here, $\tZ_E(\pi,s,k)$ is given by
    \[
    \tZ_E(\pi,s,k) = \int \beta_E(l_1,l_2) \, |x_{12}^2+\varpi^{k+1}x_1x_2|^{s-2} \, \d x_1\, \d x_{12}\, \d x_2, 
    \]
where the integral is over the set $\{(x_1, x_2, x_3)\in\fo^{\oplus 3} \mid x_{12}^2+\varpi^{k+1}x_1x_2\in d_E(F^\times)^2\}$.
It follows from Lemma \ref{lem:param} that $(l_1, l_2)$ is determined by
    \[
    q^{-2l_1}|d_E|=\frac{|x_{12}^2+\varpi^{k+1}x_1x_2|}{\max(|x_1|,|x_{12}|, q^{-k-1}|x_2|)^2},\;\; q^{-2l_2}|d_E|=\frac{|x_{12}^2+\varpi^{k+1}x_1x_2|}{\max(|x_2|,|x_{12}|,q^{-k-1}|x_1|)^2}.
    \]
For $k,j\in\Z_{\geq 0}$, we set
    \[
    \Xi_E(\pi,s,k,j) = \int_{\fo\oplus\fo}\beta_E(l_1,l_2) \, |x_{12}^2+\varpi^k x_1|^{s-2} \d x_{12} \, \d x_1,
    \]
where $x_{12}^2+\varpi^k x_1\in d_E(F^\times)^2$ and
$(l_1, l_2)$ is determined by
    \[
    q^{-2l_1}|d_E|=\frac{ |x_{12}^2+\varpi^k x_1|}{\max(|x_1|,|x_{12}|, q^{-k})^2},\quad q^{-2l_2}|d_E|=q^{-2-2j}|x_{12}^2+\varpi^{k}x_1|.
    \]

We divide the domain of integration into the following five regions: 
\begin{itemize}
\item[(i)] $x_{12}\in\fo^\times$, $x_1\in\fo$ and $x_2\in\fo$, 
\item[(ii)] $x_{12}\in\varpi\fo$, $x_1\in\fo^\times$ and $x_2\in\fo^\times$, 
\item[(iii)] $x_{12}\in\varpi\fo$, $x_1\in\fo^\times$ and $x_2\in\varpi\fo$, 
\item[(iv)] $x_{12}\in\varpi\fo$, $x_1\in\varpi\fo$ and $x_2\in\fo^\times$, 
\item[(v)] $x_{12}\in\varpi\fo$, $x_1\in\varpi\fo$ and $x_2\in\varpi\fo$.
\end{itemize}
For $k\in\Z_{\geq -1}$ and $j,m\in\Z_{\geq 0}$ set
    \[
    \Omega_E(\pi,s,k,j,m) =  \int_{\fo^\times}\d x_2 \int_{\fo}\d x_{12} \, \beta_E(l,l+m) \, \, |x_{12}^2+\varpi^k x_2|^{s-2},
    \]
where  $x_{12}^2+\varpi^k x_2\in d_E (F^\times)^2$ and $l$ is determined by $q^{-2l}|d_E|=q^{-2-2j}|x_{12}^2+\varpi^{k}x_2|$.
Set
    \[
    \Delta_{E,1} =
        \begin{cases} 
        1  & \text{if  $d_E\in (\fo^\times)^2$,} \\ 
        0  & \text{otherwise,} 
        \end{cases} \quad 
    \Delta_{E,2} =
        \begin{cases} 
        1  & \text{if  $d_E \in \fo^\times\setminus(\fo^\times)^2$,} \\ 
        0  & \text{otherwise,} 
        \end{cases} 
    \]
    \[
    \Delta_{E,3} = 
        \begin{cases} 
        1  & \text{if  $d_E \in \varpi\fo^\times$,} \\ 
        0  & \text{otherwise.} 
        \end{cases}
    \]
Then we obtain
    \begin{equation}\label{eq:04092}
    \frac{1-q^{-2s+1}}{1-q^{-1}}\, \tZ_E(\pi,s,k) = \Delta_{E,1} +  q^{-2s+3} \Omega_E(\pi,s,k-1,0,0)+2q^{-2s+2}\Xi_E(\pi,s,k,0).
    \end{equation}
\begin{lem}\label{lem:04091}
We have
    \[
    \Omega_E(\pi,s,-1,j,m)=\frac{1}{2}q^{s-2}(1-q^{-1})\, \beta_E(j,j+m) \, \Delta_{E,3}
    \]
and
    \begin{multline*}
    \Omega_E(\pi,s,0,j,m)=\\
    \frac{1}{2}(1-q^{-1})\, \left\{ 1-(1+\eta(\varpi))q^{-1} \right\}\, \beta_E(1+j,1+j+m) \{ \Delta_{E,1}+\Delta_{E,2}\} \\
    +(1-q^{-1})\, q^{-s+1} \int_{\varpi x_2\in \varpi \fo\cap d_E(F^\times)^2} \beta_E(l,l+m)\,  |x_2|^{s-2} \, \d x_2.
    \end{multline*}
Here, $l$ is determined by $q^{-2l}|\delta|=q^{-3-2j}|x_2|$.
For $k\geq 1$,  we have
    \[
    \Omega_E(\pi,s,k,j,m)=(1-q^{-1})^2 \beta_E(1+j,1+j+m)\,  \Delta_{E,1} + q^{-2s+3}\Omega_E(\pi,s,k-2,j+1,m).
    \]
\end{lem}

\begin{lem}\label{lem:04092}
For $k\geq 1$ we have
    \begin{multline*}
    \Xi_E(\pi,s,k,j)=(1-q^{-1}) \beta_E(0,1+j)\,  \Delta_{E,1} \\
    + q^{-2s+3}\Omega_E(\pi,s,k-2,0,j+1) + q^{-2s+2}\Xi_E(\pi,s,k-1,j+1).
    \end{multline*}
In addition,
    \[
    \Xi_E(\pi,s,0,j)=\int_{\fo\cap d_E(F^\times)^2} \beta_E(l,l+j+1)\, |x|^{s-2} \, \d x
    \]
where $l$ is determined by $q^{-2l}|d_E|=|x|$. 
\end{lem}

By \eqref{eq:04091}, \eqref{eq:04092} and Lemmas \ref{lem:04091} and \ref{lem:04092}, we see that the proof of \eqref{eq:Zx1} reduces to the following case by case computation.

\subsubsection{The case $d_E\in\varpi\fo^\times$ ($E$ is ramified over $F$)}
In this case, \eqref{eq:04092} becomes
    \[
    \frac{1-q^{-2s+1}}{1-q^{-1}}\, \tZ_E(\pi,s,k) = q^{-2s+3} \Omega_E(\pi,s,k-1,0,0)+2q^{-2s+2}\Xi_E(\pi,s,k,0)  .
    \]
By Lemmas \ref{lem:04091} and \ref{lem:04092}, we have
    \begin{equation}\label{eq:must}
    \frac{1-q^{-2s+1}}{1-q^{-1}}\, \tZ_E(\pi,s,k) =\mathcal{A}_1(k) + \mathcal{A}_2(k),
    \end{equation}
where
    \[
    \mathcal{A}_1(k)\coloneqq q^{-2s+3} \Omega_E(\pi,s,k-1,0,0)  + 2q^{-2s+3}\sum_{u=1}^k \Omega_E(\pi,s,k-u-1,0,u) \, q^{u(-2s+2)}
    \]
and
    \[
    \mathcal{A}_2(k)\coloneqq  q^{-3s+3}(1-q^{-1})\,  q^{k(-2s+2)}  \sum_{u=0}^\inf q^{u(-2s+2)}\beta_E(u,u+k+1).
    \]
It is easy to see that
    \begin{equation}\label{eq:a21113}
    q^{-2s} \sum_{k=0}^\inf q^{-k} \mathcal{A}_2(k)=   \frac{L(2s-1,\pi,\mathrm{Ad})}{ L(1,\trep) \, N(\ff_\eta)^{s-1}} \{ -q^{-4s+1}+q^{-4s+1}\lambda -q^{-6s+2}  \}
    \end{equation}
by Proposition \ref{prop:412x}.
Hence, the remaining task in this case is to compute the sum $q^{-2s} \sum_{k=0}^\inf q^{-k} \mathcal{A}_1(k)$.

It follows from Lemma \ref{lem:04091} that
    \[
    \Omega_E(\pi,s,2k-1,0,m)=\frac{1}{2}q^{(-2s+3)k}q^{s-2}(1-q^{-1})\, \beta_E(2k,2k+m),
    \]
    \begin{multline*}
    \Omega_E(\pi,s,2k,0,m)= \\
    \frac{1}{2}q^{(-2s+3)k}q^{-s+1}(1-q^{-1})^2 \sum_{u=0}^\inf q^{-2u(s-1)} \beta_E(2k+u+1,2k+u+1+m).    
    \end{multline*}
Hence,
    \begin{multline*}
    \mathcal{A}_1(2k)=\frac{1}{2}q^{-s+1}(1-q^{-1})\, q^{k(-2s+3)}  \beta_E(k,k) \\
    + (1-q^{-1})\,q^{-s+1}\sum_{u=0}^{k-1} \beta_E(u , -u+2k) \, q^{u(2s-1)}q^{2k(-2s+2)} \\
    + (1-q^{-1})^2\,q^{-s+2} \sum_{u=0}^{k-1} \sum_{m=0}^\inf  q^{u(2s-1)}q^{m(-2s+2)} q^{2k(-2s+2)} \beta_E(u+m+1,-u+2k+m)
    \end{multline*}
and
    \begin{multline*}
    \mathcal{A}_1(2k+1)=\frac{1}{2}q^{(-2s+3)k}q^{-s+2}(1-q^{-1})^2 \sum_{u=1}^\inf q^{u(-2s+2)} \beta_E(k+u,k+u) \\
    + q^{-s+2}(1-q^{-1})^2\sum_{u=0}^{k-1} \sum_{m=0}^\inf q^{(2s-1)u} q^{m(-2s+2)} q^{(2k+1)(-2s+2)} \\
    \hspace{5cm} \times \beta_E(u+m+1,-u+2k+m+1) \\
    + q^{-s+1}(1-q^{-1})\sum_{u=0}^{k}  q^{u(2s-1)} q^{(2k+1)(-2s+2)}\, \beta_E(u,-u+2k+1) .
    \end{multline*}
From this, we obtain
    \begin{equation}\label{eq:20191115a}
    q^{-2s} \sum_{k=0}^\inf q^{-k} \mathcal{A}_1(k)= (1-q^{-1}) \, q^{-3s+1}( I_1 + I_2 + I_3 + I_4 ),
    \end{equation}
where
    \begin{equation*}\label{eq:a3l1}
    I_1 \coloneq
    \frac{1}{2} \sum_{k=0}^\inf q^{k(-2s+1)} \beta_E(k,k) = \frac{1}{2}B_0(q^{-2s+1}) ,
    \end{equation*}
    \begin{equation*}\label{eq:a3l3}
    I_2 \coloneq
    \sum_{l_1=1}^\inf  \sum_{l_2=0}^\inf q^{l_2(-2s+1)}q^{l_1(-2s+1)} \beta_E(l_2 , l_1+l_2)=q^{-2s+1}U_0( q^{-2s+1},q^{-2s+1}),
    \end{equation*}
    \begin{equation*}\label{eq:a3l2}
    I_3 \coloneq
    \frac{1}{2}(1-q^{-1}) \sum_{k=0}^\inf   \sum_{u=1}^\inf q^{k(-2s+1)} q^{u(-2s+2)} \beta_E(k+u,k+u)
    \end{equation*}
and
    \begin{equation*}\label{eq:a3l4}
    I_4 \coloneq
    (q-1)\sum_{l_1=2}^\inf \sum_{l_2=0}^\inf \sum_{m=0}^\inf  q^{l_2(-2s+1)} q^{l_1(-2s+1)} q^{m(-2s+2)} \beta_E(l_2+m+1,l_1+l_2+m). 
    \end{equation*}
The first three terms $I_1, I_2$ and $I_3$ can be computed by Proposition \ref{prop:410x}, Proposition \ref{prop:412x}, Lemma \ref{lem:20191230} respectively.
It follows from Lemmas \ref{lem:relation}, \ref{lem:20191230} and \ref{lem:0409l} that
    \begin{multline*}
    I_4= 
    - \frac{ L(2s-\frac{1}{2},\pi)\, (1-q^{-1})\, q^{-4s+2}} {(1+q^{-2s+1})(1+q^{-2s})}\Big\{ q^{-1}\lambda B_0(q^{-2s+1})  - q^{-1} \Big\} \\
    +  (1+q^{-2s+1})^{-1}\, L(2s-\tfrac{1}{2},\pi)\, (1-q^{-1})\, q^{-4s+3} (q^{-1}\lambda-q^{-2s}-q^{-4s+1})L(2s-1,\pi,\mathrm{Ad}) \\
    \times \Big\{  B_0(q^{-2s+1})(q^{-4s+1}+q^{2s}+q^{2s-1}+1-q^{-1}\lambda^2) -q^{2s}-q^{2s-1} +2q^{-1}\lambda-2\Big\} .
    \end{multline*}
Therefore, the right hand side of \eqref{eq:20191115a} becomes
    \[
     \frac{L(2s-1,\pi,\mathrm{Ad}) }{2\,L(1,\trep)\, N(\ff_\eta)^{s-1}} q^{-2s}(1+q^{-2s+1})(1-q^{-2s}).
    \]
Together with \eqref{eq:04091}, \eqref{eq:must} and \eqref{eq:a21113}, we obtain the equation \eqref{eq:Zx1} in the case of $d_E\in\varpi\fo^\times$.

\subsubsection{The case $d_E \in\fo^\times\setminus (\fo^\times)^2$ ($E$ is unramified over $F$)}
Note that Lemma \ref{lem:04091} implies $\Omega_E(\pi,s,k,j,m)=0$ when $k$ is odd.
By \eqref{eq:04092} and Lemma \ref{lem:04092} we have
    \[
    \frac{1-q^{-2s+1}}{1-q^{-1}}\, \tZ_E(\pi,s,2k+1) = q^{-2s+3} \Omega_E(\pi,s,2k,0,0)+2(q^{-2s+2})^2 \Xi_E(\pi,s,2k,1)  
    \]
and
    \[
    \frac{1-q^{-2s+1}}{1-q^{-1}}\, \tZ_E(\pi,s,2k)  =   2 q^{-2s+2} \Xi_E(\pi,s,2k ,0)  .
    \]
It follows from Lemma \ref{lem:04091} that
    \begin{multline*}
    \Omega_E(\pi,s,2k,0,m)= \frac{1}{2}q^{-1}(1-q^{-1}) q^{(-2s+3)k} \beta_E(k+1, k+1+m)  \\
    +\frac{1}{2}(1-q^{-1})^2 q^{(-2s+3)k} \sum_{u=0}^\inf q^{-2u(s-1)} \beta_E(k+u+1, k+u+1+m).
    \end{multline*}
Together with Lemma \ref{lem:20191230}, we deduce
    \begin{multline}\label{eq:20191230e1}
    \sum_{k=0}^\inf q^{-2k-1}q^{-2s+3} \Omega_E(\pi,s,2k,0,0)= \frac{1}{2}(1-q^{-1})q^{-2s+1}B_1(q^{-2s+1})  \\
    +\frac{1}{2}(1-q^{-1})^2 q^{-2s+2}\times L(2s-1,\pi,\mathrm{Ad}) \\
    \times \Big\{  B_0(q^{-2s+1})(q^{-4s+1}+q^{2s}+q^{2s-1}+1-q^{-1}\lambda^2) \\
    -q^{2s}-q^{2s-1}-1 +q^{-1}\lambda^2 -q^{-1}\lambda^2(1+q^{-1})^{-2} \Big\} .
    \end{multline}
We write $J_1$ for the right hand side of \eqref{eq:20191230e1}.
By Lemma \ref{lem:04092}, 
    \begin{multline}\label{eq:20191230e2}
    \sum_{k=0}^\inf q^{-2k} \cdot 2q^{-2s+2} \Xi_E(\pi,s, 2k, 0) + \sum_{k=0}^\inf q^{-2k-1} \cdot 2q^{-4s+4} \Xi_E(\pi,s, 2k, 1) \\
    =  q^{-1}(1-q^{-1}) q^{-2s+2} \sum_{k=1}^\inf \sum_{l=0}^\inf q^{(-2s+1)k} q^{(-2s+1)l} \beta_E(k,k+l+1)\\
    + (1-q^{-1})^2 q^{-2s+2} \sum_{k=1}^\inf \sum_{l=0}^\inf \sum_{u=0}^\inf q^{(-2s+1)k} q^{(-2s+1)l} q^{(-2s+2)u}\beta_E(k+u,k+l+u+1) \\
    + (1-q^{-1}) q^{-2s+2} \sum_{k=0}^\inf \sum_{u=0}^\inf q^{(-2s+1)k}  q^{(-2s+2)u}\beta_E(u,u+k+1).
    \end{multline}
It follows from Lemmas \ref{lem:20191230} and \ref{lem:0409l} that the right hand side of \eqref{eq:20191230e2} equals
    \begin{multline}\label{eq:extra}
    J_2 \coloneq
    q^{-1}(1-q^{-1}) q^{-2s+2}U_1(q^{-2s+1},q^{-2s+1}) \\
    +(1-q^{-1})^2 (1+q^{-2s+1})^{-1} q^{-4s+3} (q^{-1}\lambda-q^{-2s}-q^{-4s+1})\, L(2s-\frac{1}{2},\pi)\, L(2s-1,\pi,\mathrm{Ad}) \\
    \times \Big\{  B_0(q^{-2s+1})(q^{-4s+1}+q^{2s}+q^{2s-1}+1-q^{-1}\lambda^2)\\ 
    \hspace{2cm} -q^{2s}-q^{2s-1}-1 +q^{-1}\lambda^2 -q^{-1}\lambda^2(1+q^{-1})^{-2} \Big\} \\
    - (1-q^{-1})^2 (1+q^{-2s+1})^{-1}L(2s-\frac{1}{2},\pi)\, q^{-4s+2} C_0(q^{-2s+1}) \\
    + (1-q^{-1}) q^{-2s+2} U_0(q^{-2s+1},q^{-2s+2}).
    \end{multline}
From \eqref{eq:20191230e1} and \eqref{eq:extra}, we obtain
    \begin{multline*}\label{eq:2020.1.1}
    \frac{1-q^{-2s+1}}{1-q^{-1}} \sum_{k=0}^\inf q^{-k} \tZ_E(\pi,s,k)= J_1 + J_2 \\
    =  \frac{1}{2} q^{-2s+1}(1-q^{-1})(1+q^{-1})^{-2}L(2s-1,\pi,\mathrm{Ad}) \times (-q^{-2s+1} - 2 q^{-2s} - q^{-2s-1} - q^{-4s+2} \\
    - 2 q^{-4s+1} - q^{-4s} + 2 \lambda +  2 q^{-1} \lambda - 2 q^{-2s} \lambda - 2 q^{-2s-1} \lambda + q^{-1} \lambda^2 +  q^{-2s} \lambda^2).
    \end{multline*}
This completes the proof of the equation \eqref{eq:Zx1} in the case of $d_E \in\fo^\times\setminus (\fo^\times)^2$.

\subsubsection{The case $d_E\in  (\fo^\times)^2$ ($E=F\oplus F$)}
In this case, \eqref{eq:04092} becomes
    \[
    \frac{1-q^{-2s+1}}{1-q^{-1}}\, \tZ_E(\pi,s,k) = 1 +  q^{-2s+3} \Omega_E(\pi,s,k-1,0,0)+2q^{-2s+2}\Xi_E(\pi,s,k,0)  .
    \]
By Lemma \ref{lem:04091} we have
    \[
    \Omega_E(\pi,s,2k-1,0,m)= (1-q^{-1})^2 \sum_{t=1}^{k}q^{(-2s+3)(t-1)} \beta_E(t, t+m) 
    \]
and
    \begin{multline*}
    \Omega_E(\pi,s,2k,0,m)=  (1-q^{-1})^2 \sum_{t=1}^{k}q^{(-2s+3)(t-1)} \beta_E(t, t+m) \\
    -\frac{1}{2}q^{-1}(1-q^{-1}) q^{(-2s+3)k} \beta_E(k+1, k+1+m)  \\
    +\frac{1}{2}(1-q^{-1})^2 q^{(-2s+3)k} \sum_{u=0}^\inf q^{-2u(s-1)} \beta_E(k+u+1, k+u+1+m).
    \end{multline*}
Since 
    \[
    \sum_{m=2}^\inf q^{-m}   \sum_{t=1}^{[m/2]}q^{(-2s+3)(t-1)} \beta_E(t, t)= (1-q^{-1})^{-1}  q^{-2}\sum_{t=0}^{\inf}q^{(-2s+1)t} \beta_E(t+1, t+1),
    \]
we obtain 
    \begin{multline}\label{eq:20200104e1}
    \sum_{m=0}^\inf q^{-m} q^{-2s+3} \Omega_E(\pi,s,m-1,0,0) = \\
    \frac{1}{2}(1-q^{-1})q^{-2s+1}\sum_{k=0}^\inf q^{(-2s+1)k} \beta_E(k+1, k+1)  \\
    +\frac{1}{2}(1-q^{-1})^2 q^{-2s+2}\sum_{k=0}^\inf \sum_{u=0}^\inf q^{(-2s+1)k}q^{(-2s+2)u}\beta_E(k+u+1,k+u+1)
    \end{multline}
By Lemma \ref{lem:20191230}, the right hand side of \eqref{eq:20200104e1} equals
    \begin{multline*}
     K_1 \coloneq \frac{1}{2}(1-q^{-1})q^{-2s+1}B_1(q^{-2s+1})  +\frac{1}{2}(1-q^{-1})^2 q^{-2s+2}\times L(2s-1,\pi,\mathrm{Ad}) \\
    \times \Big\{  B_0(q^{-2s+1})(q^{-4s+1}+q^{2s}+q^{2s-1}+1-q^{-1}\lambda^2) \\
    -q^{2s}-q^{2s-1}-1 +q^{-1}\lambda^2 -q^{-1}(\lambda-2)^2(1-q^{-1})^{-2} \Big\} .
    \end{multline*}
It follows from Lemma \ref{lem:04092} that 
    \[
    \Xi_E(\pi,s,0,j)=\frac{1}{2}(1-q^{-1})\sum_{u=0}^\inf q^{u(-2s+2)} \beta_E(u , u+j+1),
    \]
and we get for $k\geq 1$
    \begin{multline*}
    \Xi_E(\pi,s,k , j)= (1-q^{-1})\sum_{t=1}^{k} q^{(-2s+2)(t-1)} \beta_E(0,j+t) \\
    + q^{-2s+3}\sum_{l=1}^{k} q^{(-2s+2)(l-1)}\Omega_E(\pi,s,k-l-1 , 0, j+l)+  q^{(-2s+2)k} \Xi_E(\pi,s,0, j+k) .
    \end{multline*}
Therefore,
    \begin{multline}\label{eq:20200104e2}
    2q^{-2s+2}\sum_{m=0}^\inf  q^{-m}  \Xi_E(\pi,s,m,0) 
    = 2q^{-2s+1}\sum_{t=0}^\inf q^{(-2s+1)t}\beta_E(t+1) \\
    +q^{-1}(1-q^{-1}) q^{-2s+2} \sum_{k=1}^\inf \sum_{l=0}^\inf q^{(-2s+1)k} q^{(-2s+1)l} \beta_E(k,k+l+1)\\
    + (1-q^{-1})^2 q^{-2s+2} \sum_{k=1}^\inf \sum_{l=0}^\inf \sum_{u=0}^\inf q^{(-2s+1)k} q^{(-2s+1)l} q^{(-2s+2)u}\beta_E(k+u,k+l+u+1) \\
    + (1-q^{-1}) q^{-2s+2} \sum_{k=0}^\inf \sum_{u=0}^\inf q^{(-2s+1)k}  q^{(-2s+2)u}\beta_E(u,u+k+1).
    \end{multline}
By Lemmas \ref{lem:20191230} and \ref{lem:0409l}, the right hand side of \eqref{eq:20200104e2} equals
    \begin{multline*}
    K_2 \coloneq 2q^{-2s+1} A_1(q^{-2s+1}) +  q^{-1}(1-q^{-1}) q^{-2s+2}U_1(q^{-2s+1},q^{-2s+1}) \\
    +(1-q^{-1})^2 (1+q^{-2s+1})^{-1} q^{-4s+3} (q^{-1}\lambda-q^{-2s}-q^{-4s+1})\, L(2s-\frac{1}{2},\pi)\, L(2s-1,\pi,\mathrm{Ad}) \\
    \times \Big\{  B_0(q^{-2s+1})(q^{-4s+1}+q^{2s}+q^{2s-1}+1-q^{-1}\lambda^2) \\
    \hspace{2cm} -q^{2s}-q^{2s-1}-1 +q^{-1}\lambda^2 -q^{-1}(\lambda-2)^2(1-q^{-1})^{-2} \Big\} \\
    - (1-q^{-1})^2 (1+q^{-2s+1})^{-1}L(2s-\frac{1}{2},\pi)\, q^{-4s+2} C_0(q^{-2s+1}) \\
    + (1-q^{-1}) q^{-2s+2} U_0(q^{-2s+1},q^{-2s+2}).
    \end{multline*}
Thus, we obtain
    \begin{multline*}
    \frac{1-q^{-2s+1}}{1-q^{-1}} \sum_{k=0}^\inf q^{-k} \tZ_E(\pi,s,k)= (1-q^{-1})^{-1} + K_1 + K_2 \\
    =  \frac{1}{2} (1-q^{-1})^{-1}L(2s-1,\pi,\mathrm{Ad}) \times (2 - 2 q^{-2s+1} - 3 q^{-4s+2} - 2 q^{-4s+1} - q^{-4s} + q^{-6s+3}  \\ 
    - 2 q^{-6s+2} - q^{-6s+1} +  2 q^{-2s+1} \lambda - 2 q^{-2s} \lambda + 6 q^{-4s+1} \lambda +  2 q^{-4s} \lambda - q^{-2s} \lambda^2 - q^{-4s+1} \lambda^2).
    \end{multline*}
This completes the proof of the equation \eqref{eq:Zx1} in the case of $d_E\in(\fo^\times)^2$.

\end{document}